\documentclass[12pt]{article}
\usepackage{fullpage}
\usepackage{amsmath}
\usepackage{amstext}
\usepackage{amsfonts}
\usepackage{amsbsy}
\usepackage{eucal}
\usepackage{amssymb}
\usepackage{amsthm}
\usepackage{enumerate}
\usepackage{color}
\input xy
\xyoption{all}
\usepackage{graphicx}
\usepackage{tikz}
\usetikzlibrary{matrix,graphs,arrows,positioning,calc,decorations.markings,shapes.symbols}
\usepackage{tikz-cd}

\newtheorem{theorem}{Theorem}[section]
\newtheorem{proposition}{Proposition}[section]
\newtheorem{corollary}{Corollary}[section]

\theoremstyle{definition}

\newtheorem{example}[theorem]{Example}
\newtheorem{remark}[theorem]{Remark}

\def\Z{\mathbb{Z}}

\def\C{\mathbb{C}}
\def\P{\mathbb{P}}

\def\D{\mathcal{D}}

\newcommand{\ba}{\begin{array}}
\newcommand{\ea}{\end{array}}

\def\disp{\displaystyle}

\begin{document}

{\noindent\Large\bf Fiber-dependent deautonomization of integrable 2D\\ mappings and discrete Painlev\'e equations}
\medskip
\begin{flushleft}

\textbf{Adrian Stefan Carstea}\\
National Institute of Physics and Nuclear Engineering, Dept. of Theoretical Physics, Atomistilor 407, 077125, Magurele, Bucharest, Romania\\
E-mail: carstea@gmail.com\\[5pt] 

\textbf{Anton Dzhamay}\\
School of Mathematical Sciences, The University of Northern Colorado, Greeley, CO 80526, USA\\
E-mail: anton.dzhamay@unco.edu\\[5pt]

\textbf{Tomoyuki Takenawa}\\
Faculty of Marine Technology, Tokyo University of Marine Science and Technology, 2-1-6 Etchu-jima, Koto-ku, Tokyo, 135-8533, Japan\\
E-mail: takenawa@kaiyodai.ac.jp

\end{flushleft}

%
%
%
%
%
%
%

%
%
%
%

\begin{abstract}
It is well known that two-dimensional
mappings preserving a rational elliptic fibration, like the Quispel-Roberts-Thompson mappings, can be deautonomized to 
discrete Painlev\'e equations. However, the dependence of this procedure on the choice of a particular elliptic fiber 
has not been sufficiently investigated. In this paper we establish a way of performing the deautonomization for a pair of an
autonomous mapping and a fiber. 
Starting from a single autonomous mapping but varying the type of a chosen fiber, we obtain
different types of  discrete Painlev\'e equations using this deautonomization procedure.
We also introduce a technique for reconstructing a mapping from the knowledge of its induced action on the Picard group and
some additional geometric data. This technique allows us to obtain factorized expressions of discrete Painlev\'e equations, 
including the elliptic case. Further, by imposing certain restrictions on such non-autonomous mappings we obtain new and simple
elliptic difference Painlev\'e equations, including examples whose symmetry groups 
do not appear explicitly in Sakai's classification.
\end{abstract}

\section{Introduction}

The importance of Painlev\'e equations comes not only from their large potential applications in mathematics and physics, but also from the fact that they
represent the essence of any integrable system that is hidden in the structure of singularities \cite{Ablowitz}. 
This is even more interesting for discrete systems, since the non-locality specific to the lattice background provides a phenomenology of higher complexity. 
Moreover, it prevents the use of complex analysis methods such as the Painlev\'e test. 
Thus, integrability of discrete equations seems to be a rather difficult question. However, 
at the beginning of 1990s, several integrability detectors
have been proposed, such as the \emph{singularity confinement}, \emph{algebraic entropy/complexity growth}, or the 
\emph{Nevanlinna theory}, \cite{GRP,HV98,AH00}. 
Among them, the singularity confinement criterion proved to be instrumental in obtaining many examples of 
discrete Painlev\'e equations. Essentially, the singularity confinement requires that any
singularities appearing in the iteration process are confined after a finite number of iterations, and moreover, that the 
initial information is fully recovered. This is contrary to
the non-integrable case, where strange attractors absorb information. 
The singularity confinement criterion was viewed as a discrete analogue of the Painlev\'e property, 
and it was quite successfully applied for \emph{integrable deautonomization} \cite{basil-book1,
basil-book2}. 
It is well known that continuous Painlev\'e equations can be regarded as ``integrable deautonomization'' of nonlinear ordinary differential
equations for elliptic functions. In the discrete case, in complete parallel to this, 
the singularity confinement was successfully applied to construct deautonomizations of such paradigmatic
discrete integrable mappings for the elliptic functions as the Quispel-Roberts-Thompson (or the QRT) mappings \cite{QRT89}.
The QRT mapping itself passes the singularity confinement test. Allowing its coefficients to be functions of independent discrete variable 
and imposing the same singularity confining pattern allows one to obtain a non-autonomous QRT-type mappings. These mappings have continuum limits that
turn out to be exactly the well-known Painlev\'e equations. This simple and powerful
procedure gave rise to a very large collection of examples of discrete Painlev\'e equations. 
Very recently this type of deautonomization has been extended by allowing the
mapping to contain many terms that are not changing the singularity patterns 
\cite{RGWMK15}. 

On the
other hand, Sakai \cite{Sakai01} discovered a deep relationship between discrete Painlev\'e equations and algebraic geometry of surfaces. 
In this approach, one considers families of algebraic surfaces obtained by blowing up nine possibly infinitely close points
on a complex projective plane. The configurations of points vary from the most generic to the most degenerate, and each family
corresponds to a choice of such configuration that is encoded in the configuration of the irreducible components of the anti-canonical 
divisor of the surface, parameters in the family are essentially the coordinates of the singular points. The orthogonal complement
of the divisor classes of these irreducible components in the Picard lattice of the surface 
is a lattice of an affine Weyl group acting as symmetries of the family via Cremona isometries. 
Discrete Painlev\'e equations then correspond to the translational elements of the group and they 
act as isomorphisms of the family. One of the most important results of Sakai's approach is the discovery of the elliptic difference 
Painlev\'e equations whose coefficients are written by elliptic functions in the iteration steps.

In this paper we consider the problem of deautonomization starting from the geometry of a discrete integrable mapping. Any
integrable two dimensional mapping preserving an elliptic fibration (and each fiber) can be lifted to an automorphism of a 
rational elliptic surface. In addition to a generic fiber, this fibration also has singular fibers and choosing a fiber to play 
the role of the anti-canonical divisor results in different affine Weyl symmetry groups. Taking this into account, we propose a 
new deautonomization method that depends on the choice of a fiber and as a result we obtain different types of 
discrete Painlev\'e equations from a single autonomous mapping.

We also want to point out that B.~Grammaticos, A.~Ramani, and their collaborators have also produced various deautonomizations, some of which share the same autonomous form. E.~g., in  \cite{GrRa-2015}, it is shown that a multiplicative and an additive equations share the same autonomous form (29) of \cite{GrRa-2015} and that an elliptic and an additive equations share the same autonomous form (26) of \cite{GrRa-2015}, although the direction of discussion is opposite to ours. Another example is \cite{CGR09}, where an autonomous mapping is deautonomized to an elliptic difference equation, while the original mapping preserves an $A_1$ curve and so it is possible to deautonomize also to $E^{(1)}_{7}$ Painlev\'e. See \cite{AHJN16,HHNS15, JoNa17} about its geometric study.

We also introduce a factorization technique for rational mappings. This technique allows us to explicitly reconstructing mappings from the 
information of their actions on the Picard group. As an application, we obtain factorized expressions of discrete Painlev\'e equations, 
including the elliptic case. 

Non-autonomous mapping obtained by our method are not always difference system, i.e., in general, parameters 
do not depend linearly on the iteration step $n$. Imposing this constraint restricts our 
non-autonomous mappings on certain subspaces, and in this way we obtain new and simple elliptic difference Painlev\'e equations. 
Such restrictions of systems to a subspace in the space of parameters is often called a projective reduction \cite{AHJN16, KNT11}. 
In this way we obtain elliptic difference systems with symmetries of type $D_6^{(1)}$ and
$(A_1+A_1+A_1)^{(1)}$ that do not appear explicitly in Sakai's classification scheme, since it focuses on the generic cases. 
It may be controversial which symmetry is important between that of the mapping itself
\cite{Takenawa01,Takenawa03} or that of the space where the mapping is defined. 
Our approach corresponds to the latter, similar to Sakai's work.

Note that labelings of Painlev\'e equations are changed if their parameter space are restricted, and hence their symmetries become smaller. This kind of phenomena has been reported from their discovery as symmetric or asymmetric forms. E.g. the asymmetric or periodic coefficients version of q-$P_{\text{III}}$ in \cite{RGH91} is Jimbo-Sakai's q-$P_{\text{VI}}$. Also by introducing some arbitrary periodic functions, Ohta et al. deautonomized some mapping coming from Miura transformation (called  triholographic representation) to elliptic $E^{(1)}_{8}$ system \cite{ORG02}, 
where the resulting mapping should have lower symmetries if we take the value of periodic functions as constants.

This paper is organized as follows. In Section 2.1 we explain some notation and conventions used throughout this paper. 
We give the definition of deautonomization in Section
2.2 and explain the factorization technique in Section 2.3. 
In Section 3 we illustrate our general approach by choosing a simple QRT mapping that has six singular fibers, and constructing its different deautonomizations, one of which is $q$-Painlev\'e VI.
In Section 4 we investigate subspaces for non-autonomous elliptic mappings obtained in Section 3, and obtain new and simple 
elliptic difference Painlev\'e equations. In Appendix~\ref{app:rational} we explain how to compute the action 
of a general rational (and not birational) mapping on the divisor classes. 
In Appendix~\ref{app:deautoQRT} we construct the deautonomization for a generic QRT 
mapping on a generic elliptic fiber, and obtain a new and simple
expression of the elliptic difference Painlev\'e equation. 
In Appendix~\ref{app:CFS} we recall C.~F.~Schwartz's algorithm \cite{Schwartz94} that transforms 
a general curve of bi-degree $(2,2)$ into a Weierstrass normal form.
We give an example of a birational representation of the affine Weyl group of type $E_8^{(1)}$ in
Appendix~\ref{app:bir-rep}.

\section{Deautonomization}
\subsection{Background, notation, and conventions}

Throughout this paper $\varphi$ usually denotes a birational mapping of $\P^1\times \P^1$ to itself. One exception is
Section~\ref{SectFactorization}, where $\varphi$ can also be only a rational mapping. In a standard
coordinate system $(x,y)$  a mapping $\varphi$ can be written as $\varphi(x,y) = (\bar{x},\bar{y}) = (f(x,y),g(x,y))$.
Let $\mathbb{C}(\mathbb{P}^{1} \times \mathbb{P}^{1}) = \mathbb{C}(x,y)$ be the field
of rational functions on $\mathbb{P}^{1} \times \mathbb{P}^{1}$. Since we work with equations, 
it is often more convenient to consider the induced \emph{pull-back} map $\varphi^{*}$ on the function fields,
$\varphi^{*}: \mathbb{C}(\bar{x},\bar{y})\to \mathbb{C}(x,y)$, given by the coordinate
substitution. Unfortunately, in coordinates, $\varphi$ and $\varphi^{*}$ are usually written in the same way,
which can lead to confusion that we want to warn the reader about by means of the following example.
\begin{example} 
Let $\varphi_{1}(x,y) = (y, x+1)$ and $\varphi_{2}(x,y) = (xy,y)$ and let 
$(x_{0},y_{0})\in \mathbb{P}^{1} \times \mathbb{P}^{1}$. 
On the level of points 
$\varphi_{2}\circ \varphi_{1}$ acts as the usual composition of mappings, 
$\varphi_{2}\circ \varphi_{1}(x_{0},y_{0}) = \varphi_{2}(y_{0},x_{0}+1) = ((x_{0}+1)y_{0},x_{0}+1)$, but on 
the level of function fields it acts in the opposite order as substitution,
\begin{align*}
(\varphi_{2}\circ \varphi_{1})^{*}(\bar{\bar{x}},\bar{\bar{y}}) 
=& 
(\varphi_{1}^{*}\circ \varphi_{2}^{*})(\bar{\bar{x}},\bar{\bar{y}}) 
= 
\left((\bar{\bar{x}},\bar{\bar{y}})\Big|\text{\tiny
     $\begin{aligned}
	\bar{\bar{x}}&\leadsto\bar{x}\bar{y} \\[-3pt]
	\bar{\bar{y}}&\leadsto \bar{y}
      \end{aligned}$}\right)\Big|\text{\tiny
     $\begin{aligned}
	\bar{x}&\leadsto y\\[-3pt]
	\bar{y}&\leadsto x+1 \end{aligned}$
}\\ 
=&  (\bar{x}\bar{y},\bar{y})\Big|\text{\tiny
     $\begin{aligned}
	\bar{x}&\leadsto y\\[-3pt]
	\bar{y}&\leadsto x+1
     \end{aligned}$
} = ((x+1)y,x+1),
\end{align*}
whereas 
$(\varphi_{2}^{*}\circ \varphi_{1}^{*})(\bar{\bar{x}},\bar{\bar{y}}) = \varphi_{2}^{*}(\bar{y},\bar{x}+1) = (y, xy + 1)$. 
\end{example}

By $\mathcal{X} = \mathcal{X}_{\mathbf{b}}$ we denote a rational algebraic surface obtained from 
$\mathbb{P}^{1} \times  \mathbb{P}^{1}$ by a sequence of blowups; here $\mathbf{b}$ (that we 
sometimes omit, as above) denotes a set of parameters describing positions of the
blowup points. 
For a mapping $\varphi: \mathbb{P}^{1} \times \mathbb{P}^{1} \to \mathbb{P}^{1} \times \mathbb{P}^{1}$ 
we usually use the same letter $\varphi$ (or $\tilde{\varphi}$, when we work with both mappings) 
to denote its extension (i.e., the mapping that coincides with the original $\varphi$ on 
a Zariski open subset of $\P^1\times \P^1$) to a mapping 
$\varphi: \mathcal{X}\to \bar{\mathcal{X}}$ between such surfaces.
Here the bar in $\bar{\mathcal{X}} = \mathcal{X}_{\bar{\mathbf{b}}}$ both helps us to distinguish between the domain and the 
range of the mapping, and also indicates that for non-autonomous mappings parameters describing the 
blowup can evolve.
The mapping $\varphi: \mathcal{X}\to \bar{\mathcal{X}}$ induces the map 
$\varphi^{*}: \mathbb{C}(\bar{\mathcal{X}})\to \mathbb{C}(\mathcal{X})$ between the corresponding function fields, 
which then gives the maps $\varphi^{*}: \operatorname{Div}(\bar{\mathcal{X}})\to \operatorname{Div}(\mathcal{X})$
of the divisor groups and $\varphi^{*}: \operatorname{Pic}(\bar{\mathcal{X}})\to \operatorname{Pic}(\mathcal{X})$
on the Picard groups; we use the same notation for all these maps, since the meaning is always clear from the context.
We also have the usual \emph{push-forward} map $\varphi_{*}$ defined on the Weyl divisors, which gives
the map $\varphi_{*}: \operatorname{Div}({\mathcal{X}})\to \operatorname{Div}(\bar{\mathcal{X}})$; when 
$\varphi$ is birational, $\varphi_{*} = (\varphi^{-1})^{*}$. We refer the reader to Section 2 of \cite{TEGORS03} 
and references therein for a description of the pull-back and push-forward maps in our context.
Note that 
for a divisor $\bar{D}$ on $\bar{\mathcal{X}}$, it is important to distinguish between its \emph{total transform}
$\varphi^{*}(\bar{D})$ and its \emph{proper transform}, which is the Zariski closure of the set
$\{\varphi^{-1}(\bar{d})~|~ \bar{d} \text{ is a generic point in }\bar{D}\}$ in $\mathcal{X}$.

Other common notation and convention that we use are the following.
\begin{itemize}
	\item As usual, $V(f)$ denotes the zero set of a regular function $f$ on $\mathcal{X}$.
	\item We use Roman letter for divisors, and script letters for divisor classes, e.g., 
	$\mathcal{D} = [D]$.
	\item Let a surface $\mathcal{X}'$ be obtained from a surface $\mathcal{X}$
	by blowing up a point $p\in\mathcal{X}$, and let $\pi: \mathcal{X}'\to \mathcal{X}$ be the corresponding 
	blowing down map.
	Then $E = \pi^{-1}(p)$ is the exceptional divisor, and for any divisor $D$ in $\mathcal{X}$, 
	its proper transform is written as $\pi^*(D)- m E$, where $m$ is the multiplicity of $D$ at $p$. 
	Thus the class of the proper tranform of 
	$D$ is written as $[\pi^*(D)- m {E}]=\D- m \mathcal{E}$, where we omit $\pi^{*}$ on the level of classes, 
	so $\mathcal{D}$ stands both for the class of 
	$D$ in $\mathcal{X}$  and the class of its total transform $\pi^{*}(D)$ in $\mathcal{X}'$.	
	\item By $\mathcal{H}_{x}$ and $\mathcal{H}_{y}$ we denote the divisor classes of vertical (resp.~horizontal)
	lines in $\mathbb{P}^{1}\times \mathbb{P}^{1}$, $\mathcal{H}_{x} = [V(x-x_{0})]$, 
	$\mathcal{H}_{y} = [V(y-y_{0})]$, as well as their total transforms in 
	$\operatorname{Pic}(\mathcal{X})$. Then 
	$\operatorname{Pic}(\mathcal{X}) = \Z \mathcal{H}_x \oplus \Z \mathcal{H}_y \oplus \bigoplus_i  \Z \mathcal{E}_i$,
	 where 
	$\mathcal{E}_{i}$ denotes the class of the $i$-th exceptional divisor in $\operatorname{Pic}(\mathcal{X})$.
	\item We use $\mathcal{D}\bullet\mathcal{D}'$ to denote the intersection number of the divisor classes
	$\mathcal{D}$ and $\mathcal{D}'$ in $\operatorname{Pic}(\mathcal{X})$ and put $\mathcal{D}^{2} = \mathcal{D}\bullet \mathcal{D}$; 
	in particular, we have
	\begin{align*}
	\mathcal{H}_{x}\bullet \mathcal{H}_{x} &= \mathcal{H}_{y}\bullet\mathcal{H}_{y}
	= \mathcal{H}_{x}\bullet\mathcal{E}_{i} = \mathcal{H}_{y}\bullet\mathcal{E}_{i} = 0,\\
	\mathcal{H}_{x}\bullet\mathcal{H}_{y} &= 1,\qquad \mathcal{E}_{i}\bullet\mathcal{E}_{j}= -\delta_{ij}.
	\end{align*}
	\item For classes in $\operatorname{Pic}(\bar{\mathcal{X}})$, we write $\bar{\mathcal{D}}$, $\bar{\mathcal{E}}_{i}$, 
	$\mathcal{H}_{\bar{x}}=\bar{\mathcal{H}}_{x}$ and so on.
	\item We denote by $-\mathcal{K}_{\mathcal{X}}$ the anti-canonical divisor class of $\mathcal{X}$, in our case this class is 
	 $-\mathcal{K}_{\mathcal{X}}=2\mathcal{H}_x+2\mathcal{H}_y-\sum_{i}\mathcal{E}_i$.
	\item Of particular importance for us are the divisor classes $\mathcal{D}$ such that $\operatorname{dim}|\mathcal{D}| = 0$,
	since in this case there is a unique prime divisor in this class. For the blowup of $\mathbb{P}^{1} \times \mathbb{P}^{1}$ 
	at a point $p_{i}(x_{i},y_{i})$, one example of such divisor and its class is $E_{i}\in \mathcal{E}_{i}$, another is
	$H_{x} - E_{i}\in \mathcal{H}_{x} - \mathcal{E}_{i}$, where $H_{x} - E_{i}$ is a proper transform of the line
	$x=x_{i}$. As far as we are aware, there is no term for such divisor classes, so in what follows we call them
	\emph{deterministic classes} and for such classes we often revert back to the Roman letters, thus identifying the class with 
	its unique prime divisor. 
\end{itemize}

\subsection{Definition of deautonomization}

Let $\varphi:\mathcal{X}\to \mathcal{X}$ be an automorphism of a rational elliptic surface $\mathcal{X}$, where 
$\mathcal{X}$ is
obtained from $\P^1 \times \P^1$ by $8$ successive blowups. We restrict to the case when $\varphi$ is not 
periodic, the degree of the 
elliptic fiber is $(2,2)$, and $\varphi$ preserves each fiber, i.e.~$\varphi$ is a nontrivial translation in a Mordell-Weil lattice (Cf. Corollary~4.5.6 of \cite{Duistermaat}).  Thus we do not consider the case when $\varphi$ exchanges
fibers or when the degree of fibers is higher, as in \cite{HKY01,TMNT09,CGR09} (see \cite{AHJN16,JoNa17} for the geometric sturdy of the elliptic difference Painlev\'e obtained in \cite{CGR09}, see also \cite{CT12} for classification of automorphisms of a rational elliptic surface). 
A model example of such mapping $\varphi$ is the QRT mapping that we describe
in Section~\ref{sub:ReviewQRT}.

For deautonomization, we consider $\mathcal{X} = \mathcal{X}_{\mathbf{b}}$ to be an element 
of a family of rational surfaces $\mathcal{X}_{\mathcal{A}'}=\{\mathcal{X}_\mathbf{a}\}_{\mathbf{a}\in \mathcal{A}'}$, where each 
$\mathcal{X}_{\mathbf{a}}$ is obtained from $\P^1 \times \P^1$ by $8$ successive blowups and 
$\mathbf{a}\in \mathcal{A}'$ is a set of parameters describing the positions of blowups. Note that in general a surface 
$\mathcal{X}_{\mathbf{a}}$ in this family is not elliptic, and so our choice of 
parameters $\mathbf{b}$ for $\mathcal{X}_{\mathbf{b}}$ is special. 
Although on the level of abstract lattices
\begin{equation*}
\operatorname{Pic}(\mathcal{X}_{\mathbf{a}})\simeq \operatorname{Pic}(\mathcal{X}_{\mathbf{a}'})\simeq
\Z \mathcal{H}_x \oplus \Z \mathcal{H}_y \oplus \bigoplus_i  \Z \mathcal{E}_i,\qquad \mathbf{a},\mathbf{a}'\in \mathcal{A}',
\end{equation*}
this isomorphism in general is not induced by a mapping from $\mathcal{X}_{\mathbf{a}}$ 
to $\mathcal{X}_{\mathbf{a}'}$; we want to restrict to the situation when this actually happens. 
That is, we would like to extend the mapping 
$\varphi: \mathcal{X}_{\mathbf{b}}\to \mathcal{X}_{\mathbf{b}}$ to a mapping
$\varphi: \mathcal{X}_{\mathbf{a}}\to \mathcal{X}_{\mathbf{a}'}$ for some subfamily $\mathcal{A}\subset \mathcal{A}'$, 
where $\mathbf{a}, \mathbf{a}'\in \mathcal{A}$ may be different --- i.e., we allow the positions
of the blowup points to move (that's \emph{deautonomization}). 
For that, we need additional information provided by a choice 
of a fiber in the original elliptic fibration. 

Recall that for an elliptic fibration $\pi: \mathcal{X} \to \P^1$ a generic fiber is an elliptic curve and
there are singular fibers over 12 points (counted with multiplicities) on the $\mathbb{P}^{1}$ base.
The classification of such singular fibers into 22 types is due to Kodaira, \cite{Kodaira63}.
Fix a fiber $D$, not necessarily singular, of the fibration $\pi: \mathcal{X}\to \P^1$. Then
$D$ is an anti-canonical divisor of $\mathcal{X}_{\mathbf{b}}$ that can be written 
as a sum of irreducible divisors $D_{i}$ as $D=\sum_i m_i D_i$, where $m_i \in {\mathbb N}$ are multiplicities. 
We then restrict to the subfamily 
$\mathcal{X}_{\mathcal{A}}$ preserving this decomposition, i.e.,  for each $\mathbf{a}\in \mathcal{A}$,
there exists the same decomposition $D_{\mathbf{a}}=\sum_i m_i D_i'$ of an 
anti-canonical divisor,
$D_i'$ belongs to the class of $D_i$  under the identification $\operatorname{Pic}(\mathcal{X}_{\mathbf{a}})\simeq \operatorname{Pic}(\mathcal{X}_\mathbf{b})$, and $D_{\mathbf{a}}$ is isomorphic to $D$ via M\"obius transformations
acting on both factors of $\mathbb{P}^{1} \times \mathbb{P}^{1}$.
A surface $\mathcal{X}_\mathbf{a}, \mathbf{a}\in \mathcal{A}$, is a \emph{generalized Halphen surface}.
Here, a rational surface $X$ is called a generalized Halphen surface if the anti-canonical divisor $-K_X$  is decomposed into irreducible divisors $D_1, \dots, D_s$ as $D = \sum_{i
=1}^s m_iD_i\quad (m_i \geq 1)$, where
the linear equivalence class of $D$ is $-K_X$ and $D_i \bullet K_X = 0$ for all $i$,   
whereas
a rational elliptic surface is also called a \emph{Halphen surface}.
\begin{remark}

	What follows is true only for generic values of $\mathbf{a}\in \mathcal{A}$, but we are only 
	interested in the generic case. Without going into details, we assume that some special 
	values of parameters had been excluded from $\mathcal{A}$, see \cite{Sakai01} for more information.
\end{remark}

The deautonomization $\varphi: \mathcal{X}_\mathcal{A}\to \mathcal{X}_\mathcal{A}$ is defined by the following conditions:
\begin{enumerate}
\item[(i)] $\varphi$ acts on the parameter space $\mathcal{A}$ as a bijection and the restriction 
of $\varphi$ to $\mathcal{X}_\mathbf{a}$, i.e. 
$\disp \varphi_\mathbf{a}:=\varphi|_{\mathcal{X}_\mathbf{a}}: \mathcal{X}_\mathbf{a}\to \mathcal{X}_{\varphi(\mathbf{a})}$, is an isomorphism;
\item[(ii)] the action $\varphi_{\mathbf{a}}^*:\operatorname{Pic}(\mathcal{X}_{\varphi(\mathbf{a})})\to \operatorname{Pic}(\mathcal{X}_\mathbf{a})$ coincides with the original action $\varphi^*:\operatorname{Pic}(\mathcal{X})\to \operatorname{Pic}(\mathcal{X})$.
\end{enumerate}

The deautonomization is unique modulo M\"{o}bius transformations of each of the factor of $\P^1 \times \mathbb{P}^{1}$.
Indeed, assume that there exist two isomorphisms $\varphi:\mathcal{X} \to \bar{\mathcal{X}}$ and $\varphi':\mathcal{X} \to \bar{\mathcal{X}}'$. Then, from the condition (ii), we see that $\phi=\varphi'\circ \varphi^{-1}: \bar{\mathcal{X}} \to \bar{\mathcal{X}}'$ acts on the Picard lattice as the identity, and hence $\phi$ is nothing but a product of M\"{o}bius transformations. Controlling the ambiguity of  M\"{o}bius transformations is one of the
most intricate steps in constructing discrete Painlev\'e equations and, more generally, of rational mappings from the 
given action on the Picard groups. We address this issue by using some geometric ideas to write down our mappings in 
factorized form, as  explained in the next subsection. Similar formulae were also pointed out in \cite{KNY16,Viallet15}, 
here we generalize and prove them.


\subsection{Factorization formulae}\label{SectFactorization}
Although we usually write our mappings in affine coordinates, to understand and to prove the factorization formulae
it is much more convenient to work in homogeneous coordinates. We do so in this section and
explain how the resulting formulae specialize to affine charts.

Let $[X_{0}:X_{1}]$ be the standard homogeneous coordinate on the complex projective line $\mathbb{P}^{1}$. Then 
the corresponding  affine
coordinates for the charts $U_{0} = \{[X_{0}:X_{1}] \mid X_{1}\neq 0\}$ and $U_{\infty} = \{[X_{0}:X_{1}] \mid X_{0}\neq 0\}$ 
are $x=[x:1]$ and $X = [1:X]$, so $X= 1/x$, as usual. 

A mapping $\psi: \mathbb{P}^{1} \to \mathbb{P}^{1}$ in homogeneous coordinates can be written
in the form $\psi([X_{0}:X_{1}]) = [\bar{X}_{0}:\bar{X}_{1}]=[P_0(X_{0}, X_{1}):P_1(X_{0}, X_{1})]$, where $P_i(X_{0}, X_{1})$'s are coprime homogeneous polynomials of the same degree. 
A mapping $\psi$
is a \emph{M\"obius transformation}, 
$\psi \in \operatorname{Aut}(\mathbb{P}^{1}) = \operatorname{\mathbf{PGL}}_{2}(\mathbb{C})$,
if these polynomials are linear forms, 
\begin{equation*}
\psi([X_{0}:X_{1}]) = [a X_{0} + b X_{1} : c X_{0} + d X_{1}],\qquad a,b,c,d \in \C,\  ad - bc \neq 0.
\end{equation*}
Let a point $A_{i}\in \mathbb{P}^{1}$ have homogeneous coordinates $\big[X_0^{(i)}:X_1^{(i)}\big]$ and consider the 
linear form $L_i(X_0,X_1)=X_1^{(i)}X_0-X_0^{(i)}X_1$. Then $A_{i} = V(\lambda_{i} L_{i}(X_{0},X_{1}))$, 
$\lambda_{i}\in \mathbb{C}^{\times }$,
and so for
two distinct points $A_{1}$ and $A_{2}$ on $\mathbb{P}^{1}$ a unique, up to scaling, 
M\"obius transformation mapping them to 
the points $[0:1]$ and $[1:0]$ respectively is given by 
$\psi([X_{0}:X_{1}]) = [\lambda_{1} L_{1}(X_{0},X_{1}), \lambda_{2} L_{2}(X_{0},X_{1})]$. If both $A_{1}$ and $A_{2}$
are different from $[1:0]$, this mapping in the  affine charts 
$U_{0}$, $\bar{U}_{0}$ can be written as
\begin{equation*}
\psi(x) = \frac{\lambda_{1}}{\lambda_{2}} \frac{L_{1}(X_{0},X_{1})}{L_{2}(X_{0},X_{1})}  = 
\frac{\lambda_{1}}{\lambda_{2}} \frac{X_{1}^{(1)}}{X_{1}^{(2)}} \frac{(x - x_{1})}{(x - x_{2})}  
 = \lambda \frac{(x-x_{1})}{(x-x_{2})}, 
\end{equation*}
and so the mapping is written, again up to scaling, 
as a ratio of corresponding affine linear polynomials $l_{i}(x) = x - x_{i}$. However, if one of $A_{i}$ is
$[1:0]$, it is easy to see that the corresponding linear polynomial becomes $l_{i}(x) = 1$. 
This remark is important to keep in mind when writing 
affine expressions for our mappings. Finally, to fix the scaling, choose a point $A_{3}$, different from 
$A_{1}$ and $A_{2}$, so that $\psi(A_{3}) = [1:1]$. Then 
\begin{align*}
\psi([X_{0}:X_{1}]) &= \left[L_{2}\big(X_{0}^{(3)},X_{1}^{(3)}\big) L_{1}\big(X_{0},X_{1}\big):
L_{1}\big(X_{0}^{(3)},X_{1}^{(3)}\big) L_{2}\big(X_{0},X_{1}\big)\right],\\ 
\psi(x) &= \frac{(x_{3} - x_{2})(x - x_{1})}{(x_{3} - x_{1})(x-x_{2})}, 
\end{align*}
which is just the usual cross-ratio formula.

A M\"obius transformation mapping distinct points $A_{1}$ and $A_{2}$ to 
 distinct points $\bar{A}_{1}$ and $\bar{A}_{2}$ respectively is given, up to scaling, by
the composition of two M\"obius transformations, 
$\psi = \bar{\psi}_{1}^{-1}\circ \psi_{2}$, where $\psi_{1}(A_{1}) = \bar{\psi}_{1}(\bar{A}_{1}) = [0:1]$
and $\psi_{2}(A_{2}) = \bar{\psi}_{2}(\bar{A}_{2}) = [1:0]$. For us it is more convenient to 
write this map in an implicit form as
\begin{align}\label{basic_homo}
\left[\bar{L}_{1}(\bar{X}_0,\bar{X}_1) :  \bar{L}_{2}(\bar{X}_0,\bar{X}_1)\right] &= 
\left[\lambda_{1} L_{1}(X_0,X_1):\lambda_{2} L_{2}(X_0,X_1)\right],\\
\intertext{where $L_i$, $\bar{L}_{i}$ are the linear forms defined above and 
$\lambda_i \in \mathbb{C}^{\times}$. In the affine form this mapping can be written as}
\frac{\bar{l}_{1}(x)}{\bar{l}_{2}(x)} &= C \frac{l_{1}(x)}{l_{2}(x)},
\end{align}
where $l_{i}(x) = x - x_{i}$ or $l_{i}(x) =1$, as explained above, and similarly for $\bar{l}_{i}(x)$.

Given that expressions in the homogeneous coordinate are somewhat cumbersome, from now on we use $x$ and 
$x_{i}$ to denote coordinates and points on $\mathbb{P}^{1}$, allowing $x = \infty$ (that we can also 
write as $X=0$), and we assume that 
the reader can easily restore the corresponding expressions in the homogeneous coordinates when needed. 

We are now ready to describe how to write down  factorization formulae for mappings from the information
about their action on the divisors. We begin with the projective plane version. 

\begin{proposition}[$\mathbb{P}^1\times \mathbb{P}^1$ version]\label{prop_bir}
Let
$\varphi:\mathbb{P}^1\times \mathbb{P}^1 \dashrightarrow \mathbb{P}^1\times \mathbb{P}^1$ be a dominant rational mapping, 
$\varphi(x,y) = (\bar{x},\bar{y})$.
For $\bar{x}_{1}\neq \bar{x}_{2}\in \mathbb{P}^{1}$ and $\bar{y}_{1}\neq \bar{y}_{2}\in \mathbb{P}^{1}$, 
consider the divisors $\bar{D}_{1} = V(\bar{x}-\bar{x}_{1})$, $\bar{D}_{2} = V(\bar{x}-\bar{x}_{2})$, 
$\bar{D}_{3} = V(\bar{y}-\bar{y}_{1})$, and
$\bar{D}_{4} = V(\bar{y}-\bar{y}_{2})$. Let their pull-backs $\varphi^{*}(\bar{D}_{i})$ be written as a sum of 
distinct irreducible divisors as
$\varphi^{*}(\bar{D}_{i}) = m_{i,1} D_{i,1} + \cdots + m_{i,S_{i}}D_{i,S_{i}}$, where $m_{i,j}\in \mathbb{N}$
are multiplicities. Let $p_{i,j}(x,y)$ 
be an irreducible defining polynomial of $D_{i,j}$; $D_{i,j} = V(p_{i,j}(x,y))$.
Then the mapping $\varphi$ is given implicitly by
\begin{align*}
	\frac{\bar{l}_{1}(\bar{x})}{\bar{l}_{2}(\bar{x})} &= 
	\frac{\lambda_{1} \prod_{s=1}^{S_{1}} p_{1,s}(x,y)^{m_{1,s}} }{\lambda_{2} \prod_{s=1}^{S_{2}} p_{2,s}(x,y)^{m_{2,s}}}
	= 
	C_{1}\frac{\prod_{s=1}^{S_{1}} p_{1,s}(x,y)^{m_{1,s}} }{ \prod_{s=1}^{S_{2}} p_{2,s}(x,y)^{m_{2,s}}},  \\
\frac{\bar{l}_{1}(\bar{y})}{ \bar{l}_{2}(\bar{y})} &=  
\frac{\mu_{1}\prod_{s=1}^{S_{3}} p_{3,s}(x,y)^{m_{3,s}} }{\mu_{2}\prod_{s=1}^{S_{4}} p_{4,s}(x,y)^{m_{4,s}}}
=C_{2}   \frac{\prod_{s=1}^{S_{3}} p_{3,s}(x,y)^{m_{3,s}} }{\prod_{s=1}^{S_{4}} p_{4,s}(x,y)^{m_{4,s}}}, 
\end{align*}
where $C_{i}$ are some non-zero constants that can be computed if we also know the 
pull-backs of $V(\bar{x} - \bar{x}_{3})$
and $V(\bar{y} - \bar{y}_{3})$ for some other points $\bar{x}_{3}$ and $\bar{y}_{3}$.
\end{proposition}

\begin{proof}
	This result is immediate when written in the homogeneous coordinates, 
	\begin{align*}
	\left[ \bar{L}_{1}(\bar{X}_0,\bar{X}_1) :  \bar{L}_{2}(\bar{X}_0,\bar{X}_1)\right] =&
	[\lambda_{1}P_1(X_0,X_1,Y_0,Y_1): \lambda_2 P_2(X_0,X_1,Y_0,Y_1)],\\
	\left[ \bar{L}_{3}(\bar{Y}_0,\bar{Y}_1) :  \bar{L}_{4}(\bar{Y}_0,\bar{Y}_1)\right] =&
	[\lambda_{3}P_3(X_0,X_1,Y_0,Y_1): \lambda_4 P_4(X_0,X_1,Y_0,Y_1)],
	\end{align*}
	where
	\begin{align*}
	P_i([X_0:X_1],[Y_0:Y_1])=& \prod_{s=1}^{S_{i}} P_{i,s}([X_0:X_1],[Y_0:Y_1])^{m_{i,s}},
	\end{align*}
	$\lambda_i$'s are non-zero constants, and 
	$P_{i,s}([X_0:X_1],[Y_0:Y_1])$ are the homogeneous defining polynomial of $D_{i,s}$.
	The affine version then follows similar to the discussion about M\"obius transformations, except that now we need to pay
	attention to the multiplicities. 
\end{proof}

\begin{example} Consider the following simple example illustrating this idea. Let the mapping $\varphi$
	be given by
\begin{equation*}
	\varphi: (x,y)\mapsto (\bar{x},\bar{y}) = \left( \lambda \frac{y(y-a_{3})}{x(y-a_{4})}, y \right)
	= \left( \frac{\lambda_{1} X_{1}Y_{0}(Y_{0} - a_{3}Y_{1})}{\lambda_{2} X_{0}Y_{1}(Y_{0} - a_{4}Y_{1})}, 
	\frac{Y_{0}}{Y_{1}}  \right),
\end{equation*}
and let $\bar{x}_{1}=0$, $\bar{x}_{2}=\infty$, $\bar{y}_{1}=0$, and $\bar{y}_{2}=\infty$.
Then
\begin{alignat*}{2}
\varphi^{*}(V(\bar{X}_{0})) &= V(X_{1}) + V(Y_{0}) + V(Y_{0} - a_{3} Y_{1}), &\quad \varphi^{*}(V(\bar{Y}_{0})) &= V(Y_{0}),\\
\varphi^{*}(V(\bar{X}_{1})) &= V(X_{0}) + V(Y_{1}) + V(Y_{0} - a_{4} Y_{1}), &\quad  \varphi^{*}(V(\bar{Y}_{1})) &= V(Y_{1}).
\end{alignat*}
Conversely, knowing the above expressions for 
$V(\varphi^{*}(\bar{X}_{0}))$ and $V(\varphi^{*}(\bar{X}_{1}))$, we can immediately write factorized 
homogeneous polynomials for $\bar{X}_{0}$ and $\bar{X}_{1}$, uniquely up to non-zero proportionality
constant, which gives
\begin{align*}
\left[\bar{X}_{0}:\bar{X}_{1}\right] &= \left[\lambda_{1} X_{1} Y_{0} (Y_{0} - a_{3} Y_{1}) : 
\lambda_{2} X_{0} Y_{1} (Y_{0} - a_{4} Y_{1})\right],
\intertext{and from this expression the affine formula for $\bar{x}$ immediately follows. Similarly,}
\left[\bar{Y}_{0}:\bar{Y}_{1}\right] &= \left[\mu_{1} Y_{0}: \mu_{2} Y_{1}\right],\qquad \bar{y} = C_{2} y;\\
\intertext{and if we know further that}
\varphi^{*}(V(\bar{Y}_{0} - a_{3} \bar{Y}_{1})) &= V(Y_{0} - a_{3} Y_{1}),
\end{align*}
we see that $[a_{3}:1] = [\mu_{1} a_{3} : \mu_{2}]$,
and hence $[\mu_{1}:\mu_{2}]=[1:1]$. Thus, we recover the normalization constant for $\bar{y}=y$.
\end{example}

Let us now consider the blowup version of Proposition~\ref{prop_bir}. The main idea is the same, the only change 
is that now we assume that we have some knowledge of the divisor map between the surfaces $\bar{\mathcal{X}}$ and 
$\mathcal{X}$ that we want to use to recover the expression for the map $\varphi$. For that, we recover
the necessary data for the divisor map $\varphi^{*}$ by using the composition of the blowdown maps and the
divisor map between the blowup surfaces.

\begin{proposition}[blowup version]\label{prop_blow} Consider the following commutative diagrams,
	\begin{center}
	\begin{tikzpicture}[>=stealth]
		\begin{scope}
		\node (P1)  at (0,0) {$\mathbb{P}^{1} \times \mathbb{P}^{1} $}; 	
		\node (P2)  at (3,0) {${\mathbb{P}^{1} \times \mathbb{P}^{1}}$}; 	
		\node (X1)  at (0,1.5) {$\mathcal{X}$}; 	
		\node (X2)  at (3,1.5) {$\bar{\mathcal{X}}$}; 	
		\draw[->,dashed] (P1)--(P2) node [above,align=center,midway] {$\varphi$};
		\draw[->,dashed] (X1)--(X2) node [above,align=center,midway] {$\tilde{\varphi}$};
		\draw[->] (X1)--(P1) node [left,align=center,midway] {$\pi$};
		\draw[->] (X2)--(P2) node [right,align=center,midway] {$\bar{\pi}$};
		\end{scope}
		\begin{scope}[xshift=6cm]
		\node (P1)  at (0,0) {$\operatorname{Div}\left(\mathbb{P}^{1} \times \mathbb{P}^{1}\right)$}; 	
		\node (P2)  at (4,0) {$\operatorname{Div}\left({\mathbb{P}^{1} \times \mathbb{P}^{1}}\right)$}; 	
		\node (X1)  at (0,1.5) {$\operatorname{Div}\left(\mathcal{X}\right)$}; 	
		\node (X2)  at (4,1.5) {$\operatorname{Div}\left(\bar{\mathcal{X}}\right)$}; 	
		\draw[->] (P2)--(P1) node [above,align=center,midway] {$\varphi^{*}$};
		\draw[->] (X2)--(X1) node [above,align=center,midway] {$\tilde{\varphi}^{*}$};
		\draw[->] (X1)--(P1) node [left,align=center,midway] {$\pi_{*}$};
		\draw[->] (P2)--(X2) node [right,align=center,midway] {$\bar{\pi}^{*}$};
		\end{scope}		
	\end{tikzpicture},
	\end{center}
	where rational surfaces $\mathcal{X}$ and $\bar{\mathcal{X}}$ are obtained from $\mathbb{P}^{1} \times \mathbb{P}^{1}$ 
	via a sequence of blowups,  $\pi$ and $\bar{\pi}$ are the corresponding blowdown maps,  
	$\varphi$ and $\tilde{\varphi}$ are dominant rational mappings with $\tilde{\varphi}$ an extension of 
	$\varphi$, $\varphi^{*}$, $\tilde{\varphi}^{*}$, $\bar{\pi}^{*}$ are the pull-back maps on the divisors, 
	and $\pi_{*}$ is the push-forward map on the algebraic cycles. 
	For $\bar{x}_1\neq \bar{x}_2 \in \P^1$ and $\bar{y}_1\neq \bar{y}_2 \in \P^1$ consider the divisors 
	$\bar{D}_{1}  = V(\bar{x}-\bar{x}_{1})$, $\bar{D}_{2} = V(\bar{x} - \bar{x}_{2})$, 
	$\bar{D}_{3} = V(\bar{y} - \bar{y}_{1})$, and $\bar{D}_{4} = (\bar{y} - \bar{y}_{2})$.
	Let 
	\begin{equation*}
	\bar{\pi}^{*}(\bar{D}_{i}) = \sum_{j=1}^{J_{i}}n_{i,j} \bar{D}_{i,j} \qquad\text{and}\qquad 
	{\tilde{\varphi}}^{*}(\bar{D}_{i,j}) =  \sum_{k=1}^{K_{i}}m_{i,j,k} {D}'_{i,j,k}
	\end{equation*}
	be the expressions of pull-back divisors in terms of distinct irreducible ones, let 
	$D_{i,j,k} = \pi_{*}({D}'_{i,j,k})$ (recall that for exceptional divisors, $\pi_{*}(E_{i}) = 0$), and let 
	$p_{i,j,k}(x,y)$ be the irreducible defining polynomials of $D_{i,j,k}$ (and so if $D'_{i,j,k}$ is an 
	exceptional divisor, $p_{i,j,k}(x,y)=1$). Then  the mapping $\varphi$ is given implicitly by
	\begin{align}
	\frac{\bar{l}_{1}(\bar{x})}{\bar{l}_{2}(\bar{x})} &= 
	C_{1}\frac{\prod\limits_{j=1}^{J_{1}} \left(\prod\limits_{k=1}^{K_{1,j}} p_{1,j,k}(x,y)^{m_{1,j,k}}\right)^{n_{1,j}} }{ 
	\prod_{j=1}^{J_{2}} \left(\prod\limits_{k=1}^{K_{2,j}} p_{2,j,k}(x,y)^{m_{2,j,k}}\right)^{n_{2,j}}},  \label{eq:xup_gen}\\
	\frac{\bar{l}_{1}(\bar{y})}{ \bar{l}_{2}(\bar{y})} &=  C_{2}   
	\frac{\prod\limits_{j=1}^{J_{3}} \left(\prod\limits_{k=1}^{K_{3,j}} p_{3,j,k}(x,y)^{m_{3,j,k}}\right)^{n_{3,j}}  	}{\prod\limits_{j=1}^{J_{4}} 	\left(\prod\limits_{k=1}^{K_{4,j}} p_{4,j,k}(x,y)^{m_{4,j,k}}\right)^{n_{4,j}} }, \label{eq:yup_gen}
	\end{align}
	where $C_{i}$ are again some non-zero proportionality constants.
\end{proposition}

\begin{proof}

	This result follows immediately from Proposition~\ref{prop_bir} and the commutativity of the diagram,
	$\varphi^{*} = \pi_{*}\circ \tilde{\varphi}^{*} \circ \bar{\pi}^{*}$. 
	Such commutativity, however,
	is not automatic and requires that the composition $\bar{\pi} \circ \tilde{\varphi} \circ \pi^{-1}$
	is \emph{algebraically stable}. Recall that a composition $f\circ g$ of dominant rational mappings
	$f\circ g: X\to Y\to Z$, where $X$, $Y$ and $Z$ are rational surfaces, is called algebraically stable if 
	$g^* \circ f^*=(f\circ g)^*$ holds. It is known that algebraic stability is equivalent to the condition that there are no points point $p\in Y$ such that $p$ is simultaneously a critical point of
	$g$, $\dim(g)^{-1}(p)\geqslant 1$, and an indeterminate point of $f$. Since both $\pi$ and
	$\bar{\pi}$ are blowups, this condition holds. We refer to
	 Proposition~2.1 of \cite{TEGORS03} or \cite{DF01} for details.
%
%
%
%
%
\end{proof}

\begin{remark}\qquad
	
\begin{enumerate}[(i)]
\item  An example of multiplicity $n_{i,j}>1$ is when we first blowup the point $\bar{x}_{1}=(0,0)$ 
and next blowup the intersection point of the proper transform of $x=0$ and the exceptional 
divisor $E_{1}$; then 
\begin{equation*}
\bar{\pi}^{*}(V(\bar{x} - \bar{x}_{1})) = (H_{x} - E_{1} - E_{2}) + (E_{1} - E_{2}) + 2 E_{2}\in \mathcal{H}_{x}.
\end{equation*}
\item  If $\tilde{\varphi}$ is an isomorphism, $K_{i,j}=m_{i,j,k}=1$ for all $i,j,k$.
\end{enumerate}
\end{remark}

This proposition gives a way to recover the rational mapping from the action on the divisor group, but not on the Picard group.  
In order to recover the mapping from the action on the Picard group, we need to consider the case where divisor classes are
decomposed into deterministic ones.

There are many more examples of application of this technique for birational mappings in the next sections.
However, we want to emphasize that this approach is quite general. Below we illustrate it by considering
in detail an example of a general rational, but not birational, mapping.

\begin{example}\label{ex:rational}
We start from a rational mapping $\varphi: \P^1\times \P^1\to \P^1\times \P^1$ given by
\begin{equation*}
\varphi:(\bar{x},\bar{y})=\left(\frac{x(y-1)}{y^2}, \frac{x^2}{y} \right).
\end{equation*}
It is easy to see that generically this mapping is three-to-one, and so it is not birational.
Indeed, given $(\bar{x},\bar{y})$, we get $y = x^{2}(\bar{y})^{-1}$. Substituting it into 
the equation for $\bar{x}$, simplifying, and excluding  the special value $x=0$ we get the 
cubic equation $\bar{x} x^{3} - \bar{y} x^{2} + \bar{y}^{2} = 0$.

Using $\bar{x}_{1}=0$, $\bar{x}_{2} = \infty$, $\bar{y}_{1} = 0$, and $\bar{y}_{2}=\infty$, we see that
\begin{alignat*}{2}
	\varphi^{*}(V(\bar{x})) &= V(X_{0}) + V(Y_{0} - Y_{1}) + V(Y_{1}) &&=V(x) + V(y-1) + V(Y),\\
	\varphi^{*}(V(\bar{X})) &= V(X_{1}) + 2 V(Y_{0}) &&=V(X) + 2V(y),\\
	\varphi^{*}(V(\bar{y})) &= 2 V(X_{0}) + V(Y_{1}) &&= 2 V(x) + V(Y),\\
	\varphi^{*}(V(\bar{Y})) &= 2 V(X_{1}) + V(Y_{0}) &&=  2 V(X) + V(y),	 
\end{alignat*}
and from Proposition~\ref{prop_bir} we immediately recover our mapping up to 
normalization constants,
\begin{equation}\label{eq:ex-rat-recov}
\varphi:(\bar{x},\bar{y})=\left( \frac{\lambda_{1} X_{0} (Y_{0} - Y_{1}) Y_{1}}{\lambda_{2} X_{1} Y_{0}^{2}},
\frac{\mu_{1} X_{0}^{2} Y_{1}}{\mu_{2} X_{1}^{2} Y_{0}}  \right)
= \left(C_{1} \frac{x(y-1)}{y^{2}}, C_{2} \frac{x^{2}}{y} \right).
\end{equation}

Let now $\mathcal{X}_{12} = \bar{\mathcal{X}}_{12}$ be a blowup of $\mathbb{P}^{1} \times \mathbb{P}^{1}$ at the points
$A_{1}(0,0)$ and $A_{2}(\infty,\infty)$. The mapping 
$\tilde{\varphi}^{*}: \operatorname{Div}(\bar{\mathcal{X}}_{12})\to \operatorname{Div}({\mathcal{X}}_{12})$
for the divisors $\bar{E}_{i}$, $\bar{H}_{x} - \bar{E}_{i}$, and $\bar{H}_{y} - \bar{E}_{i}$, $i=1,2$,
can be computed directly from the definition of the pull-back. This computation is straightforward, but cumbersome,
and is explained in Appendix~\ref{app:rational}, here we only state the result. We get:
\begin{alignat*}{2}
	\tilde{\varphi}^{*}(\bar{E}_{1}) &= (H_{x} - E_{1}) + (H_{y} - E_{2}), & \quad 
	\tilde{\varphi}^{*}(\bar{H}_{x} - \bar{E}_{2}) &= (H_{y} - E_{1}) + E_{1},\\
	\tilde{\varphi}^{*}(\bar{E}_{2}) &= (H_{x} - E_{2}) + (H_{y} - E_{1}), &\quad 
	\tilde{\varphi}^{*}(\bar{H}_{y} - \bar{E}_{1}) &= (H_{x} - E_{1}) + E_{1},\\
	\tilde{\varphi}^{*}(\bar{H}_{x} - \bar{E}_{1}) &= \pi^{*}(V(y-1)),	 &\quad 
	\tilde{\varphi}^{*}(\bar{H}_{y} - \bar{E}_{2}) &= (H_{x} - E_{2}) + E_{2}.
\end{alignat*}

Then
\begin{align}
	\varphi^{*}(V(\bar{X}_{0})) &= 	\varphi^{*}(V(\bar{x}))
	= (\pi_{*}\circ \tilde{\varphi}^{*}\circ \bar{\pi}^{*})(V(\bar{x}))\label{eq:ex-rat-x0}\\
	&= (\pi_{*}\circ \tilde{\varphi}^{*}) \left((\bar{H}_{x} - \bar{E}_{1}) + \bar{E}_{1}\right)\notag\\
	&= \pi_{*} (\pi^{*}(V(y-1)) + (H_{x} - E_{1}) + (H_{y} - E_{2}) )\notag\\
	&= V(y-1) + \pi_{*}((H_{x} - E_{1}) + E_{1}) + \notag\\
	&\qquad \pi_{*}((H_{y} - E_{2}) + E_{2}) - \pi_{*}(E_{1} + E_{2}) \notag\\
	&= V(y-1) + V(x) + V(Y).\notag
\end{align}
Proceeding in exactly the same way, we get
\begin{align}
	\varphi^{*}(V(\bar{X}_{1})) 
	&= (\pi_{*}\circ \tilde{\varphi}^{*}) \left((\bar{H}_{x} - \bar{E}_{2}) + \bar{E}_{2}\right)\label{eq:ex-rat-x1}\\
	&= \pi_{*}((H_{y} - E_{1}) + (E_{1}) + (H_{x} - E_{2}) + (H_{y} - E_{1}) )\notag\\
	&= V(X) + 2V(y),\notag\\
	\varphi^{*}(V(\bar{Y}_{0})) 
	&= (\pi_{*}\circ \tilde{\varphi}^{*}) \left((\bar{H}_{y} - \bar{E}_{1}) + \bar{E}_{1}\right)\label{eq:ex-rat-y0}\\
	&= \pi_{*}((H_{x} - E_{1}) + (E_{1}) + (H_{x} - E_{1}) + (H_{y} - E_{2}) )\notag\\
	&= 2V(x) + V(Y),\notag\\
	\varphi^{*}(V(\bar{Y}_{1})) 
	&= (\pi_{*}\circ \tilde{\varphi}^{*}) \left((\bar{H}_{y} - \bar{E}_{2}) + \bar{E}_{2}\right)\label{eq:ex-rat-y1}\\
	&= \pi_{*}((H_{x} - E_{2}) + (E_{2}) + (H_{x} - E_{2}) + (H_{y} - E_{1}) )\notag\\
	&= 2V(X) + V(y),\notag
\end{align}
and so we get back equations~\eqref{eq:ex-rat-recov}.


Note that $\tilde{\varphi}^{*}(\bar{H}_{x} - \bar{E}_{1}) = \pi^{*}(V(y-1))\in \mathcal{H}_{y}$,
which is not a deterministic class, and thus $\pi^{*}(V(y-1))\in \operatorname{Div}(\mathcal{X}_{12})$ can not be
specified by its class. To fix this issues we need to blowup at a point on $y=1$.

Thus, let now $\mathcal{X} = \bar{\mathcal{X}}$ be a blowup of $\mathbb{P}^{1} \times \mathbb{P}^{1}$ at the points
$A_{1}(0,0)$, $A_{2}(\infty,\infty)$, and $A_{3}(0,1)$. The mapping 
$\tilde{\varphi}^{*}: \operatorname{Div}(\bar{\mathcal{X}})\to \operatorname{Div}({\mathcal{X}})$
for the divisors that we are interested in can be computed using the same approach as explained in 
Appendix~\ref{app:rational}. We get
\begin{align*}
	\tilde{\varphi}^{*}(\bar{E}_{1}) &= (H_{x} - E_{1} - E_{3}) + (H_{y} - E_{2}) + 2 E_{3}, \\
	\tilde{\varphi}^{*}(\bar{E}_{2}) &= (H_{x} - E_{2}) + (H_{y} - E_{1}), \quad 
	\tilde{\varphi}^{*}(\bar{E}_{3}) = 0,\\
	\tilde{\varphi}^{*}(\bar{H}_{x} - \bar{E}_{1} - \bar{E}_{3}) &= H_{y} - E_{3},\\
	\tilde{\varphi}^{*}(\bar{H}_{x} - \bar{E}_{2}) &= (H_{y} - E_{1}) + E_{1},\\
	\tilde{\varphi}^{*}(\bar{H}_{y} - \bar{E}_{1}) &= (H_{x} - E_{1} - E_{3}) + E_{1},\\
	\tilde{\varphi}^{*}(\bar{H}_{y} - \bar{E}_{2}) &= (H_{x} - E_{2}) + E_{2}.
\end{align*}
Then \eqref{eq:ex-rat-x0} becomes
\begin{align*}
	\varphi^{*}(V(\bar{X}_{0})) &= 	\varphi^{*}(V(\bar{x})) 
	= (\pi_{*}\circ \tilde{\varphi}^{*}\circ \bar{\pi}^{*})(V(\bar{x}))\\
	&= (\pi_{*}\circ \tilde{\varphi}^{*}) \left((\bar{H}_{x} - \bar{E}_{1} - E_{3}) + \bar{E}_{1} + E_{3}\right)\\
	&= \pi_{*} ((H_{y} - E_{3}) + ((H_{x} - E_{1} - E_{3}) + (H_{y} - E_{2}) + 2 E_{3}) + 0 )\\
	&= V(y-1) + V(x) + V(Y),
\end{align*}
and changes in \eqref{eq:ex-rat-x1}--\eqref{eq:ex-rat-y1} are minimal and are omitted.  

\end{example}

Finally, we remark that in Proposition~\ref{prop_blow}, equation \eqref{eq:xup_gen} has the form
\begin{equation*}
\frac{p(\bar{\mathcal{H}}_x-\bar{\mathcal{E}}_1)}{p(\bar{\mathcal{H}}_x-\bar{\mathcal{E}}_2)}=  C \  \frac{\disp \prod_{j=1}^{J_1}\left(\prod_{k=1}^{K_{1,j}} p(D_{1,j,k})^{m_{1,j,k}}\right)^{n_{1,j}} }
{\disp \prod_{j=1}^{J_2}\left(\prod_{k=1}^{K_{2,j}} p(D_{2,j,k})^{m_{2,j,k}}\right)^{n_{2,j}}}.
\end{equation*}
This formula has the following useful generalization, that can be proved along the same lines as 
Propositions~\ref{prop_bir}~and~\ref{prop_blow}.

\begin{proposition}\label{prop_d}
Let $\tilde{\varphi}:\mathcal{X} \to \bar{\mathcal{X}}$ be a dominant rational mapping and let
$\bar{D_1}$ and $\bar{D_2}$ be two linearly equivalent divisors. Assume that 
\begin{equation*}
\bar{D}_i=\sum_{j=1}^{J_{i}}n_{i,j}\bar{D}_{i,j},\quad (i=1,2),\qquad 
\tilde{\varphi}^*(\bar{D}_{i,j}) = \sum_{k}^{K_{i,j}} m_{i,j,k}D_{i,j,k}
\end{equation*}
are decompositions in terms of distinct irreducible divisors, and that 
$\bar{p}_{i,j}(\bar{x},\bar{y})$ and  ${p}_{i,j,k}(x,y)$ are irreducible
defining polynomials for $\bar{\pi}_{*}(\bar{D}_{i,j})$ and $\pi_{*}(D_{i,j,k})$ 
respectively, as in Proposition~\ref{prop_blow}.
Then
\begin{equation}
\frac{\prod\limits_{j=1}^{J_{1}}\bar{p}_{1,j}(\bar{x},\bar{y})}{\prod\limits_{j=1}^{J_{2}}\bar{p}_{2,j}(\bar{x},\bar{y})} = 
C\frac{\prod\limits_{j=1}^{J_{1}} \left(\prod\limits_{k=1}^{K_{1,j}} p_{1,j,k}(x,y)^{m_{1,j,k}}\right)^{n_{1,j}} }{ 
\prod_{j=1}^{J_{2}} \left(\prod\limits_{k=1}^{K_{2,j}} p_{2,j,k}(x,y)^{m_{2,j,k}}\right)^{n_{2,j}}},  
\end{equation}
where $C \in \C$ is again some non-zero proportionality constant.
\end{proposition}

%

\section{Deautonomization of a QRT mapping}

\subsection{Brief Review of the QRT mapping} 
\label{sub:ReviewQRT}


A paradigmatic example of an automorphism of a rational elliptic surface is the QRT, or the 
Quispel-Roberts-Thompson mapping. We now give a brief geometric description of this mapping 
following T.~Tsuda \cite{Tsuda04}, see also the recent monograph by J.~Duistermaat \cite{Duistermaat}.

We start with a \emph{bi-quadractic} curve $\Gamma$ on $\mathbb{P}^{1} \times \mathbb{P}^{1}$ which, in the
affine $\mathbb{C}^{2}$-chart, is given by a bi-degree $(2,2)$ polynomial equation that can be compactly
written with the help of a coefficient matrix $\mathbf{A}\in \operatorname{Mat}_{3\times3}(\mathbb{C})$
as
\begin{equation*}
\mathbf{x}^{T} \mathbf{A} \mathbf{y} = 
\begin{bmatrix}
	x^2 & x & 1 
\end{bmatrix}
\begin{bmatrix}
	a_{00} & a_{01} & a_{02} \\			a_{10} & a_{11} & a_{12} \\			a_{20} & a_{21} & a_{22} \\
\end{bmatrix} \begin{bmatrix}
	y^{2} \\ y \\1
\end{bmatrix} = 
\sum_{i,j=0}^{2} a_{ij} x^{2-i} y^{2-j} =0.
\end{equation*}
In general, $\Gamma$ is an \emph{elliptic curve} that can be rewritten in the 
Weierstrass normal form $Y^{2} = 4 X^{3} - g_{2} X - g_{3}$ as explained in Appendix~\ref{app:CFS}. 
Since $\Gamma$ is of bi-degree $(2,2)$, any vertical line $V(x-a)$ intersects $\Gamma$ in two points
$x$ and $\bar{x}$, and any vertical line $V(y-b)$ intersects $\Gamma$ in two points
$y$ and $\bar{y}$, these points are generically distinct. This allows us to define two involutions on $\Gamma$, 
$r_{x}:(x,y)\to (\bar{x},y)$, and $r_{y}:(x,y)\to (x,\bar{y})$, as well as their composition 
$r_{x}\circ r_{y}: (x,y)\to (\bar{x},\bar{y})$. When $\Gamma$ is elliptic, 
the mapping $r_{x}\circ r_{y}$ becomes a translation with respect to the abelian group
structure on $\Gamma$.

The QRT map is an extension of the composition $r_{x}\circ r_{y}$ 
from $\Gamma$ to $\mathbb{P}^{1}\times \mathbb{P}^{1}$. For that, 
we take \emph{two} coefficient matrices $\mathbf{A}, \mathbf{B}\in \operatorname{Mat}_{3\times3}(\mathbb{C})$ and
consider a \emph{pencil} of bi-quadratic curves 
\begin{align}\label{ell_family}
\Gamma_{[\alpha:\beta]}:\ \alpha \mathbf{x}^{T} \mathbf{A} \mathbf{y} + \beta \mathbf{x}^{T} \mathbf{B} \mathbf{y} &=0,
\end{align}
parameterized by $[\alpha:\beta]\in \mathbb{P}^{1}$. 
Then, for a generic point $(x_{*},y_{*})\in \mathbb{P}^{1}\times \mathbb{P}^{1}$, 
there is only one curve in this pencil passing through
that point and the corresponding value of the parameter is 
$[\alpha:\beta] = [\mathbf{x}_{*}^{T} \mathbf{B}\mathbf{y}_{*}:- \mathbf{x}_{*}^{T}\mathbf{A}\mathbf{y}_{*}]$.
The only exceptions are  eight (when counted with multiplicities) \emph{base points}  of the pencil that are given by 
$\mathbf{x}^{T}\mathbf{A}\mathbf{y} = \mathbf{x}^{T}\mathbf{B}\mathbf{y} = 0$.
Resolving these points using the blowup, we get a rational 
elliptic surface $\mathcal{X}$ together with the QRT automorphism $r_{x}\circ r_{y}$ preserving the elliptic fibration 
$\pi: \mathcal{X}\to \mathbb{P}^{1}$, whose fibers $\pi^{-1}([\alpha:\beta])$ are elliptic curves except for 
12 parameter values, again counted with multiplicities, that correspond to \emph{singular fibers}; such fibers are classified by K.~Kodaira 
\cite{Kodaira63} into 22 types.

In this paper we work with a mapping $\varphi$ such that $r_{x} \circ r_{y} = \varphi^{2}$, and so $\varphi$ can be thought of
as a ``half'' of the QRT mapping. To define $\varphi$, we introduce an involution $\sigma_{xy}$ acting on both 
the coordinates and the coefficient matrices as $\sigma_{xy}(x,y;\mathbf{A},\mathbf{B}) = (y,x;\mathbf{A}^{T},\mathbf{B}^{T})$
and put $\varphi = \sigma_{xy}\circ r_{y}$. It is then immediate that $r_{x} \circ r_{y} = \varphi^{2}$. 
Explicitly, substituting $[\alpha:\beta] = [\mathbf{x}^{T} \mathbf{B}\mathbf{y}:- \mathbf{x}^{T}\mathbf{A}\mathbf{y}]$
into the equation
\begin{equation*}
\alpha \mathbf{x}^{T} \mathbf{A} \bar{\mathbf{y}} + \beta \mathbf{x}^{T} \mathbf{B} \bar{\mathbf{y}}=0
\end{equation*}
and simplifying, we get the equation $(\bar{y} - y)(f_{3}(x) y \bar{y} - f_{2}(x) (y + \bar{y}) + f_{1}(x)) = 0$,
where $[f_{1}(x),f_{2}(x),f_{3}(x)] = 
(\mathbf{x}^{T} \mathbf{A})\times (\mathbf{x}^{T}\mathbf{B})$. Solving the second factor for $\bar{y}$ to get 
$r_{y}:(x,y)\mapsto (x,\bar{y})$ and then applying
$\sigma_{xy}$ we get the following expression for the mapping $\varphi: (x,y)\to (\bar{x},\bar{y})$:
%
%
%
%
\begin{align}\label{eq:generic-phi}
\left\{\ba{rcl}
\bar{x}&=&\disp \frac{f_1(x)-f_2(x)y}{f_2(x)-f_3(x)y},\\
\bar{y}&=&x.
\ea\right.\,
\end{align} 


\subsection{Singular fibers and blowups}

We consider in detail a particular example of the QRT mapping (generic case is outlined in Appendix~\ref{app:deautoQRT}) defined by the 
following symmetric matrices $\mathbf{A}$ and $\mathbf{B}$: 
\begin{align*}
\mathbf{A}=\begin{bmatrix}0&0&0\\0&1&0\\0&0&0\end{bmatrix}, \quad
\mathbf{B}=\begin{bmatrix}1&a+a^{-1}&1\\a+a^{-1}&0&-a-a^{-1}\\1&-a-a^{-1}&1\end{bmatrix},
\end{align*}
where $a\neq 0,\pm 1$ is a constant.

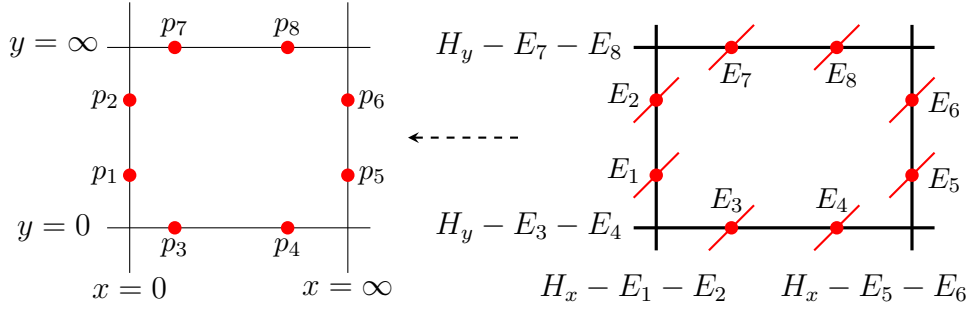
\begin{figure}[h]
	\centering
	\begin{tikzpicture}[>=stealth, 
		elt/.style={circle,draw=red!100, fill=red!100, thick, inner sep=0pt,minimum size=1.5mm}]
		\draw[black] (1,0.8) -- (4.5,0.8);  \draw[black] (1,3.2) -- (4.5,3.2);
		\draw[black] (1.3,0.2) -- (1.3,3.8); \draw[black] (4.2,0.2) -- (4.2,3.8);
		\node at (0.5,0.8) {$y = 0 \quad $}; 
		\node at (0.5,3.2) {$y = \infty \quad $}; 
		\node at (1.3,0) {$x = 0$}; 
		\node at (4.2,0) {$x = \infty$}; 
		\node[style=elt] (p1) at (1.3,1.5) {}; 	\node [left] at (p1) 	{\small $p_{1}$};
		\node[style=elt] (p2) at (1.3,2.5) {};	\node [left] at (p2)	{\small $p_{2}$};
		\node[style=elt] (p3) at (1.9,0.8) {}; 	\node [below] at (p3) 	{\small $p_{3}$};
		\node[style=elt] (p4) at (3.4,0.8) {};	\node [below] at (p4) 	{\small $p_{4}$};
		\node[style=elt] (p5) at (4.2,1.5) {}; 	\node [right] at (p5) 	{\small $p_{5}$};
		\node[style=elt] (p6) at (4.2,2.5) {};	\node [right] at (p6) 	{\small $p_{6}$};
		\node[style=elt] (p7) at (1.9,3.2) {}; 	\node [above] at (p7) 	{\small $p_{7}$};
		\node[style=elt] (p8) at (3.4,3.2) {};	\node [above] at (p8) 	{\small $p_{8}$};
		\draw[thick, black,<-,dashed] (5,2) -- (6.5,2);	
		\begin{scope}[xshift=8cm]
		\draw[black,very thick] (0,0.8) -- (4,0.8);  \draw[black,very thick] (0,3.2) -- (4,3.2);
		\draw[black,very thick] (0.3,0.5) -- (0.3,3.5); \draw[black,very thick] (3.7,0.5) -- (3.7,3.5);
		\node [left] at (0,0.8) {${H}_{y} - {E}_{3} - {E}_{4}$};
		\node [left] at (0,3.2) {${H}_{y} - {E}_{7} - {E}_{8}$};		
		\node at (0,0) {${H}_{x} - {E}_{1} - {E}_{2}$};		
		\node at (3.2,0) {${H}_{x} - {E}_{5} - {E}_{6}$};		
		\node[style=elt] (E1) at (0.3,1.5) {}; 	\draw[red,thick] (0.0,1.2)--(0.6,1.8); 
		\node [left] at (E1.north west) {\small ${E}_{1}$};
		\node[style=elt] (E2) at (0.3,2.5) {};	\draw[red,thick] (0.0,2.2)--(0.6,2.8); 
		\node [left] at (E2.north west) {\small ${E}_{2}$};	
		\node[style=elt] (E3) at (1.3,0.8) {}; \draw[red,thick] (1.0,0.5)--(1.6,1.1); 
		\node [above] at (E3.north west) {\small ${E}_{3}$};			
		\node[style=elt] (E4) at (2.7,0.8) {};	\draw[red,thick] (2.4,0.5)--(3.0,1.1); 
		\node [above] at (E4.north west) {\small${E}_{4}$};				
		\node[style=elt] (E5) at (3.7,1.5) {}; \draw[red,thick] (3.4,1.2)--(4.0,1.8); 
		\node [right] at (E5.south east) {\small ${E}_{5}$};	
		\node[style=elt] (E6) at (3.7,2.5) {};	\draw[red,thick] (3.4,2.2)--(4.0,2.8); 
		\node [right] at (E6.south east) {\small ${E}_{6}$};		
		\node[style=elt] (E7) at (1.3,3.2) {}; \draw[red,thick] (1.0,2.9)--(1.6,3.5);	
		\node [below] at (E7.south east) {\small ${E}_{7}$};				
		\node[style=elt] (E8) at (2.7,3.2) {};	\draw[red,thick] (2.4,2.9)--(3.0,3.5);	
		\node [below] at (E8.south east) {\small ${E}_{8}$};			
		\end{scope}
		\end{tikzpicture}.
	\caption{Okamoto Space of Initial Conditions for the QRT Map~\eqref{eq:phi-map}.}
	\label{fig:QRT-surface}
\end{figure}

Put $k=\alpha/\beta$, then the corresponding pencil of bi-quadratic curves 
(or the elliptic fibration) is given by 
\begin{align}\label{exy}
\Gamma_{k}: k xy+ \big(xy+a^{-1}(x+y)-1\big)\big(xy+a(x+y)-1\big)=0. 
\end{align}
The base points 
of this pencil on $\mathbb{P}^{1} \times \mathbb{P}^{1}$ are
$p_1(0,a)$, $p_2(0,a^{-1})$, $p_3(a,0)$, $p_4(a^{-1},0)$, 
$p_5(\infty, -a)$,  $p_6(\infty, -a^{-1})$,  $p_7(-a, \infty)$, and $p_8(-a^{-1}, \infty)$. Blowing them up 
we obtain a rational surface $\mathcal{X}$, as shown on Figure~\ref{fig:QRT-surface}.

\begin{figure}[h]
	\centering
	\begin{tikzpicture}[>=stealth, 
		elt/.style={circle,draw=red!100, fill=red!60, thick, inner sep=0pt,minimum size=2mm}]
		\draw[blue,very thick] (-0.5,0) -- (7.5,0) (-0.5,6) -- (7.5,6) (0,-0.5) -- (0,6.5) (7,-0.5) -- (7,6.5);
		\draw[olive,thick] (7.2,1.5) --  (1,1.5) .. controls (0.6,1.5) and (0.4,1.8) .. (0.3,2.0);
		\draw[olive,thick] (-0.2,2.5) --  (6,2.5) .. controls (6.4,2.5) and (6.6,2.3) .. (6.7,2.1);
		\draw[olive,thick] (7.2,3.5) --  (1,3.5) .. controls (0.6,3.5) and (0.4,3.8) .. (0.3,4.0);
		\draw[olive,thick] (-0.2,4.5) --  (6,4.5) .. controls (6.4,4.5) and (6.6,4.3) .. (6.7,4.1);
		\draw[red,thick] (-0.7,1.2)--(0.7,1.8) (-0.7,3.2)--(0.7,3.8) (6.4,2.2)--(7.6,2.8) (6.4,4.2)--(7.6,4.8);
		\draw[red,thick] (1.8,-0.5)--(2.2,0.5) (4.3,-0.5)--(4.7,0.5) (4.8,5.5)--(5.2,6.5) (2.8,5.5)--(3.2,6.5);
		\draw[dashed,->] (0.1,0.8)--(0.8,0.1); 	\draw[dashed,->] (6.2,0.1)--(6.9,0.8);
		\draw[dashed,->] (6.9,5.1)--(6.2,5.9);	\draw[dashed,->] (0.8,5.9)--(0.1,5.1);
		\draw[dashed,->] (0.1,3.4)--(1.9,0.1);	\draw[dashed,->] (0.1,1.4)--(4.4,0.1);
		\draw[dashed,->] (6.9,4.6)--(3.1,5.9);	\draw[dashed,->] (6.9,2.6)--(5.1,5.9);
		\draw[dashed,->] (2,0)--(2,1.5); \draw[dashed,->] (4.5,0)--(4.5,3.5);
		\draw[dashed,->] (3,6)--(3,2.5); \draw[dashed,->] (5,6)--(5,4.5);
		\node [below] at (-1,-0.5) {$H_{x} - E_{1} - E_{2}$};
		\node [below] at (7,-0.5) {$H_{x} - E_{5} - E_{6}$};
		\node [left] at (-0.5,0) {$H_{y} - E_{3} - E_{4}$};
		\node [left] at (-0.5,6) {$H_{y} - E_{7} - E_{8}$};
		\node [right] at (7.2,1.5) {$H_{y} - E_{1}$};		
		\node [right] at (7.2,3.5) {$H_{y} - E_{2}$};		
		\node [left] at (-0.2,2.5) {$H_{y} - E_{5}$};		
		\node [left] at (-0.2,4.5) {$H_{y} - E_{6}$};		
		\node[style=elt] (p1) at (0,1.5) {}; 	\node [above left] at (p1) {\small $p_{1}$};
		\node [left] at (-0.6,1.2) {$E_{1}$};
		\node[style=elt] (p2) at (0,3.5) {}; 	\node [above left] at (p2) {\small $p_{2}$};
		\node [left] at (-0.6,3.2) {$E_{2}$};
		\node[style=elt] (p3) at (2,0) {}; 	\node [below right] at (p3) {\small $p_{3}$};
		\node [below] at (1.8,-0.5) {$E_{3}$};
		\node[style=elt] (p4) at (4.5,0) {}; 	\node [below right] at (p4) {\small $p_{4}$};
		\node [below] at (4.3,-0.5) {$E_{4}$};
		\node[style=elt] (p5) at (7,2.5) {}; 	\node [below right] at (p5) {\small $p_{5}$};
		\node [right] at (7.5,2.8) {$E_{5}$};
		\node[style=elt] (p6) at (7,4.5) {}; 	\node [below right] at (p6) {\small $p_{6}$};
		\node [right] at (7.5,4.8) {$E_{6}$};
		\node[style=elt] (p7) at (3,6) {}; 	\node [above left] at (p7) {\small $p_{7}$};
		\node [above] at (3.2,6.5) {$E_{7}$};
		\node[style=elt] (p8) at (5,6) {}; 	\node [above left] at (p8) {\small $p_{8}$};
		\node [above] at (5.2,6.5) {$E_{8}$};
		\end{tikzpicture}.	
	\caption{Action of $\varphi_{*}$ on $\operatorname{Div}(\mathcal{X})$.}
	\label{fig:phi-action}
\end{figure}

The expression \eqref{eq:generic-phi} 
for the ``half'' of the QRT mapping $\varphi=\sigma_{xy}\circ r_y$ becomes
\begin{equation}\label{eq:phi-map}
\varphi:(x,y)\mapsto \left(\frac{(x-a)(x-a^{-1})}{y(x+a)(x+a^{-1})},\ x\right).
\end{equation}
This mapping preserves the elliptic fibration~\eqref{exy} and lifts to an automorphism $\varphi$ of the surface $\mathcal{X}$,
$\varphi:\mathcal{X}\to \bar{\mathcal{X}}(=\mathcal{X})$. 
On the Picard lattice 
$\operatorname{Pic}(\mathcal{X}) = \operatorname{Span}_{\mathbb{Z}}\{ \mathcal{H}_{x}, \mathcal{H}_{y},\mathcal{E}_{1},\dots,\mathcal{E}_{8} \}$, 
the push-forward and the pull-back actions of the mapping are given by
\begin{align}
\varphi_*: &\begin{array}{l}
\mathcal{H}_x\mapsto \bar{\mathcal{H}}_y,\quad 
\mathcal{H}_y\mapsto \bar{\mathcal{H}}_x+2\bar{\mathcal{H}}_y-\bar{\mathcal{E}}_1-\bar{\mathcal{E}}_2-\bar{\mathcal{E}}_5-\bar{\mathcal{E}}_6,\\
\mathcal{E}_1\mapsto \bar{\mathcal{E}}_4,\quad 
\mathcal{E}_2\mapsto \bar{\mathcal{E}}_3,\quad
\mathcal{E}_3\mapsto \bar{\mathcal{H}}_y-\bar{\mathcal{E}}_1,\quad
\mathcal{E}_4\mapsto \bar{\mathcal{H}}_y-\bar{\mathcal{E}}_2,\\
\mathcal{E}_5\mapsto \bar{\mathcal{E}}_8,\quad
\mathcal{E}_6\mapsto \bar{\mathcal{E}}_7,\quad
\mathcal{E}_7\mapsto \bar{\mathcal{H}}_y-\bar{\mathcal{E}}_5,\quad
\mathcal{E}_8\mapsto \bar{\mathcal{H}}_y-\bar{\mathcal{E}}_6
\end{array}
\label{eq:varphi_*-map}
\end{align}
\begin{align}
\varphi^*: &\begin{array}{l}
\bar{\mathcal{H}}_x \mapsto  2\mathcal{H}_x+\mathcal{H}_y-\mathcal{E}_3-\mathcal{E}_4-\mathcal{E}_7-\mathcal{E}_8, \quad \bar{\mathcal{H}}_y\mapsto \mathcal{H}_x\\
\bar{\mathcal{E}}_1\mapsto \mathcal{H}_x-\mathcal{E}_3,\quad \bar{\mathcal{E}}_2\mapsto \mathcal{H}_x-\mathcal{E}_4,\quad \bar{\mathcal{E}}_3\mapsto \mathcal{E}_2,\quad  \bar{\mathcal{E}}_4\mapsto \mathcal{E}_1\\
\bar{\mathcal{E}}_5\mapsto \mathcal{H}_x-\mathcal{E}_7,\quad \bar{\mathcal{E}}_6\mapsto \mathcal{H}_x-\mathcal{E}_8,\quad \bar{\mathcal{E}}_7\mapsto \mathcal{E}_6,\quad  \bar{\mathcal{E}}_8\mapsto \mathcal{E}_5. \end{array}
\label{eq:varphi^*-map}
\end{align}

The push-forward map $\varphi_{*}$ on $\operatorname{Div}(\mathcal{X})$ can also be described by the diagram on 
Figure~\ref{fig:phi-action}.

\begin{figure}[h]
	\centering
\begin{tikzpicture}[>=stealth,
	elt/.style={circle,draw=purple!100, fill=purple!60, thick, inner sep=0pt,minimum size=2mm}]
	\draw[black,thick,->] (-1,0) -- (13,0);
	\node at (13.5,0) {$k$};
	\node[style=elt] (p0) at (0,0) {}; \node [below] at (p0.south) {$0$};
	\node[style=elt] (p1) at (2,0) {}; \node [below] at (p1.south) {\scriptsize $4 \mathfrak{i} (a + a^{-1})$};
	\node[style=elt] (p2) at (4,0) {}; \node [below] at (p2.south) {\scriptsize \quad $-4 \mathfrak{i} (a + a^{-1})$};
	\node[style=elt] (p3) at (7,0) {}; \node [below] at (p3.south) {\scriptsize $(a-a^{-1})^{2}$};
	\node[style=elt] (p4) at (9,0) {}; \node [below] at (p4.south) {\scriptsize \quad $-(a-a^{-1})^{2}$};
	\node[style=elt] (p5) at (11,0) {}; \node [below] at (p5.south) {$\infty$};
	\draw[black, thick,dotted,rounded corners] (-1,5) -- (-1,8) -- (13,8) -- (13,2) -- (-1,2) -- (-1,5);
	\draw[red, thick, ->] (5.5,1.8) -- (5.5,0.2); \node[red] at (5.7,1) {$\pi$};
	\draw[red, dashed, ->] (0,1.8) -- (0,0.2); 	\draw[red, dashed, ->] (11,1.8) -- (11,0.2);
	\draw[red, dashed, ->] (8,1.8) -- (8,1) -- (7,0.2);	\draw[red, dashed, ->] (8,1) -- (9,0.2);
	\draw[red, dashed, ->] (3,1.8) -- (3,1) -- (2,0.2);	\draw[red, dashed, ->] (3,1) -- (4,0.2);
	\begin{scope}[xshift=5.5cm,yshift=5cm]
	\draw[thick,purple] (0,0) ellipse (0.8cm and 2cm);
	\node at (0,-2.5) {generic $A_{0}^{(1)}$};
	\end{scope}
	\begin{scope}[xshift=0cm,yshift=3cm]
	\draw[thick,purple] (-0.5,0) .. controls (1,1) and (1,3) .. (-0.5,4);
	\draw[thick,purple] (0.5,0) .. controls (-1,1) and (-1,3) .. (0.5,4);
	\node at (0,-0.5) {$A_{1}^{(1)}$};
	\end{scope}
	\begin{scope}[xshift=8cm,yshift=3cm]
	\draw[thick,purple] (-0.5,0) .. controls (1,1) and (1,3) .. (-0.5,4);
	\draw[thick,purple] (0.5,0) .. controls (-1,1) and (-1,3) .. (0.5,4);
	\node at (0,-0.5) {$A_{1}^{(1)}$};
	\end{scope}
	\begin{scope}[xshift=3cm,yshift=3cm]
	\draw[thick,purple] (-0.5,4) .. controls (1.5,1) and (1,0) .. (0,0) .. controls (-1,0) and (-1.5,1) .. (0.5,4);
	\node at (0,-0.5) {$A_{0}^{(1)*}$};
	\end{scope}
	\begin{scope}[xshift=11cm,yshift=3cm]
	\draw[thick,purple] (-0.5,0) -- (1.5,2.5) (-1.5,1.5) -- (0.5,4);
	\draw[thick,purple] (0.5,0) -- (-1.5,2.5) (1.5,1.5) -- (-0.5,4);
	\node at (0,-0.5) {$A_{3}^{(1)}$};
	\end{scope}
	\node at (5.5,7.5) {$\mathbb{P}^{1} \times  \mathbb{P}^{1}$};
	\end{tikzpicture}	
	\caption{Different Types of Fibers of the QRT Map.}
	\label{fig:fibers}
\end{figure}
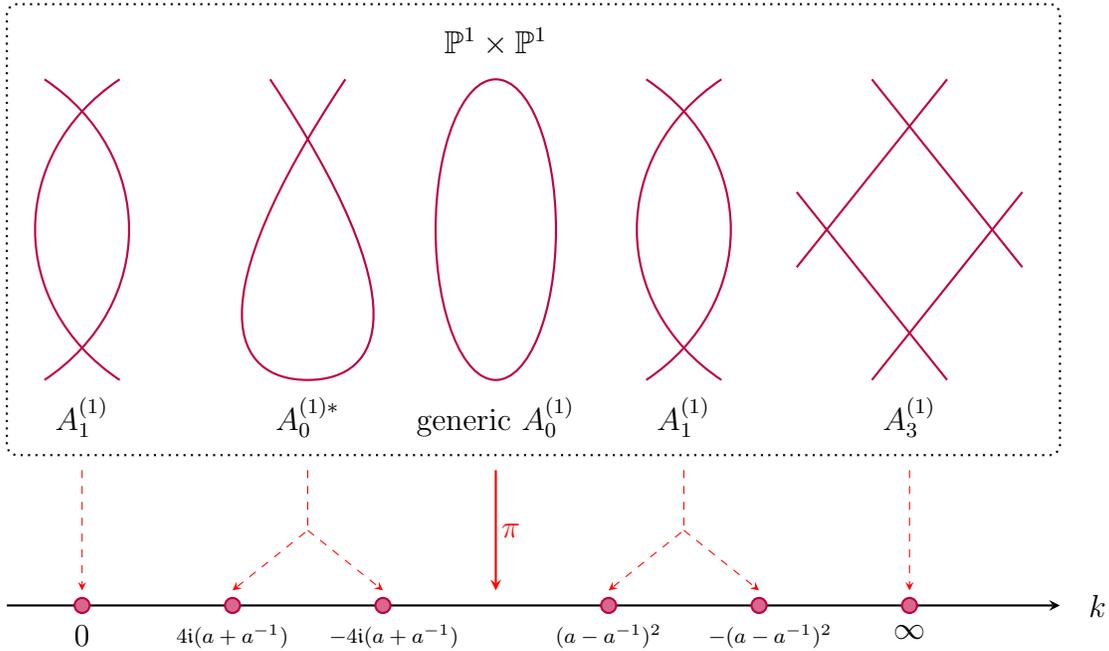

In order to see the singular fibers of our elliptic fibration, we put curves $\Gamma_{k}$ 
in this pencil to the Weierstrass normal form using Schwartz's algorithm (see Appendix~\ref{app:CFS}) 
and compute the elliptic discriminant $\Delta$. For our choice of the matrices 
$\mathbf{A}$ and $\mathbf{B}$  singular fibers appear at the parameter values
$k=0$, $k=\pm4i(a+a^{-1})$, $k=\pm(a-a^{-1})$, and $k=\infty$, 
as shown on Figure~\ref{fig:fibers}.

Choosing fibers of different type we obtain different deautonomizations of the same mapping $\varphi$, as we
describe in the next few subsections.

\subsection{Smooth fiber of type $A_{0}^{(1)}$}\label{Sect_generic}  
For a generic value of $k$, $\Gamma_{k}$ given by \eqref{exy} is a smooth elliptic curve. 
For this choice of a generic fiber the surfaces in the family $\mathcal{X}_A$ are of type
$A_{0}^{(1)}$ and the symmetry group of the family is the affine Weyl group $W(E_8^{(1)})$.
It is generated by reflections $w_{i}$ defined by simple roots $\alpha_{i}$, where 
$w_{i}$ act on $\operatorname{Pic}(\mathcal{X})$ in the usual way, for $\mathcal{D} \in \operatorname{Pic}(\mathcal{X})$,
$w_i^{*}(\bar{\mathcal{D}})= \mathcal{D}+ (\alpha_i\bullet\mathcal{D})\alpha_i$. 
We take simple roots $\alpha_{i}$ to be the following:
\begin{equation}\label{eq:a8roots}
\begin{aligned}
	\alpha_0&=\mathcal{E}_1-\mathcal{E}_2, & 
	\alpha_1&=\mathcal{H}_x-\mathcal{H}_y,  &	
	\alpha_2&=\mathcal{H}_y-\mathcal{E}_1-\mathcal{E}_2, \\
	\alpha_3& =\mathcal{E}_2-\mathcal{E}_3,&
	\alpha_4&=\mathcal{E}_3-\mathcal{E}_4, &
	\alpha_5&=\mathcal{E}_4-\mathcal{E}_5, \\
	\alpha_6&=\mathcal{E}_5-\mathcal{E}_6, &	
	\alpha_7&=\mathcal{E}_6-\mathcal{E}_7, &
	\alpha_8&=\mathcal{E}_7-\mathcal{E}_8.
	\end{aligned} 	
\end{equation}
It is easy to check that the configuration of these roots $\alpha_{i}$ is described by the affine Dynkin diagram $E_8^{(1)}$,
as shown on Figure~\ref{fig:dyn-e8}, where we also indicated the coefficients of the representation of the 
null root vector $\delta$, which corresponds to the class of the anti-canonical divisor $-\mathcal{K}_{\mathcal{X}}$, 
for the $E_{8}^{(1)}$ Cartan matrix in terms of the simple roots $\alpha_{i}$,
\begin{equation*}
\delta=-\mathcal{K}_{\mathcal{X}}=3\alpha_0+2\alpha_1+4\alpha_2+6\alpha_3+3\alpha_6+5\alpha_4+4\alpha_5+2\alpha_7+\alpha_8.
\end{equation*}

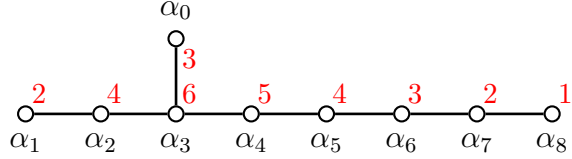
\begin{figure}[ht]
	\centering
	\begin{tikzpicture}[>=stealth, elt/.style={circle,draw=black!100,thick, inner sep=0pt,minimum size=2mm,baseline=-20pt}]
	\path 	( -2,0) 	node 	(D1) [elt] {}
	        ( -1,0) 	node 	(D2) [elt] {}
	        ( 0,0) 		node  	(D3) [elt] {}
	        ( 1,0) 		node 	(D4) [elt] {}
	        ( 2,0) 		node 	(D5) [elt] {}
	        ( 3,0) 		node 	(D6) [elt] {}
	        ( 4,0) 		node 	(D7) [elt] {}
	        ( 5,0) 		node 	(D8) [elt] {}
	        ( 0,1) 		node 	(D0) [elt] {};
	\draw [black,line width=1pt ] (D1) -- (D2) -- (D3) -- (D4) -- (D5) -- (D6) -- (D7) -- (D8);
	\draw [black,line width=1pt ] (D3) -- (D0);

	\node[below] at (D1.south) {$\alpha_{1}$};
	\node[below] at (D2.south) {$\alpha_{2}$};
	\node[below] at (D3.south) {$\alpha_{3}$};
	\node[below] at (D4.south) {$\alpha_{4}$};
	\node[below] at (D5.south) {$\alpha_{5}$};
	\node[below] at (D6.south) {$\alpha_{6}$};
	\node[below] at (D7.south) {$\alpha_{7}$};
	\node[below] at (D8.south) {$\alpha_{8}$};
	\node[above] at (D0.north) {$\alpha_{0}$};
	
						\node[red,above] at (D1.east) {\small $\ 2$};
						\node[red,above] at (D2.east) {\small $\ 4$};
						\node[red,above] at (D3.east) {\small $\ 6$};
						\node[red,above] at (D4.east) {\small $\ 5$};
						\node[red,above] at (D5.east) {\small $\ 4$};
						\node[red,above] at (D6.east) {\small $\ 3$};
						\node[red,above] at (D7.east) {\small $\ 2$};
						\node[red,above] at (D8.east) {\small $\ 1$};
						\node[red,below] at (D0.east) {\small $\ 3$};
	
		\end{tikzpicture}	
		\caption{Affine Dynkin diagram $E_{8}^{(1)}$, the nodes represent $\alpha_{i}\bullet\alpha_{i} = -2$ 
		and nodes corresponding to $\alpha_{i}$ and $\alpha_{j}$ are connected with an edge when 
		$\alpha_{i}\bullet\alpha_{j} = 1$.
		}
	\label{fig:dyn-e8}
\end{figure}

The action~\eqref{eq:varphi^*-map} of $\varphi^{*}$, when restricted on the symmetry root sublattice
$Q = Q(E_{8}^{(1)}) = \operatorname{Span}_{\mathbb{Z}}\{\alpha_{0},\dots,\alpha_{8}\} = (- \mathcal{K}_{\mathcal{X}})^{\perp}$, is 
\begin{align}
\bar{\alpha}_0&\mapsto-\alpha_4\notag\\
\bar{\alpha}_1&\mapsto2\alpha_0+\alpha_1+2\alpha_2+4\alpha_3+3\alpha_4+2\alpha_5+2\alpha_6+2\alpha_7+\alpha_8\notag\\
\bar{\alpha}_2&\mapsto-\alpha_0-\alpha_1-\alpha_2-2\alpha_3-\alpha_4\notag\\
\bar{\alpha}_3&\mapsto\alpha_0+\alpha_1+\alpha_2+\alpha_3+ \alpha_{4}\notag\\
\bar{\alpha}_4&\mapsto-\alpha_0\\
\bar{\alpha}_5&\mapsto-\alpha_1-\alpha_2-\alpha_3-\alpha_4-\alpha_5-\alpha_6-\alpha_7\notag\\
\bar{\alpha}_6&\mapsto-\alpha_8\notag\\
\bar{\alpha}_7&\mapsto\alpha_0+\alpha_1+\alpha_2+2\alpha_3+2\alpha_4+2\alpha_5+2\alpha_6+\alpha_7+\alpha_8\notag\\
\bar{\alpha}_8&\mapsto-\alpha_6.\notag
\end{align}
Thus, $\varphi^{*}$ is not a translation on $Q$ but we can check that $(\varphi^{4})^{*}$ is:
\begin{equation}\label{eq:f4-act}
(\varphi^{4})^{*}:(\bar{\alpha}_0,...,\bar{\alpha}_8)\mapsto
(\alpha_0,...,\alpha_8)+(0,2,-2,1,0,-1,0,1,0)\delta.	
\end{equation}	




The mapping $\varphi$ induces the mapping $\varphi_{*} = (\varphi^{*})^{-1}$ on the Picard lattice. Using the 
standard technique of Lemma~3.11 of \cite{Kac}, we can represent $\varphi_{*}$ in terms of the generators
of $W(E_{8}^{(1)})$ as
\begin{align}
\varphi_{*}=& 	w_{6} w_{5} w_{4}\,  w_{3} w_{2} w_{0}\,  w_{3} w_{4} w_{5}\,  
				w_{6} w_{7} w_{8}\,  w_{5} w_{4} w_{3}\,  w_{2} w_{0} w_{3}\,  w_{4} w_{5} w_{6} \notag\\ 
			&	w_{7} w_{6} w_{0}\,  w_{2} w_{3} w_{4}\,  w_{0} w_{3} w_{1}\, w_{0}.
\label{decomp_E8}\end{align}
One way to perform the deautonomization procedure is to find the birational representation of the 
affine Weyl group (we give an example of how to do it in Appendix~\ref{app:bir-rep}), and take the composition
of these birational maps. However, the following issues arise.
\begin{enumerate}[(i)]\itemsep0em
\item We need to construct birational transformations for all of the generators.
\item We need to find the decomposition of $\varphi$ in terms of the generators of the group, as in \eqref{decomp_E8}.
\item Most significantly, we expect the resulting expression for the mapping to be very complicated and most likely it 
will not be written in a simple or factorized form.
\end{enumerate}
Therefore we use a different approach to deautonomization that is based on the factorization formulae of 
Section~\ref{SectFactorization}. 

Let us pick $\mathbf{a}\in \mathcal{A}$ and consider the corresponding 
surface $\mathcal{X}_{\mathbf{a}}$ that is obtained from $\mathbb{P}^{1} \times  \mathbb{P}^{1}$ by 
blowing up at eight points $(x_{i},y_{i})$ that lie on an elliptic curve $\mathcal{C}_{\mathbf{a}}$;
for now our parameters are the coordinates of the blowup points. There are $16$ of them, but we 
can use the $6$-parameter group  of M\"obius transformations 
$\operatorname{\mathbf{PGL}}_{2}(\mathbb{C}) \times  \operatorname{\mathbf{PGL}}_{2}(\mathbb{C})$ 
acting on each of the factors to normalize some of the parameters, and so the true number of parameters is $10$, 
as it should be. For our normalization, we use shifts to ensure that 
\begin{equation}\label{eq:norm-pts}
x_1=0,\qquad x_5=\infty,\qquad y_4=0,\qquad y_8=\infty,	
\end{equation}
exactly as it was in the autonomous case. This still 
leaves us with two free dilation parameters that we fix later. At this point we 
are able to apply Proposition~\ref{prop_blow} to obtain an explicit expressions for the deautonomized mapping. 
We have
\begin{align*}
	\varphi^{*}(V(\bar{x} - 0)) &= \pi_{*}\circ \tilde{\varphi}^{*}\circ \bar{\pi}^{*} (V(\bar{x} - 0))
	= \pi_{*} \circ \tilde{\varphi}^{*}((\bar{H}_{x} - \bar{E}_{1}) + \bar{E}_{1}) \\
	&= \pi_{*} ((H_{x} + H_{y} - E_{4} - E_{7} - E_{8}) + (H_{x} - E_{3}))\\
	&= V((x_4 - x_7) (x - x_8) y - (x_8 - x_7)(x - x_4) y_7) + V(x - x_{3}),\\
	\varphi^{*}(V(\bar{x} - \infty)) &= \pi_{*}\circ \tilde{\varphi}^{*}\circ \bar{\pi}^{*} (V(\bar{x} - \infty))
	= \pi_{*} \circ \tilde{\varphi}^{*}((\bar{H}_{x} - \bar{E}_{5}) + \bar{E}_{5}) \\
	&= \pi_{*} ((H_{x} + H_{y} - E_{3} - E_{4} - E_{8}) + (H_{x} - E_{7}))\\
	&= V((x_4 - x_3) (x - x_8) y - (x_8 - x_3) (x - x_4) y_3) + V(x - x_{7}),\\
	\varphi^{*}(V(\bar{y} - 0)) &= \pi_{*}\circ \tilde{\varphi}^{*}\circ \bar{\pi}^{*} (V(\bar{y} - 0))
	= \pi_{*} \circ \tilde{\varphi}^{*}((\bar{H}_{y} - \bar{E}_{4}) + \bar{E}_{4}) \\	
	&= \pi_{*} ((H_{x} - E_{1}) + (E_{1})) = V(x-0),\\
	\varphi^{*}(V(\bar{y} - \infty)) &= \pi_{*}\circ \tilde{\varphi}^{*}\circ \bar{\pi}^{*} (V(\bar{y} - \infty))
	= \pi_{*} \circ \tilde{\varphi}^{*}((\bar{H}_{y} - \bar{E}_{8}) + \bar{E}_{8}) \\	
	&= \pi_{*} ((H_{x} - E_{5}) + (E_{5})) = V(x-\infty).
\end{align*}
Thus, the deautonomized mapping is given by 
\begin{equation}\label{eq:ell_xyup}
	\left\{ 
	\begin{aligned}
		\bar{x} &= C_{1} \frac{(x-x_{3})}{(x-x_{7})}\cdot 
			\frac{(x_4 - x_7) (x - x_8) y - (x_8 - x_7)(x - x_4) y_7}{(x_4 - x_3) (x - x_8) y - (x_8 - x_3) (x - x_4) y_3},  \\
		\bar{y} &= C_{2} x,	
	\end{aligned}
	\right.
\end{equation}
and we can use the remaining dilation parameters to set, for example, $C_{1} = C_{2} = 1$.

\begin{remark}\label{rem:det-form}
	In making these computations the following observation is convenient. Expression of the form
	$\pi_{*} (H_{x} + H_{y} - E_{i} - E_{j} - E_{k})$ corresponds to a bi-degree $(1,1)$ curve
	passing through the points $p_{i}$, $p_{j}$, and $p_{k}$. The affine defining polynomial of such curve can 
	be computed using the following simple construction. To allow for coordinates of points to be 
	infinite, let $p_{i}$ have homogeneous coordinates
	$([X_{0}^{(i)}:X_{1}^{(i)}],[Y_{0}^{(i)}:Y_{1}^{(i)}])$. Then the defining polynomial $p(x,y)$
		can be expressed in the determinantal form as
	\begin{equation*}
	p(x,y) = \begin{vmatrix}
		1 & x & y & xy \\
		X_{1}^{(i)} Y_{1}^{(i)} & X_{0}^{(i)} Y_{1}^{(i)} & X_{1}^{(i)} Y_{0}^{(i)} & X_{0}^{(i)} Y_{0}^{(i)} \\
		X_{1}^{(j)} Y_{1}^{(j)} & X_{0}^{(j)} Y_{1}^{(j)} & X_{1}^{(j)} Y_{0}^{(j)} & X_{0}^{(j)} Y_{0}^{(j)} \\
		X_{1}^{(k)} Y_{1}^{(k)} & X_{0}^{(k)} Y_{1}^{(k)} & X_{1}^{(k)} Y_{0}^{(k)} & X_{0}^{(k)} Y_{0}^{(k)} \\	
	\end{vmatrix}.	
	\end{equation*}
	In particular, for $\pi_{*} (H_{x} + H_{y} - E_{4} - E_{7} - E_{8})$ with 	
	$p_{4}([x_{4}:1],[0:1])$, 	$p_{7}([x_{7}:1],[y_{7}:1])$, and	$p_{8}([x_{8}:1],[1:0])$,
	we get
	\begin{equation*}
	p(x,y) 
	= \begin{vmatrix}
		1 & x & y & xy \\
		1 & x_{4} & 0 & 0 \\
		1 & x_{7} & y_{7} & x_{7} y_{7} \\
		0 & 0 & 1 & x_{8}	
	\end{vmatrix}
	 =(x_8 - x_7)(x - x_4) y_7 -(x_4 - x_7) (x - x_8) y.
	\end{equation*}
\end{remark}

In the non-autonomous case, in addition to the mapping, we also need to keep track of the evolution of the 
parameters. It is possible to obtain the evolution of the remaining parameters $x_{i}$ and $y_{i}$ directly
from the mapping using \eqref{eq:varphi_*-map}. For example, 
from $\varphi_{*}(\mathcal{E}_{1}) = \bar{\mathcal{E}}_{4}$
we see that 
\begin{equation*}
\left(\bar{x}_{4},\bar{y}_{4}\right) = \varphi(x_{1},y_{1}) = \varphi(0,y_{1})
= \left(\frac{x_{3}}{x_{7}}\cdot 
\frac{(x_{4} - x_{7})x_{8}y_{1} - (x_{8} - x_{7})x_{4}y_{7}}{(x_{4} - x_{3})x_{8}y_{1} - (x_{8} - x_{3})x_{4}y_{3}},0 \right).\\
\end{equation*}
{Similarly, from $\varphi_{*}(\mathcal{H}_{x} - \mathcal{E}_{3}) = \bar{\mathcal{E}}_{1}$, we get
$\left(\bar{x}_{1},\bar{y}_{1}\right) = \varphi(x_{3},y) = (0,x_{3})$,
and so on. Unfortunately, the evolution of parameters obtained in this way is complicated and nonlinear. 
A better approach is to use the parameterization of the elliptic curve $\mathcal{C}_{\mathbf{a}}$ 
(which is essentially the period map).

The curve $\mathcal{C}_{\mathbf{a}}$ can be written in a parametric form as
\begin{equation*}
\iota: \mathbb{T} = \mathbb{C}/ (\mathbb{Z} + \mathbb{Z}\tau) 
\to \mathcal{C}_{\mathbf{a}}\subset \mathbb{P}^{1} \times \mathbb{P}^{1},\quad \iota(u) = (F(u),G(u)),
\end{equation*}
where $F$ and $G$ are some elliptic functions. 
In this way we obtain points
$e_{1},\dots,e_{8}\in \mathbb{T}$ such that $(x_{i},y_{i}) = \iota(e_{i})= (F(e_{i}),G(e_{i}))$, these 
points $e_{i}$ are our new parameters. 

We also get the pull-back map 
$\iota^{*}:\operatorname{Pic}(\mathcal{X}_{\mathbf{a}})\to \operatorname{Pic}(\mathbb{T})$. Since the abelian group structure on $\mathbb{T}$ is compatible with the linear equivalence relation, 
$[p_{1}] + [p_{2}] = [p_{1} + p_{2}]$, we can identify $\operatorname{Pic}(\mathbb{T})\simeq \mathbb{T}$.
Then $\iota^{*}(\mathcal{H}_{x}) = h_{x}$,  $\iota^{*}(\mathcal{H}_{y}) = h_{y}$ and
$\iota^{*}(\mathcal{E}_{i}) = e_{i}$, where $h_{x}$, $h_{y}$, and $e_{i}$ are just points in $\mathbb{T}$.
In general, for any $l_x \mathcal{H}_x+ l_y \mathcal{H}_y+ m_1\mathcal{E}_1+\dots+m_8\mathcal{E}_8\in \operatorname{Pic}(\mathcal{X}_{\mathbf{a}})$, we get
\begin{equation*}
\iota^*(l_x \mathcal{H}_x+ l_y \mathcal{H}_y+ m_1\mathcal{E}_1+\dots+m_8\mathcal{E}_8) =
l_x h_x+ l_y h_y+ m_1e_1+\dots+m_8e_8 \in \mathbb{T}.
\end{equation*}

Next, let $\varphi:\mathcal{X}=\mathcal{X}_{\mathbf{a}} \to \bar{\mathcal{X}}=\mathcal{X}_{\varphi(\mathbf{a})}$ 
be our deautonomized mapping. To see the evolution of parameters $e_{i},h_{x},h_{y}\in \mathbb{T}$,  
consider the following
two commuting diagrams:
%
%
\begin{center}
\begin{tikzpicture}[>=stealth,baseline=0cm]
	\begin{scope}
	\node (T)  at (0,0) {$\mathbb{T}$}; 	
	\node (X1)  at (0,1.5) {$\mathcal{X}$}; 	
	\node (X2)  at (2.5,1.5) {$\bar{\mathcal{X}}$}; 	
	\draw[->] (X1)--(X2) node [above,align=center,midway] {\scriptsize $\varphi$};
	\draw[->] (T)--(X1) node [left,align=center,midway] {\scriptsize $\iota$};
	\draw[->] (T)--(X2) node [below right,align=center,midway] {\scriptsize $\bar{\iota}$};
	\end{scope}
	\begin{scope}[xshift=5cm]
	\node (T)  at (0,0) {$\operatorname{Pic}(\mathbb{T})$}; 	
	\node (X1)  at (0,1.5) {$\operatorname{Pic}(\mathcal{X})$}; 	
	\node (X2)  at (3,1.5) {$\operatorname{Pic}(\bar{\mathcal{X}})$}; 	
	\draw[->] (X2)--(X1) node [above,align=center,midway] {\scriptsize $\varphi^{*}$};
	\draw[->] (X1)--(T) node [left,align=center,midway] {\scriptsize $\iota^{*}$};
	\draw[->] (X2)--(T) node [below right,align=center,midway] {\scriptsize $\bar{\iota}^{*}$};
	\end{scope}		
\end{tikzpicture}.
\end{center}
Thus, for any $\bar{D}\in \operatorname{Pic}(\bar{\mathcal{X}})$, 
we have $\bar{\iota}^*(\bar{D})=\iota^* \circ \varphi^*(\bar{D})$.
 Let 
$\bar{h}_x$, $\bar{h}_y$ and $\bar{e}_i$ denote the pull-backs 
$\bar{\iota}^*(\bar{\mathcal{H}}_x)$, $\bar{\iota}^*(\bar{\mathcal{H}}_y)$ and $\bar{\iota}^*(\bar{\mathcal{E}}_i)$.
Then \eqref{eq:varphi^*-map} gives us the following equations for the the parameter evolution on $\mathbb{T}$:  
\begin{alignat}{2}\label{eq:phi-act_points}
\bar{h}_x &=2h_x+h_y-e_3-e_4-e_7-e_8, & \quad \bar{h}_y&=h_x,\notag\\
\bar{e}_1&=h_x-e_3,\quad \bar{e}_2=h_x-e_4,\quad \bar{e}_3=e_2,&\quad  \bar{e}_4&=e_1,\\
\bar{e}_5&=h_x-e_7,\quad \bar{e}_6=h_x-e_8,\quad \bar{e}_7=e_6,&\quad  \bar{e}_8&=e_5.	\notag
\end{alignat}
Indeed, for example,
\begin{align*}
\bar{h}_{x} &= \bar{\iota}^*(\bar{\mathcal{H}}_x)=(\iota^* \circ \varphi^*)(\bar{\mathcal{H}}_x)
=\iota^*(2\mathcal{H}_x+\mathcal{H}_y-\mathcal{E}_3-\mathcal{E}_4-\mathcal{E}_7-\mathcal{E}_8)\\
&=2h_x+h_y-e_3-e_4-e_7-e_8.
\end{align*}
Such expression was originally found in 
\cite{KMNOY03} and later interpreted by using pull-back in \cite{ET05}. 
%

To determine the evolution of the points $(x_{i},y_{i})$ on $\mathbb{P}^{1} \times \mathbb{P}^{1}$, we need to 
describe the parameterization $\iota$ explicitly. Since
$\mathcal{C}_{\mathbf{a}}$ is given by an equation of bi-degree $(2,2)$, 
for a generic $x_{0}$, $\iota^{*}(V(x - x_{0})) = p_{1} + p_{2}\in \operatorname{Div(\mathbb{T})}$.
In view of the identification $\operatorname{Pic}(\mathbb{T})\simeq \mathbb{T}$, we have
$p_{1} + p_{2} = h_{x}$, and a similar statement is true for $\iota^{*}(V(y - y_{0}))$. Thus, we 
can write
\begin{alignat*}{2}
	\iota^{*}(V(x - 0)) &= u_{1} + (h_{x} - u_{1}), &\qquad 	\iota^{*}(V(y - 0)) &= u_{3} + (h_{y} - u_{3}),  \\
\iota^{*}(V(x - \infty)) &= u_{2} + (h_{x} - u_{2}), &\qquad \iota^{*}(V(y - \infty)) &= u_{4} + (h_{y} - u_{4}),
\end{alignat*}
for some points $u_1$, $u_2$, $u_3$, $u_4$ in $\mathbb{T}$. 
Hence the functions $F(u)$ and $G(u)$ parameterizing $\mathcal{C}_{\mathbf{a}}$ 
are elliptic functions of order $2$, i.e.~they have two zeroes and two poles (counted with multiplicities) 
in the fundamental domain $\mathbb{T}$.
Therefore, $F$ and $G$ can be explicitly written as
\begin{align}\label{FGu}
F(u)=c_1 \frac{[u-u_1][u-h_x+u_1]}{[u-u_2][u-h_x+u_2]},\quad
G(u)=c_2 \frac{[u-u_3][u-h_y+u_3]}{[u-u_4][u-h_y+u_4]},
\end{align}
for some normalization constants $c_1,\ c_2 \in \C^\times$. Here
$[t]$ denotes the  $\sigma$-function (or the odd theta function)\footnote{Recall that we have degenerations to  $[t]=\sin t$ (or $\sinh t$) and $[t]=t$ 
in the trigonometric and rational limits respectively.},
and thus $[-t]=-[t]$, $[0]=0$, and
the so called the Riemann relation
holds \cite{WW27},
\begin{align*}
&[t+a][t-a][b+c][b-c]+[t+b][t-b][c+a][c-a]\\
&\qquad   +[t+c][t-c][a+b][a-b]=0.
\end{align*}

Our chosen
$\operatorname{\mathbf{PGL}}_{2}(\mathbb{C}) \times  \operatorname{\mathbf{PGL}}_{2}(\mathbb{C})$ 
normalization \eqref{eq:norm-pts}  is equivalent to choosing  $u_1=e_1$, $u_2=e_5$, $u_3=e_4$, and $u_4=e_8$.
In this approach it is actually more convenient to use dilations to make $c_{1} = c_{2} = 1$.
Thus, we can assume that our parametric representation $\iota$ of $\mathcal{C}_{\mathbf{a}}$
is given by
\begin{align}\label{eq:FG}
F(u)=\frac{[u-e_1][u-h_x+e_1]}{[u-e_5][u-h_x+e_5]},\quad
G(u)=\frac{[u-e_4][u-h_y+e_4]}{[u-e_8][u-h_y+e_8]},
\end{align}
and similarly for $\bar{F}$ and $\bar{G}$. 
Then the action of $\varphi^{*}$ given by \eqref{eq:phi-act_points} results in the following equations:
%
\begin{align}\label{eq:FGup-phi}
\bar{F}(u)&=\frac{[u-\bar{e}_1][u-\bar{h}_x+\bar{e}_1]}{[u-\bar{e}_5][u-\bar{h}_x+\bar{e}_5]}
=\frac{[u-h_x+e_3][u-h_x-h_y+e_4+e_7+e_8]}{[u-h_x+e_7][u-h_x-h_y+e_3+e_4+e_8]},\notag\\
\bar{G}(u)&=\frac{[u-\bar{e}_4][u-\bar{h}_y+\bar{e}_4]}{[u-\bar{e}_8][u-\bar{h}_y+\bar{e}_8]}=
\frac{[u-e_1][u-h_x+e_1]}{[u-e_5][u-h_x+e_5]}=F(u).
\end{align}

\begin{remark} Note that in this way we have completely fixed the M\"obius transformation
	ambiguity. However, not only is $\bar{\mathbf{a}}$ in general different from $\mathbf{a}$,
	which is exactly the deautonomization phenomenon, but also the curve can move, 
	$\bar{\mathcal{C}}$ can in general be different from 
	$\mathcal{C}$.
\end{remark}
%
%
%
%
%
%
%
%
To evaluate the normalization constants $C_{1}$ and $C_{2}$, we note that, in view of \eqref{eq:varphi^*-map}, \eqref{eq:ell_xyup}, $\varphi^*(\bar{E}_2)=H_x-E_4$ implies
\begin{align*}
(\bar{x}_{2},\bar{y}_{2}) &=\varphi_*(V(x-x_4))= (C_{1},C_{2}x_{4}),
\end{align*}
while from \eqref{eq:FG}, \eqref{eq:FGup-phi} and \eqref{eq:phi-act_points} we have 
\begin{align*}
(\bar{x}_{2},\bar{y}_{2}) &= \left(\bar{F}(\bar{e}_{2}),\bar{G}(\bar{e_{2}})\right) 
= \left(\bar{F}(h_{x} - e_{4}),F(h_{x} - e_{4})\right)\\
&= \left(\frac{[e_{4} - e_{3}][h_{y} - e_{7} - e_{8}]}{[e_{4} - e_{7}][h_{y} - e_{3} - e_{8}]}, x_{4} \right),
\end{align*}
and therefore 
\begin{align}\label{ell_C12}
C_1=\frac{[e_4-e_3][h_y-e_7-e_8]}{[e_4-e_7][h_y-e_3-e_8]} \text{ and } C_2=1.
\end{align}

To summarize, the mapping $\varphi$ is deautonomized as \eqref{eq:ell_xyup} with normalization constant
given by \eqref{ell_C12}
and the transformation of parameters governed by \eqref{eq:phi-act_points}. We also want to point
out that an additive analogue of this mapping was considered in \cite{RaGr-2017}, where it was called
Class $I$, $(M, N, P, Q) = (2, 2, 2, 2)$ equation. 

%

\begin{remark}
	It is important to note that, since $\varphi_{*}$ is not a translation on the symmetry root lattice, 
	the resulting equation is not a so-called \emph{elliptic difference equation}. For an elliptic
	difference equation, the coefficients in the equation should change translationally in the argument of some
	elliptic function. 
	However, $\varphi^{4}$ is an elliptic 
	difference equation with the evolution of parameters $(x_{i},y_{i})$ under the mapping $(\varphi^{4})^{n}$
	given by 
	\begin{alignat*}{2}
		(x_{1},y_{1}) &= \left(0, G(e_{1} + n d )\right),\qquad & 
		(x_{2},y_{2}) &= \left(F(e_{2}), G(e_{2} + n d )\right), \\
		(x_{3},y_{3}) &= \left(F(e_{3} - n d), G(e_{3})\right),\qquad & 
		(x_{4},y_{4}) &= \left(F(e_{4} - n d), 0 \right), \\
		(x_{5},y_{5}) &= \left(\infty, G(e_{5} + n d )\right),\qquad & 
		(x_{6},y_{6}) &= \left(F(e_{6}), G(e_{6} + n d )\right), \\
		(x_{7},y_{7}) &= \left(F(e_{7} - n d), G(e_{7})\right),\qquad & 
		(x_{8},y_{8}) &= \left(F(e_{8} - n d ), G(e_{8} )\right), 
	\end{alignat*}
	where $F(u)$ and $G(u)$ are given by \eqref{eq:FG}, and 
	\begin{equation*}
	d = \iota^{*}(\delta) = 2 h_{x} + 2 h_{y} - e_{1} - e_{2} - e_{3} - e_{4} - e_{5} - e_{6} - e_{7} - e_{8}
	\end{equation*}
	is the pull-back on $\mathbb{T}$ of the class of the anticanocnical divisor on $\mathcal{X}$ (which is invariant
	under our mapping).
	We return to this observation in Section~\ref{sec:Special}, where we impose  constraints on the parameters 
	to ensure that either $\varphi^{2}$ or $\varphi$ become elliptic difference equations. This leads to examples
	of elliptic discrete Painlev\'e equations whose symmetry groups do not explicitly appear in the Sakai classification scheme.
\end{remark}


\subsection{Singular fiber of type $A_{0}^{(1)*}$}
\label{subs:A01}

This singular fiber occurs when $k=\pm 4i(a+a^{-1})$. 
In this case the curve $\Gamma_{k}$ given by \eqref{exy} is an irreducible nodal curve, 
and so, similarly to the generic case, 
the symmetry group of $\mathcal{X}_A$ is $W(E_8^{(1)})$ and the 
deautonomized mapping is given by same equation \eqref{eq:ell_xyup}. However, the evolution 
of parameters is \emph{different}. 

An irreducible nodal curve can be parameterized by 
$\iota: \mathbb{P}^{1}\to \mathbb{P}^{1} \times \mathbb{P}^{1}$ with $\iota(u) = (F(u),G(u))$
with the additional constraint $\iota(0) = \iota(\infty)$ that corresponding to the node. 
Thus, outside of the node $\mathcal{C}_{\mathbf{a}}$
can be parameterized by 
$\iota: \mathbb{C}^{\times}\to \mathbb{P}^{1}\times \mathbb{P}^{1}$. Since $\mathcal{C}_{\mathbf{a}}$
is of bi-degree $(2,2)$, $F$ and $G$ are \emph{rational} functions of degree $2$ 
on $\mathbb{C}^{\times}$ that can be written as 
\begin{align*}
F(u)=c_1 \frac{(u-u_{1})(u-u_{2})}{(u-u_{3})(u-u_{4})},\quad
G(u)=c_2 \frac{(u-u_{5})(u-u_{6})}{(u-u_{7})(u-u_{8})},
\end{align*}
where, in view of the nodal condition $\iota(0) = \iota(\infty)$ parameters
$u_{i}$ must satisfy $u_{1} u_{2} = u_{3} u_{4}$ and $u_{5} u_{6} = u_{7} u_{8}$.
On $\mathbb{C}^{\times }$ any two divisors of the same degree are linearly equivalent,
which is too strong for our purposes. Instead, we introduce the following weaker equivalence
relation on $\operatorname{Div}(\mathbb{C}^{\times})$:
\begin{equation}\label{eq:equivC*}
	\sum_{i=1}^{I} p_{i} \sim \sum_{j=1}^{J} q_{j}\qquad\text{if and only if}\qquad 
	I = J \text{ and } p_{1}\cdots p_{I} = q_{1}\cdots q_{J}.
\end{equation}
For this equivalence relation it is convenient to identify an equivalence class with the point
$p_{1}\cdots p_{I}\in \mathbb{C}^{\times}$.
%
In particular,
it is easy to see that pullbacks of two linearly equivalent divisors on 
$\mathcal{X}$ via $\iota^{*}$ are also equivalent in that way. 
For example, 
\begin{equation*}
u_{1} + u_{2} = \iota^{*}(V(x-0)) \sim u_{3} + u_{4}= \iota^{*}(V(x-\infty))\sim s + t= \iota^{*}(V(x-k)),
\end{equation*}
and so we put $h_{x} = u_{1}u_{2} = u_{3}u_{4}=st$. Similarly, $h_{y} = u_{5} u_{6} = u_{7} u_{8}$,
where we consider $h_{x}$ and $h_{y}$ as equivalence classes in $\operatorname{Div}(\mathbb{C}^{\times })/\sim$,
and so on.

Using this convention, we can write our parameterization of $\mathcal{C}_{\mathbf{a}}$ as 
\begin{align}\label{A0*FGu}
F(u)=c_1 \frac{(u-u_1)(u-h_x/u_1)}{(u-u_2)(u-h_x/u_2)},\quad
G(u)=c_2 \frac{(u-u_3)(u-h_y/u_3)}{(u-u_4)(u-h_y/u_4)},
\end{align}
and we can impose the same normalization \eqref{eq:norm-pts} as before, 
$c_1=c_2=1$, $u_1=e_1$, $u_2=e_5$, $u_3=e_4$, and $u_4=e_8$.
The parameter evolution 
now takes the multiplicative form
\begin{alignat}{2}\label{nodal_act_points}
\bar{h}_x &=h_x^2h_ye_3^{-1}e_4^{-1}e_7^{-1}e_8^{-1}, &\quad \bar{h}_y&=h_x,\notag\\
\bar{e}_1&=h_x e_3^{-1},\quad  \bar{e}_2=h_x e_4^{-1},&\quad \bar{e}_3&=e_2,\quad  \bar{e}_4=e_1,\\
\bar{e}_5&=h_xe_7^{-1},\quad \bar{e}_6=h_xe_8^{-1},&\quad \bar{e}_7&=e_6,\quad  \bar{e}_8=e_5.\notag
\end{alignat} 
In particular, the normalization constant $C_{1}$ in \eqref{eq:ell_xyup} becomes 
\begin{equation}
C_1=\bar{F}(\bar{e}_{2}) = \frac{(e_{4} - e_{3}) (h_{y} - e_{7} e_{8})}{(e_{4} - e_{7})(h_{y} - e_{3} e_{8})}. 
\end{equation}

\begin{remark}
	This parameterization and the multiplicative structure of the parameter evolution 
	can also be obtained via trigonometric 
	degeneration of \eqref{eq:FG}, but the above description is more concrete and 
	so is better for our purposes. 
\end{remark}


%

\subsection{Singular fiber of type $A_{1}^{(1)}$}

This case corresponds to $k=0$, or to $k = \pm \left(a-a^{-1}\right)^{2}$, so we need to consider three different cases.
We explain the case $k=0$ in detail and briefly outline the changes needed for the other two cases.

\subsubsection{Case $k=0$} 
\label{ssub:case_k_0}

When $k=0$, the equation defining the fiber $\Gamma_{0}$ factors as 
\begin{equation*}
(xy+a(x+y)-1)(xy+a^{-1}(x+y)-1) = 0, 
\end{equation*}
and hence an anti-canonical divisor decomposes as $D_{0} + D_{1}$, where the prime divisors
$D_0 = V(xy+a(a+y)-1)$ and $D_1 = V(xy+a^{-1}(x+y)-1)$ have classes
\begin{alignat*}{2}
{\mathcal D}_0&=\mathcal{H}_x+\mathcal{H}_y-\mathcal{E}_1-\mathcal{E}_3-\mathcal{E}_6-\mathcal{E}_8,
\qquad & &\ \mathcal{D}_{0} \bullet \mathcal{D}_{1} = 1, \\
{\mathcal D}_1&=\mathcal{H}_x+\mathcal{H}_y-\mathcal{E}_2-\mathcal{E}_4-\mathcal{E}_5-\mathcal{E}_7,
\qquad & 
&\begin{tikzpicture}[baseline=-0.2cm]
	\node (c1) at (0,0) [circle, thick, draw=black!100, inner sep=0pt,minimum size=1.3ex] {};
	\node (c2) at (1,0) [circle, thick, draw=black!100, inner sep=0pt,minimum size=1.3ex] {};
	\draw[thick, double distance = .4ex] (c1) -- (c2);
	\node[left] at (c1.west) {$\mathcal{D}_{0}$};
	\node[right] at (c2.east) {$\mathcal{D}_{1}$};
\end{tikzpicture}.					
\end{alignat*}
Thus, in this case we get the family $\mathcal{X}_{\mathcal{A}}$ of type $A_{1}^{(1)}$ with the
symmetry group of type $E_7^{(1)}$.

\begin{remark}	
We can take the simple roots $\alpha_{i}$ for the symmetry root sublattice $Q(E_{7}^{(1)}) = \operatorname{Span}_{\mathbb{Z}}\{\alpha_{0},\dots,\alpha_{7}\}$  
of $E_{7}^{(1)}$ to be the following:
\begin{align*}
	\alpha_{0} &=\mathcal{H}_{x}-\mathcal{H}_{y}, & 
	\alpha_{1}&=\mathcal{E}_{6}-\mathcal{E}_{8},  &	
	\alpha_{2} &=\mathcal{E}_{3}-\mathcal{E}_{6}, & 
	\alpha_{3}&=\mathcal{E}_{1}-\mathcal{E}_{3},  \\	
	\alpha_{4} &=\mathcal{H}_{y}-\mathcal{E}_{1}-\mathcal{E}_{2}, &
	\alpha_{5}&=\mathcal{E}_{2}-\mathcal{E}_{4}, &
	\alpha_{6} &=\mathcal{E}_{4}-\mathcal{E}_{5}, &
	\alpha_{7}&=\mathcal{E}_{5}-\mathcal{E}_{7}. \\
\end{align*} 

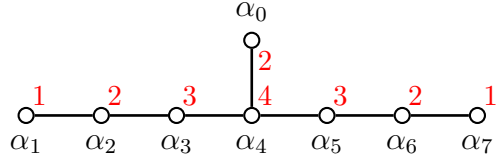
\begin{figure}[h]
	\centering
	\begin{tikzpicture}[>=stealth, elt/.style={circle,draw=black!100,thick, inner sep=0pt,minimum size=2mm,baseline=-20pt}]
	\path 	( -3,0) 	node 	(D1) [elt] {}
	        ( -2,0) 	node 	(D2) [elt] {}
	        ( -1,0) 	node  	(D3) [elt] {}
	        ( 0,0) 		node 	(D4) [elt] {}
	        ( 1,0) 		node 	(D5) [elt] {}
	        ( 2,0) 		node 	(D6) [elt] {}
	        ( 3,0) 		node 	(D7) [elt] {}
	        ( 0,1) 		node 	(D0) [elt] {};
	\draw [black,line width=1pt ] (D1) -- (D2) -- (D3) -- (D4) -- (D5) -- (D6) -- (D7);
	\draw [black,line width=1pt ] (D4) -- (D0);

	\node[below] at (D1.south) {$\alpha_{1}$};
	\node[below] at (D2.south) {$\alpha_{2}$};
	\node[below] at (D3.south) {$\alpha_{3}$};
	\node[below] at (D4.south) {$\alpha_{4}$};
	\node[below] at (D5.south) {$\alpha_{5}$};
	\node[below] at (D6.south) {$\alpha_{6}$};
	\node[below] at (D7.south) {$\alpha_{7}$};
	\node[above] at (D0.north) {$\alpha_{0}$};
	
						\node[red,above] at (D1.east) {\small $\ 1$};
						\node[red,above] at (D2.east) {\small $\ 2$};
						\node[red,above] at (D3.east) {\small $\ 3$};
						\node[red,above] at (D4.east) {\small $\ 4$};
						\node[red,above] at (D5.east) {\small $\ 3$};
						\node[red,above] at (D6.east) {\small $\ 2$};
						\node[red,above] at (D7.east) {\small $\ 1$};
						\node[red,below] at (D0.east) {\small $\ 2$};
	
		\end{tikzpicture}	
		\caption{Affine Dynkin diagram $E_{7}^{(1)}$.}
	\label{fig:dyn-e7}
\end{figure}

The configuration of $\alpha_{i}$ and the coefficients of the representation of the 
null root vector $\delta$ in terms of $\alpha_{i}$ are 
shown on Figure~\ref{fig:dyn-e7}, and hence the anti-canonical divisor class can be written as
\begin{equation*}
-\mathcal{K}_{\mathcal{X}}=\delta=\mathcal{D}_{0}+\mathcal{D}_{1}
=2 \alpha_{0}+ \alpha_{1} + 2\alpha_{2} + 3\alpha_{3} + 4\alpha_{4} + 3\alpha_{5} + 2\alpha_{6} + \alpha_{7}.
\end{equation*}

The action~\eqref{eq:varphi^*-map} of $\varphi^{*}$ on $\alpha_{i}$ is 
\begin{align*}
\bar{\alpha}_0&\mapsto \alpha_{0} + \alpha_{1} + \alpha_{2} + 2(\alpha_{3} + \alpha_{4} + \alpha_{5}) + \alpha_{6} + \alpha_{7},\\
\bar{\alpha}_1&\mapsto \alpha_{0} + \alpha_{1} + \alpha_{2} + \alpha_{3} + \alpha_{4} + \alpha_{5} + \alpha_{6},\\
\bar{\alpha}_2&\mapsto - (\alpha_{0} + \alpha_{1} + \alpha_{2} + \alpha_{3} + \alpha_{4}),\\
\bar{\alpha}_3&\mapsto \alpha_{0} + \alpha_{3} + \alpha_{4}, \\
\bar{\alpha}_4&\mapsto - (\alpha_{0} + \alpha_{3} + \alpha_{4} + \alpha_{5}),\\
\bar{\alpha}_5&\mapsto \alpha_{0} + \alpha_{4} + \alpha_{5}, \\
\bar{\alpha}_6&\mapsto  - (\alpha_{0} + \alpha_{4} + \alpha_{5} + \alpha_{6} + \alpha_{7}),\\
\bar{\alpha}_7&\mapsto \alpha_{0} + \alpha_{2} + \alpha_{3} + \alpha_{4} + \alpha_{5} + \alpha_{6} + \alpha_{7}.
\end{align*}

So $\varphi^{*}$ is not a translation but we can check that again, $(\varphi^{4})^{*}$ is:
\begin{equation*}
(\varphi^{4})^{*}:(\bar{\alpha}_{0},...,\bar{\alpha}_{7})\mapsto
(\alpha_{0},...,\alpha_{7})+(2,1,-1,1,-2,1,-1,1)\delta.
\end{equation*}	

The action of $\varphi_{*}$ on $\operatorname{Pic}(\mathcal{X})$ 
can be written in the generators of $\tilde{W}(E_{7}^{(1)})$ as
\begin{equation*}
\varphi_{*} = \sigma\, w_{1} w_{2} w_{3}\, w_{2} w_{0} w_{5}\, w_{4} w_{3} w_{7}\, w_{6} w_{5} w_{6}\, w_{4} w_{3} w_{2},
\end{equation*}
where $\sigma$ is the reflection symmetry of the $E_{7}^{(1)}$ Dynkin diagram
exchanging $\alpha_{1}$ with $\alpha_{7}$, $\alpha_{2}$ with $\alpha_{6}$, and $\alpha_{3}$ with $\alpha_{5}$.

\end{remark}

Consider now the evolution of parameters. First, we parameterize the irreducible divisors 
$D_{0}$ and $D_{1}$ by $\iota_i:\mathbb{P}^{1} \to \mathbb{P}^1\times \mathbb{P}^1$, where 
$\iota_{i}(u) = (F_i(u),G_i(u))$, $i=0,1$. In this case 
$F_i(u)$ and $G_i(u)$ are a rational function of degree $1$ on $\mathbb{P}^{1}$,
\begin{equation}\label{eq:param-A11}
	\begin{aligned}
		F_{0}(u) &= c_{0} \frac{u - u_{1}}{u - u_{2}}, &\qquad G_{0}(u) &= d_{0} \frac{u - u_{3}}{u - u_{4}},\\
		F_{1}(u) &= c_{1} \frac{u - u_{5}}{u - u_{6}}, &\quad G_{1}(u) &= d_{1} \frac{u - u_{7}}{u - u_{8}}.		
	\end{aligned}
\end{equation}
 The 
divisors $D_{0}$ and $D_{1}$ intersect at two points, and without loss of generality
we can assume that the corresponding values of the parameters are $0$ and $\infty$; in fact,
it is convenient to require $\iota_{0}(0) = \iota_{1}(\infty)$ and $\iota_{0}(\infty) = \iota_{1}(0)$,
which results in
\begin{equation}\label{eq:param-constr}
	c_{0} u_{1} = c_{1} u_{2},\quad d_{0}u_{3} = d_{1}u_{4}, \quad c_{0}u_{6} = c_{1}u_{5}, 
	\quad d_{0}u_{8} = d_{1} u_{7},
\end{equation}
and so $u_{1}u_{5} = u_{2}u_{6}$ and $u_{3}u_{7} = u_{4}u_{8}$. 
The part of $\mathcal{C}_{\mathbf{a}}$ away from the intersection points is parameterized by 
$\iota_i:\mathbb{C}^{\times } \to \mathbb{P}^1\times \mathbb{P}^1$, and that is where our 
mapping parameters lie.

In view of \eqref{eq:param-constr} we can write 
\begin{equation*}
	F_{1}(u) = c_{0} \frac{u u_{1} - u_{1} u_{5}}{u u_{2} - u_{1} u_{5}}, \qquad
	G_{1}(u) = d_{0} \frac{u u_{3} - u_{3} u_{7}}{u u_{4} - u_{3} u_{7}}.
\end{equation*}

For a divisor $D\in \operatorname{Div(\mathcal{X})}$,
its pre-image $\iota^{*}(D)\in \mathbb{C}^{\times }$ is now defined as 
$\iota^{*}(D) = \iota_{0}^{*}(D) + \iota_{1}^{*}(D)$. Also, 
as in Subsection~\ref{subs:A01}, we impose a weak equivalence relation \eqref{eq:equivC*}
on $\operatorname{Div}(\mathbb{C}^{\times})$.
Then $\iota^{*}V(x-k) = s + t\in \mathbb{C}^{\times}$, where $st = u_{1}u_{5} = u_{2}u_{6}$,
so we again put $h_{x}=u_{1}u_{5}\in \operatorname{Div}(\mathbb{C}^{\times })/\sim$, 
$h_{y} = u_{3}u_{7}$, and so on. It is now clear that this weak equivalence relation 
is compatible with the linear equivalence on $\operatorname{Pic}(\mathcal{X})$,
and in this way we get the multiplicative parameter evolution, as in \eqref{nodal_act_points}. 

The parametrerization \eqref{eq:param-A11} then takes the form
\begin{equation}
	\begin{aligned}
		F_{0}(u) &= c \frac{u - u_{1}}{u - u_{2}}, &\qquad G_{0}(u) &= d \frac{u - u_{3}}{u - u_{4}},\\
		F_{1}(u) &= c \frac{u u_{1} - h_{x}}{u u_{2} - h_{x}}, &\quad G_{1}(u) &= d \frac{u u_{3} - h_{y}}{u u_{4} - h_{y}}.		
	\end{aligned}
\end{equation}

Using M\"obius transformations to ensure the normaliztion \eqref{eq:norm-pts}, we see that 
$u_{1} = e_{1}$, $u_{2} = h_{x}/e_{5}$, $u_{3} = h_{y}/e_{4}$, and $u_{4} = e_{8}$. 
To also ensure that $\bar{y}=x$, we use rescaling to put $c = 1$ and $d = (e_{4}e_{8})/h_{y}$.
Our final parameterization expressions then are:
\begin{equation}\label{A1*FGu2}
	\begin{aligned}
		F_{0}(u) &= \frac{u - e_{1}}{u - h_{x}/e_{5}}, &\qquad G_{0}(u) &= \frac{e_{4}e_{8}}{h_{y}}  \frac{u - h_{y}/e_{4}}{u - e_{8}},\\
		F_{1}(u) &= \frac{e_{1}e_{5}}{h_{x}}  \frac{u - h_{x}/e_{1}}{u - e_{5}}, &\quad G_{1}(u) &=  \frac{u -e_{4}}{u - h_{y}/e_{8}}.		
	\end{aligned}
\end{equation}
Thus, the deautonomization of the mapping $\varphi$ in this case is again given by~\eqref{eq:ell_xyup}, with the normalization 
\begin{equation}
	C_{1} = \frac{e_4(h_y-e_7e_8)}{h_y(e_4-e_7)},\qquad C_{2} = 1,
\end{equation}
and parameter evolution given by \eqref{nodal_act_points}.

\begin{remark}
	Note that this equation is equivalent, modulo M\"obius transformations, to Equation (2.16) of \cite{RWGCS05}. 
	Since equation (2.16) of \cite{RWGCS05} is given in the implicit form, instead of Proposition~\ref{prop_blow}
	we should use Proposition~\ref{prop_d} and apply it to the following decomposition of pull-backs into 
	deterministic classes:
	\begin{multline*}
	\varphi^{*}\left((\bar{H}_x+\bar{H}_y-\bar{E}_1-\bar{E}_3-\bar{E}_6-\bar{E}_8) +(\bar{E}_1)+(\bar{E}_3)+(\bar{E}_6)+(\bar{E}_8)\right) = \\
	(H_x+H_y-E_2-E_4-E_5-E_7) +(H_x-E_3)+(E_2)+(H_x-E_8)+(E_5)
	\end{multline*} and 
	\begin{multline*}
	\varphi^{*}\left((\bar{H}_x+\bar{H}_y-\bar{E}_2-\bar{E}_4-\bar{E}_5-\bar{E}_7) +(\bar{E}_2)+(\bar{E}_4)+(\bar{E}_5)+(\bar{E}_7)\right) = \\
	(H_x+H_y-E_1-E_3-E_6-E_8) +(H_x-E_4)+(E_1)+(H_x-E_7)+(E_6).
	\end{multline*}
	This kind of decompositions of the pull-back action on the Picard group are heuristic and appear only in cases lower than $E_7^{(1)}$.
\end{remark}

\subsubsection{Case $k=(a-a^{-1})^{2}$} 
\label{ssub:case_k_a_a_1_2}

This case is very similar, the only change is in the decomposition of the anticanonical
divisor, and the  resulting change in the parameterizations $\iota_{i}$. We have
$D_0=V(xy+ax+a^{-1}y-1)$, $D_1=V(xy+a^{-1}x+ay-1)$, their classes are
\begin{alignat*}{2}
{\mathcal D}_0&=\mathcal{H}_x+\mathcal{H}_y-\mathcal{E}_1-\mathcal{E}_4-\mathcal{E}_5-\mathcal{E}_8,
\qquad & &\ \mathcal{D}_{0} \bullet \mathcal{D}_{1} = 1, \\ 
{\mathcal D}_1&=\mathcal{H}_x+\mathcal{H}_y-\mathcal{E}_2-\mathcal{E}_3-\mathcal{E}_6-\mathcal{E}_7,
\qquad & 
&\begin{tikzpicture}[baseline=-0.2cm]
	\node (c1) at (0,0) [circle, thick, draw=black!100, inner sep=0pt,minimum size=1.3ex] {};
	\node (c2) at (1,0) [circle, thick, draw=black!100, inner sep=0pt,minimum size=1.3ex] {};
	\draw[thick, double distance = .4ex] (c1) -- (c2);
	\node[left] at (c1.west) {$\mathcal{D}_{0}$};
	\node[right] at (c2.east) {$\mathcal{D}_{1}$};
\end{tikzpicture},
\end{alignat*}   
%
%
and the parameterizations $\iota_{i}$ are now given by 
\begin{equation*}
	\begin{aligned}
		F_{0}(u) &= \frac{u - e_{1}}{u - e_{5}}, &\qquad G_{0}(u) &= \frac{u - e_{4}}{u - e_{8}},\\
		F_{1}(u) &= \frac{e_{1}}{e_{5}} \frac{u - h_{x}/e_{1}}{u - h_{x}/e_{5}}, 
		&\quad G_{1}(u) &= \frac{e_{4}}{e_{8}} \frac{u -h_{y}/e_{4}}{u - h_{y}/e_{8}}.		
	\end{aligned}
\end{equation*}
The deautonomization of the mapping $\varphi$ is given by~\eqref{eq:ell_xyup}, with the normalization 
\begin{equation}\label{eq:norm-A11}
	C_{1} = \frac{h_{y}-e_{7}e_{8}}{h_y - e_{3}e_{4}},\qquad C_{2} = 1,
\end{equation}
and parameter evolution given by \eqref{nodal_act_points}.



\subsubsection{Case $k= - (a - a^{-1})^{2}$} 
\label{ssub:case_k_ma_a_1_2}
In this case
$D_0=V(xy+ax+ay-a^2)$, $D_1= V(a^2xy+ax+ay-1)$,
\begin{alignat*}{2}
{\mathcal D}_0&=\mathcal{H}_x+\mathcal{H}_y-\mathcal{E}_1-\mathcal{E}_3-\mathcal{E}_5-\mathcal{E}_7,
\qquad & &\ \mathcal{D}_{0} \bullet \mathcal{D}_{1} = 1, \\ 
{\mathcal D}_1&=\mathcal{H}_x+\mathcal{H}_y-\mathcal{E}_2-\mathcal{E}_4-\mathcal{E}_6-\mathcal{E}_8\qquad & 
&\begin{tikzpicture}[baseline=-0.2cm]
	\node (c1) at (0,0) [circle, thick, draw=black!100, inner sep=0pt,minimum size=1.3ex] {};
	\node (c2) at (1,0) [circle, thick, draw=black!100, inner sep=0pt,minimum size=1.3ex] {};
	\draw[thick, double distance = .4ex] (c1) -- (c2);
	\node[left] at (c1.west) {$\mathcal{D}_{0}$};
	\node[right] at (c2.east) {$\mathcal{D}_{1}$};
\end{tikzpicture},
\end{alignat*}  
and the parameterizations change to
\begin{equation*}
	\begin{aligned}
		F_{0}(u) &= \frac{u - e_{1}}{u - e_{5}}, &\qquad G_{0}(u) &= \frac{e_{4}}{e_{8}}  \frac{u - h_{y}/e_{4}}{u - h_{y}/e_{8}},\\
		F_{1}(u) &= \frac{e_{1}}{e_{5}} \frac{u - h_{x}/e_{1}}{u - h_{x}/e_{5}}, 
		&\quad G_{1}(u) &= \frac{u -e_{4}}{u - e_{8}}.		
	\end{aligned}
\end{equation*}
The deautonomization of the mapping $\varphi$ is given by~\eqref{eq:ell_xyup}, normalization stays the same as
\eqref{eq:norm-A11},
and parameter evolution is again given by \eqref{nodal_act_points}.

%
%


\subsection{Singular fiber of type $A_{3}^{(1)}$}

This case is well-known and it corresponds to the famous $q$-difference equation q-$P_{\text{VI}}$ of Jimbo-Sakai \cite{JS96}.
It occurs when $k=\infty$, i.e., we just take 
the bi-quadratic curve $\Gamma_{\infty}$ defined by the matrix 
$\mathbf{A}$. When written in the homogeneous coordinates, 
$\Gamma_{\infty}$ is given by the equation $X_{0} X_{1} Y_{0} Y_{1} = 0$,
and so the anti-canonical divisor decomposes as $- \mathcal{K}_{X} = D_{0} + D_{1}+D_{2}+D_{3}$, where the prime divisors
$D_{i}$ are
\begin{alignat*}{2}
	D_{0} &= V(x - 0) &&= H_{x} - E_{1} - E_{2} \in 	\mathcal{H}_{x}-\mathcal{E}_{1}-\mathcal{E}_{2},\\
	D_{1} &= V(y - 0) &&= H_{y} - E_{3} - E_{4} \in 	\mathcal{H}_{y}-\mathcal{E}_{3}-\mathcal{E}_{4},\\
	D_{2} &= V(x - \infty) &&= H_{x} - E_{5} - E_{6} \in 	\mathcal{H}_{x}-\mathcal{E}_{5}-\mathcal{E}_{6},\\
	D_{3} &= V(y - \infty) &&= H_{y} - E_{7} - E_{8} \in 	\mathcal{H}_{y}-\mathcal{E}_{7}-\mathcal{E}_{8},
\end{alignat*}
and we get the family $\mathcal{X}_{\mathcal{A}}$ of type $A_{3}^{(1)}$ with the
symmetry group $W(D_5^{(1)})$.

The parameters of the mapping are coordinates $e_{i}$
of the blowup points $p_1(0,e_{1})$, $p_2(0,e_{2})$, $p_3(e_{3},0)$, $p_4(e_{4},0)$, 
$p_5(\infty, e_{5})$,  $p_6(\infty, e_{6})$,  $p_7(e_{7}, \infty)$, and $p_8(e_{8}, \infty)$. 
Note that in this case the degenerate structure of $\Gamma_{\infty}$ fixes 
the translational part of the M\"obius transformations action on the factors --- we require that 
$D_{0} = V(x - 0)$, $D_{1} = V(y - 0)$, $D_{2} = V(x - \infty)$, and $D_{3} = V(y - \infty)$, but the 
rescaling freedom on each factor still remains.

The deautonomized mapping can be obtained by direct application of Proposition~\ref{prop_blow} 
and equations \eqref{eq:varphi^*-map}. We have
\begin{align*}
	\varphi^{*}(V(\bar{x} - 0)) &= \pi_{*}\circ \tilde{\varphi}^{*}\circ \bar{\pi}^{*} (V(\bar{x} - 0)) \\
	&= \pi_{*} \circ \tilde{\varphi}^{*}((\bar{H}_{x} - \bar{E}_{1} - \bar{E}_{2}) + \bar{E}_{1} + \bar{E}_{2}) \\
	&= \pi_{*} (( H_{y} -  E_{7} - E_{8}) + (H_{x} - E_{3}) + (H_{x} - E_{4}))\\
	&= V(y - \infty) + V(x - e_{3}) + V(x - e_{4}),\\
	\varphi^{*}(V(\bar{x} - \infty)) &= \pi_{*}\circ \tilde{\varphi}^{*}\circ \bar{\pi}^{*} (V(\bar{x} - \infty))\\
	&= \pi_{*} \circ \tilde{\varphi}^{*}((\bar{H}_{x} - \bar{E}_{5} - \bar{E}_{6}) + \bar{E}_{5} + \bar{E}_{6}) \\
	&= \pi_{*} ((H_{y} - E_{3} - E_{4}) + (H_{x} - E_{7}) + (H_{x} - E_{8}))\\
	&= V(y - 0) + V(x - e_{7}) + V(x - e_{8}),\\
	\varphi^{*}(V(\bar{y} - 0)) &= \pi_{*}\circ \tilde{\varphi}^{*}\circ \bar{\pi}^{*} (V(\bar{y} - 0))\\
	&= \pi_{*} \circ \tilde{\varphi}^{*}((\bar{H}_{y} - \bar{E}_{3} - \bar{E}_{4}) + \bar{E}_{3} + \bar{E}_{4}) \\	
	&= \pi_{*} ((H_{x} - E_{1} - E_{2}) + E_{1} + E_{2}) = V(x-0),\\
	\varphi^{*}(V(\bar{y} - \infty)) &= \pi_{*}\circ \tilde{\varphi}^{*}\circ \bar{\pi}^{*} (V(\bar{y} - \infty))\\
	&= \pi_{*} \circ \tilde{\varphi}^{*}((\bar{H}_{y} - \bar{E}_{7} - \bar{E}_{8}) + \bar{E}_{7} + \bar{E}_{8}) \\	
	&= \pi_{*} ((H_{x} - E_{5} - E_{6}) + E_{5} + E_{6}) = V(x-\infty),
\end{align*}
and so the deautonomized mapping $\varphi$ is given by
\begin{equation}\label{a31_xyup}
	\left\{ 
	\begin{aligned}
		\bar{x} &= C_{1} \frac{(x-e_{3})(x-e_{4})}{y(x-e_{7})(x-e_{8})},  \\
		\bar{y} &= C_{2} x.	
	\end{aligned}
	\right.,
\end{equation}
where $C_{1}$ and $C_{2}$ are arbitrary constants. 

\begin{remark}
In the previous subsections, we needed to know the action of $\varphi$ on parameters before deautonomization. This is because 
 the characteristic curves themselves depend on those parameters and 
it is difficult to determine the coefficients $C_1$ and $C_2$ without that information,
while this problem does not arise in degenerated cases.   
\end{remark}

Evolution of the parameters can now be computed directly by looking at the coordinates of the blowup points. 
For example, from \eqref{eq:varphi^*-map} we see that
\begin{align*}
	\bar{p}_{1}(0,\bar{e}_{1}) &= \varphi(x = x_{3}) = \left(0,C_{2}e_{3}\right)\qquad\text{ or that}\\
	\bar{p}_{3}(\bar{e}_{3},0) &= \varphi(p_{2}) = \varphi(0,e_{2}) = \left(C_{1}\frac{e_{3} e_{4}}{e_{2}e_{7}e_{8}},0\right).
\end{align*}
Other points are computed in the similar fashion and we get
\begin{equation*}
(e_{1},\ldots,e_{8})\mapsto 
\left(
C_{2}e_{3}, C_{2}e_{4}, \frac{C_{1} e_{3} e_{4}}{e_{2} e_{7} e_{8}}, \frac{C_{1} e_{3}e_{4}}{e_{1}e_{7}e_{8}},
 C_{2} e_{7}, C_{2} e_{8}, \frac{C_{1}}{e_{6}}, \frac{C_{1}}{e_{5}}   
\right),
\end{equation*}
and so $q=e_{3}e_{4}e_{5}e_{6} (e_{1}e_{2}e_{7}e_{8})^{-1}$ is invariant under this mapping. It is convenient to choose
$C_{1} = e_{5} e_{6}$ and $C_{2} = 1$. Then the parameter evolution becomes 
\begin{align*}
(e_{1},e_{2},e_{3},e_{4},e_{5},e_{6},e_{7},e_{8})\mapsto (e_{3}, e_{4}, q e_{1}, q e_{2}, e_{7}, e_{8}, e_{5}, e_{6}). 
\end{align*}
\begin{figure}[ht]
	\centering
\begin{equation*}
	\raisebox{-40pt}{\begin{tikzpicture}[
			elt/.style={circle,draw=black!100,thick, inner sep=0pt,minimum size=2mm}]
		\path 	(-2,1.3) 	node 	(a0) [elt] {}
		        (-2,-1.3) 		node 	(a1) [elt] {}
		        ( -1,0) 	node  	(a2) [elt] {}
		        ( 1,0) 	node  	(a3) [elt] {}
		        ( 2,1.3) 		node 	(a4) [elt] {}
		        ( 2,-1.3)		node 	(a5) [elt] {};
		\draw [black,line width=1pt ] (a1) -- (a2) -- (a3) -- (a5)  (a3) -- (a4)  (a0) -- (a2); 
			\node at ($(a0.east) + (-0.5,0)$) 	{$\alpha_{0}$};
			\node at ($(a1.east) + (-0.5,0)$) 	{$\alpha_{1}$};
			\node at ($(a2.east) + (-0.5,0)$) 	{$\alpha_{2}$};
			\node at ($(a3.west) + (0.5,0)$) 	{$\alpha_{3}$};		
			\node at ($(a4.west) + (0.5,0)$) 	{$\alpha_{4}$};		
			\node at ($(a5.west) + (0.5,0)$) 	{$\alpha_{5}$};		
			
						\node[red,right] at (a0) {\small $\ 1$};
						\node[red,right] at (a1) {\small $\ 1$};
						\node[red,above] at (a2.east) {\small $\ 2$};
						\node[red,above] at (a3.west) {\small $\ 2$};
						\node[red,left] at (a4) {\small $\ 1$};
						\node[red,left] at (a5) {\small $\ 1$};
			
			\end{tikzpicture}
			} \hskip1in
			\begin{aligned}
			\alpha_{0} &= \mathcal{E}_{1} - \mathcal{E}_{2}, \\
			\alpha_{1} &= \mathcal{E}_{5} - \mathcal{E}_{6}, \\
			\alpha_{2} &= \mathcal{H}_{y} - \mathcal{E}_{1} - \mathcal{E}_{5}, \\
			\alpha_{3} &= \mathcal{H}_{x} - \mathcal{E}_{3} - \mathcal{E}_{7}, \\
			\alpha_{4} &= \mathcal{E}_{3} - \mathcal{E}_{4}, \\
			\alpha_{5} &= \mathcal{E}_{7} - \mathcal{E}_{8}.
			\end{aligned}
\end{equation*}
	\caption{Affine Dynkin diagram $D_{5}^{(1)}$.}
	\label{fig:dyn-d5}
\end{figure}

In this case the symmetry root sublattice $Q(D_{5}^{(1)}) = \operatorname{Span}_{\mathbb{Z}}\{\alpha_{0},\dots,\alpha_{5}\}$ is of 
type $D_{5}^{(1)}$ and we take the simple roots 
$\alpha_{i}$ as shown on Figure~\ref{fig:dyn-d5}, and so 
\begin{equation*}
\delta = - \mathcal{K}_{\mathcal{X}} = \mathcal{D}_{0} +  \mathcal{D}_{1} + 
 \mathcal{D}_{2} +  \mathcal{D}_{3} = \alpha_{0} + \alpha_{1} + 2 \alpha_{2} + 2 \alpha_{3} + \alpha_{4} + \alpha_{5}. 
\end{equation*}
The action of $\varphi^{*}$ on $Q(D_{5}^{(1)})$ is given by 
\begin{equation*}
(\varphi)^{*}:(\bar{\alpha}_{0},\bar{\alpha}_{1},\bar{\alpha}_{2},\bar{\alpha}_{3},\bar{\alpha}_{4},\bar{\alpha}_{5})\mapsto
(-\alpha_{4},-\alpha_{5},-\alpha_{3}, \delta - \alpha_{2}, -\alpha_{0}, -\alpha_{1}),
\end{equation*} 
which is not a translation, but $(\varphi^{*})^{2}$ is,
\begin{equation*}
(\varphi^{2})^{*}:(\bar{\alpha}_{0},\bar{\alpha}_{1},\bar{\alpha}_{2},\bar{\alpha}_{3},\bar{\alpha}_{4},\bar{\alpha}_{5})\mapsto
(\alpha_{0},\ldots,\alpha_{5}) + (0,0,-1,1,0,0)\delta.
\end{equation*}

This is exactly the $q$-difference equation q-$P_{\text{VI}}$ of Jimbo-Sakai \cite{JS96}, whose 
reduction form also appeared in \cite{PNGR92}:
\begin{align*}
\varphi^2:\left\{
\ba{rcl}
\bar{x}&=&\disp \frac{e_7e_8 (\bar{y}-e_1q)(\bar{y}-e_2q)}{x(\bar{y}-e_5)(\bar{y}-e_6)}\\
\bar{y}&=&\disp \frac{e_5e_6 (x-e_3)(x-e_4)}{y(x-e_7)(x-e_8)}
\ea\right.
\end{align*}
with the parameter evolution given by
\begin{align*}
(e_{1},e_{2},e_{3},e_{4},e_{5},e_{6},e_{7},e_{8})\mapsto
(q e_{1},q e_{2},q e_{3},q e_{4},e_{5},e_{6},e_{7},e_{8}).
\end{align*}

We also remark that the actions of $\varphi_{*}$ 
on $\operatorname{Pic}(\mathcal{X})$ 
can be written in the generators of $\tilde{W}(D_{5}^{(1)})$ as
\begin{align*}
\varphi_{*} &= \sigma\, w_{5} w_{4} w_{2}\, w_{1} w_{0} w_{2}\, w_{1} w_{0},
\end{align*}
where $\sigma$ is the automorphism of the $D_{5}^{(1)}$ Dynkin diagram that in the standard cycle notation can be written as
$\sigma = (\alpha_{0} \alpha_{5} \alpha_{1} \alpha_{4})(\alpha_{2} \alpha_{3})$.

\section{Elliptic difference systems with symmetry groups of type $D_{6}^{(1)}$ and
 $(A_{1}+A_{1}+A_{1})^{(1)}$}\label{sec:Special}

We now return to the mapping considered in Section~\ref{Sect_generic} that is obtained as a deautonomization of
our chosen QRT map and a choice of the generic elliptic fiber. This mapping $\varphi$ is given by 
\eqref{eq:ell_xyup} with the evolution of parameters given by \eqref{eq:phi-act_points}. Since
$\varphi$ does not induce a translation on the symmetry sublattice, $\varphi$ is not an elliptic difference system,
the parameters do not change translationally when we iterate the mapping. However, imposing certain constraints on the 
parameters makes it an elliptic difference system. Such constraints select a subfamily 
$\tilde{\mathcal{A}}\subset \mathcal{A}$ whose symmetry group does not appear explicitly in the Sakai classification scheme \cite{Sakai01}, since this is no longer a generic situation. 
In this section we consider two different examples of such constraints. 
We use the same notation as in Section~\ref{Sect_generic}.

\subsection{Restrictions with $D_{6}^{(1)}$ symmetry}\label{ssec:D6}
We know that $(\varphi^{4})^{*}$ acts on the symmetry sublattice $Q$ as a translation, see~\eqref{eq:f4-act}. 
In this section we investigate the constraints that we need to impose on the parameters
so that $(\varphi^{2})^{*}$ becomes a translation as well.

Via the embedding $\iota:\mathbb{T}\to \mathcal{X}$, \eqref{eq:f4-act} becomes an 
actual translation on $\mathbb{T}$,
\begin{equation*}
(\bar{a}_{0},...,\bar{a}_{8})=(a_{0},...,a_{8})+(0,2,-2,1,0,-1,0,1,0)d,
\end{equation*}
where $a_{i} = \iota^{*}(\alpha_{i})$, $d = \iota^{*}(\delta)$, and
$\bar{a}_{i} = \iota^{*}\circ(\varphi^{4})^{*}(\bar{\alpha}_{i})$.
Then, for $(\varphi^{2})^{*}$ to become a translation on $\mathbb{T}$, it has to be
\begin{align}\label{eq:half_trans}
(\bar{a}_{0},...,\bar{a}_{8})=
(a_{0},...,a_{8})+(0,2,-2,1,0,-1,0,1,0)\frac{d}{2}, 
\end{align}
where now $\bar{a}_{i} = \iota^{*}\circ(\varphi^{2})^{*}(\bar{\alpha}_{i})$.

Using \eqref{eq:varphi^*-map} we compute the action of $(\varphi^{2})^{*}$ on $\operatorname{Pic}(\mathbb{T})$
to be 
\begin{equation}\label{eq:act_points2}
\begin{aligned}
\bar{h}_x &= 5h_x+2h_y-e_{1256}-2e_{3478}, &\quad \bar{h}_y &= 2h_x+h_y-e_{3478},\\
\bar{e}_1 &= 2h_x+h_y-e_{23478},&\quad \bar{e}_2 &= 2h_x+h_y-e_{13478},\\ 
\bar{e}_3 &= h_x-e_4,&\quad  \bar{e}_4 &= h_x-e_3,\\
\bar{e}_5 &= 2h_x+h_y-e_{34678},&\quad \bar{e}_6 &= 2h_x+h_y-e_{34578},\\ 
\bar{e}_7 &= h_x-e_8,&\quad  \bar{e}_8&=h_x-e_7,
\end{aligned}
\end{equation}
where we have abbreviated $e_{i_1}+e_{i_2}+\dots+e_{i_s}$ as $e_{i_1i_2\dots i_s}$.
From \eqref{eq:a8roots} we get
\begin{equation}\label{eq:ais}
\begin{aligned}
a_0 &= e_1-e_2,  &  a_1 &= h_x-h_y,  &  a_2 &= h_y-e_1-e_2,  &  a_3 &= e_2-e_3,\\
a_4 &= e_3-e_4,  &  a_5 &= e_4-e_5,  &  a_6 &= e_5-e_6,  &      a_7 &= e_6-e_7,\\
a_8 &= e_7-e_8,  &  d &= 2h_x + \rlap{$2h_y  -e_{12345678},$}.
\end{aligned}	
\end{equation}
%
Using \eqref{eq:ais} and \eqref{eq:act_points2} in \eqref{eq:half_trans}
we see that, for $(\varphi^{2})^{*}$ to become a translation on $\mathbb{T}$,
the following constraints (on $\mathbb{T}$) must hold:
\begin{align}\label{eq:cond_phi2}
e_1+e_2-e_5-e_6=0, \quad e_3+e_4-e_7-e_8=0.
\end{align}

\begin{remark}
	It is important to note that condition \eqref{eq:cond_phi2} holds on $\mathbb{T}$  and not on 
		$\operatorname{Pic}(\mathcal{X})$, where similar expressions
		\begin{equation*}
			\mathcal{E}_1+\mathcal{E}_2-\mathcal{E}_5-\mathcal{E}_6=0,\quad  
			\mathcal{E}_3+\mathcal{E}_4-\mathcal{E}_7-\mathcal{E}_8=0
		\end{equation*}
		do not hold.
\end{remark}

Under constraints \eqref{eq:cond_phi2} $\varphi^2$ becomes
an elliptic difference equation.
Its explicit expression is obtained by iterating the mapping
\eqref{eq:ell_xyup} twice and noticing that, under constraints \eqref{eq:cond_phi2},
the constant $C_{1}$ in \eqref{eq:ell_xyup} and the mapping $\bar{\iota}$ given by 
\eqref{eq:FGup-phi} become
\begin{equation*}
C_{1} = G(e_{3}) = y_{3},\qquad \bar{\iota}(e_{i}) = (G(e_{i} - (h_{x} - e_{3} - e_{4})), F(e_{i})).
\end{equation*}
Using this and the parameter evolution under $\varphi$ given by \eqref{eq:phi-act_points}
we get the follownig explicit expression for $\varphi^2$,
\begin{equation}
\left\{	
\begin{aligned}
	\bar{x} &= x_{2} \frac{(\bar{y} - y_{2}')}{(\bar{y} - y_{6}')}\cdot 
	\frac{(y_{1}' - y_{6}') (\bar{y} - y_{5}') x - (y_{5}' - y_{6}') (\bar{y} - y_{1}') x_{6}}{
	(y_{1}' - y_{2}') (\bar{y} - y_{5}') x - (y_{5}' - y_{2}')(\bar{y} - y_{1}') x_{2}},\\[3pt]
	\bar{y} &= y_{3}
\frac{(x - x_{3})}{(x - x_{7})} \cdot
\frac{(x_{4} - x_{7}) (x - x_{8}) y - (x_{8} - x_{7})(x - x_{4}) y_{7}}{
(x_{4} - x_{3}) (x - x_{8}) y - (x_{8} - x_{3}) (x - x_{4}) y_{3}},
\end{aligned}	
\right.
\end{equation}
%
%
where $y_i'=G\left(e_{i}-(h_{x}-e_{3}-e_{4})\right)$.

Looking at the above expressions we observe that it is convenient to add to our dynamic a translation 
$\tau: \mathbb{T}\to \mathbb{T}$ given by $\tau(u) = \bar{u} = u - (h_{x} - e_{3} - e_{4})$ to get the 
following commutative diagram:
\begin{center}
\begin{tikzpicture}[>=stealth,baseline=0cm]
	\begin{scope}
	\node (T1)  at (0,0) {$\mathbb{T}$}; 	
	\node (T2)  at (2.5,0) {$\mathbb{T}$}; 	
	\node (X1)  at (0,1.5) {$\mathcal{X}$}; 	
	\node (X2)  at (2.5,1.5) {$\bar{\mathcal{X}}$}; 	
	\draw[->] (X1)--(X2) node [above,align=center,midway] {\scriptsize $\varphi^{2}$};
	\draw[->] (T1)--(T2) node [above,align=center,midway] {\scriptsize $\tau$};
	\draw[->] (T1)--(X1) node [left,align=center,midway] {\scriptsize $\iota$};
	\draw[->] (T2)--(X2) node [right,align=center,midway] {\scriptsize $\bar{\iota}$};
	\end{scope}
	\begin{scope}[xshift=5cm]
	\node (T1)  at (0,0) {$\operatorname{Pic}(\mathbb{T})$}; 	
	\node (T2)  at (3.5,0) {$\operatorname{Pic}(\mathbb{T})$}; 	
	\node (X1)  at (0,1.5) {$\operatorname{Pic}(\mathcal{X})$}; 	
	\node (X2)  at (3.5,1.5) {$\operatorname{Pic}(\bar{\mathcal{X}})$}; 	
	\draw[->] (X2)--(X1) node [above,align=center,midway] {\scriptsize $(\varphi^{2})^{*}$};
	\draw[->] (T2)--(T1) node [above,align=center,midway] {\scriptsize $\tau^{*}$};
	\draw[->] (X1)--(T1) node [left,align=center,midway] {\scriptsize $\iota^{*}$};
	\draw[->] (X2)--(T2) node [right,align=center,midway] {\scriptsize $\bar{\iota}^{\,*}$};
	\end{scope}
\end{tikzpicture}.
\end{center}
Then 
\begin{alignat*}{2}
i^{*}\circ (\varphi^{2})^{*} (\bar{\mathcal{E}}_{i}) &= \tau^{*}\circ \bar{\iota}^{*}(\mathcal{E}_{i}) & &= 
\bar{e}_{i} + (h_{x} - e_{3} - e_{4}),\\
i^{*}\circ (\varphi^{2})^{*} (\bar{\mathcal{H}}_{x}) &= \tau^{*}\circ \bar{\iota}^{*}(\mathcal{H}_{x}) & &= 
\bar{h}_{x} + 2(h_{x} - e_{3} - e_{4}),\\
i^{*}\circ (\varphi^{2})^{*} (\bar{\mathcal{H}}_{y}) &= \tau^{*}\circ \bar{\iota}^{*}(\mathcal{H}_{y}) & &= 
\bar{h}_{y} + 2(h_{x} - e_{3} - e_{4}),
\end{alignat*}
and so the new transformation of parameters is given by
\begin{align*}
&\bar{h}_x=h_x+d,\qquad \bar{h}_y=h_y,\\
&\bar{e}_1=e_1+d/2,\qquad \bar{e}_2=e_2+d/2,\qquad \bar{e}_3=e_3,\qquad \bar{e}_4=e_4,\\ 
&\bar{e}_5=e_5+d/2,\qquad \bar{e}_6=e_6+d/2,\qquad \bar{e}_7=e_7,\qquad \bar{e}_8=e_8.
\end{align*}
The curve $\bar{\mathcal{C}}$ is then parameterized by 
\begin{align}\label{eq:changeC}
\bar{F}(\bar{u})=\frac{[\bar{u}-\bar{e}_1][\bar{u}-\bar{h}_x+\bar{e}_1]}{[\bar{u}-\bar{e}_5][\bar{u}-\bar{h}_x+\bar{e}_5]}, \quad
\bar{G}(\bar{u})=\frac{[\bar{u}-\bar{e}_4][\bar{u}-\bar{h}_y+\bar{e}_4]}{[\bar{u}-\bar{e}_8][\bar{u}-\bar{h}_y+\bar{e}_8]}.
\end{align}
%
%

\begin{remark}
%
The translation $\tau$ on $\mathbb{T}$ 
is \emph{non-autonomous}, at each step the 
direction can change, and so 
$\bar{\bar{u}} = \bar{\tau}(\bar{u})$ should be computed as 
\begin{equation*}
\bar{\bar{u}}= \bar{\tau}(\bar{u}) = \bar{u}-(\bar{h}_x-\bar{e}_3-\bar{e}_4).
\end{equation*}
\end{remark}

\begin{theorem}
The subgroup of elements in $W(E_8^{(1)})$ preserving condition \eqref{eq:cond_phi2}
is isomorphic to the extended affine Weyl group $\widetilde{W}(D_6^{(1)})$. We can take
generators of ${W}(D_6^{(1)})$ to be the roots
\begin{alignat*}{2}
	\beta_{0} &= \mathcal{E}_{1} - \mathcal{E}_{2} &&= \alpha_{0},\\
	\beta_{1} &= \mathcal{E}_{5} - \mathcal{E}_{6} &&= \alpha_{6}, \\
	\beta_{2} &= \mathcal{H}_{y} - \mathcal{E}_{1} - \mathcal{E}_{5} &&= \alpha_{2} + \alpha_{3} + \alpha_{4} + \alpha_{5},\\
\beta_{3} &= \mathcal{H}_{x} - \mathcal{H}_{y} &&= \alpha_{1},\\
	\beta_{4} &= \mathcal{H}_{y} - \mathcal{E}_{3} - \mathcal{E}_{7}  &&= \alpha_{1} + 
	\alpha_{2} + 2\alpha_{3} + \alpha_{4} + \alpha_{5} + \alpha_{6} + \alpha_{7},\\
	\beta_{5} &= \mathcal{E}_{3} - \mathcal{E}_{4} &&= \alpha_{4},\\
	\beta_{6} &= \mathcal{E}_{7} - \mathcal{E}_{8} &&= \alpha_{8},
\end{alignat*}
whose configuration is described by the Dynkin diagram of type $D_{6}^{(1)}$,
\begin{center}
       \begin{tikzpicture}[
			elt/.style={circle,draw=black!100,thick, inner sep=0pt,minimum size=2mm}]
		\path 	(-2,1.3) 	node 	(g0) [elt] {}
		        (-2,-1.3) 	node 	(g1) [elt] {}
		        ( -1,0) 	node  	(g2) [elt] {}
		        (0.5,0) 	node  	(g3) [elt] {}
		        ( 2,0) 		node  	(g4) [elt] {}
		        ( 3,1.3) 	node 	(g5) [elt] {}
		        ( 3,-1.3)	node 	(g6) [elt] {};
		\draw [black,line width=1pt ] (g1) -- (g2) -- (g3) -- (g4) -- (g5)  (g0) -- (g2)  (g4) -- (g6); 
			\node at ($(g0.east) + (-0.5,0)$) 	{$\beta_{0}$};
			\node at ($(g1.east) + (-0.5,0)$) 	{$\beta_{1}$};
			\node at ($(g2.east) + (-0.6,0)$) 	{$\beta_{2}$};
			\node at ($(g3.south) + (0,-0.3)$) 	{$\beta_{3}$};		
			\node at ($(g4.west) + (0.6,0)$) 	{$\beta_{4}$};		
			\node at ($(g5.west) + (0.5,0)$) 	{$\beta_{5}$};		
			\node at ($(g6.west) + (0.5,0)$) 	{$\beta_{6}$};		

						\node[red,right] at (g0) {\small $\ 1$};
						\node[red,right] at (g1) {\small $\ 1$};
						\node[red,above] at (g2.east) {\small $\ 2$};
						\node[red,above] at (g3) {\small $\ 2$};
						\node[red,above] at (g4.west) {\small $\ 2$};
						\node[red,left] at (g5) {\small $\ 1$};
						\node[red,left] at (g6) {\small $\ 1$};

			\end{tikzpicture}.
\end{center}
The extended affine Weyl group is 
$\widetilde{W}(D_{6}^{(1)}) = \operatorname{Aut}(D_{6}^{(1)}) \ltimes W(D_{6}^{(1)})$, where 
the group of $D_{6}^{(1)}$ Dynkin diagram automorphism is isomorphic to the dihedral group 
$\mathbb{D}_{4}$,
\begin{equation*}
\operatorname{Aut}(D_{6}^{(1)}) \simeq \mathbb{D}_{4} = \langle \sigma_{1}, \sigma_{2} \mid 
\sigma_{1}^{4} = \sigma_{2}^{2} = (\sigma_{2}\sigma_{1})^{2} = e \rangle, 
\end{equation*}
with the generating automorphisms $\sigma_{1}$ and $\sigma_{2}$, when written in the standard cycle notation,
given by
\begin{alignat*}{2}
\sigma_{1} &= (\beta_{0}\beta_{5}\beta_{1}\beta_{6})(\beta_{2}\beta_{4}) & &= 
(\mathcal{E}_{1} \mathcal{E}_{3} \mathcal{E}_{5} \mathcal{E}_{7})
(\mathcal{E}_{2} \mathcal{E}_{4}\mathcal{E}_{6} \mathcal{E}_{8}),\\
\sigma_{2} &= (\beta_{0}\beta_{1})
& &= (\mathcal{E}_{1}\mathcal{E}_{5})(\mathcal{E}_{2}\mathcal{E}_{6})(\mathcal{E}_{3}\mathcal{E}_{7})
(\mathcal{E}_{4}\mathcal{E}_{8}).
\end{alignat*}
The induced action $\varphi^{*}$ is 
\begin{multline*}
\varphi^{*}:(\bar{\beta}_{0},
\ldots,\bar{\beta}_{6}) \mapsto
(-\beta_{5},-\beta_{6},-\beta_{3}-\beta_{4},\beta_{3} + 2 \beta_{4} + \beta_{5} + \beta_{6},\\  \beta_{0} + \beta_{1} + \beta_{2} + \beta_{3},
-\beta_{0},-\beta_{1}),
\end{multline*}	
which is not a translation on the sublattice $Q_{B}=\operatorname{Span}_{\mathbb Z}\{\beta_{0},\dots,\beta_{6}\}$, but 
$(\varphi^{2})^{*}$ is,
\begin{equation}\label{eq:d6-act}
(\varphi^{2})^{*}:(\bar{\beta}_{0},\ldots,\bar{\beta}_{6})\mapsto
(\beta_{0},\ldots,\beta_{6})+(0,0,-1,1,0,0,0)\delta,	
\end{equation}	
where $\delta$ now can be represented in terms of the simple roots $\beta_{i}$ as
\begin{equation*}
\delta=-\mathcal{K}_{\mathcal{X}}=\beta_{0} + \beta_{1} + 2 \beta_{2} + 2 \beta_{3} + 2 \beta_{4} + \beta_{5} + \beta_{6}.
\end{equation*}
The actions of $(\varphi)_{*}$ and $(\varphi^{2})_{*}$ on $\operatorname{Pic}(\mathcal{X})$ 
can be written in terms of the generators of $\widetilde{W}(D_{6}^{(1)})$ as
\begin{align*}
\varphi_{*} &= \sigma_{1}^{3}\, w_{\beta_6} w_{\beta_5} w_{\beta_3}\, w_{\beta_2} w_{\beta_1} w_{\beta_0}\, 
 w_{\beta_2} w_{\beta_1} w_{\beta_0},\\
(\varphi^{2})_{*} &= \sigma_{1}^{2}\, w_{\beta_3} w_{\beta_4} w_{\beta_6}\, w_{\beta_5} w_{\beta_4} w_{\beta_3}\, 
w_{\beta_2} w_{\beta_1} w_{\beta_0}\, w_{\beta_2},
\end{align*}
where $\sigma_{1}^{2}$ is the reflection symmetry of the Dynkin diagram
exchanging $\beta_{0}$ with $\beta_{1}$ and $\beta_{5}$ with $\beta_{6}$.
\end{theorem}

\begin{proof}
First, note that the generators $w_{\beta_{i}}$ and $\sigma_{j}$ of $\widetilde{W}(D_{6}^{(1)})$ act on the 
parameters as follows:	
\begin{equation}\label{D6_parameters} 
	\begin{aligned}
		w_{\beta_{0}}&:\quad & \bar{e}_{1} &= e_{2},\ & \bar{e}_{2} &= e_{1},\\
		w_{\beta_{1}}&:\quad & \bar{e}_{5} &= e_{6},\ & \bar{e}_{6} &= e_{5},\\
		w_{\beta_{2}}&:\quad & \bar{h}_{x} &= h_{x} + h_{y} - e_{1} - e_{5},\ & \bar{e}_{1} &= h_{y} - e_{5},\ & \bar{e}_{5} &= h_{y} - e_{1},\\
		w_{\beta_{3}}&:\quad & \bar{h}_{x} &= h_{y},\ & \bar{h}_{y} &= h_{x},\\
		w_{\beta_{4}}&:\quad & \bar{h}_{x} &= h_{x} + h_{y} - e_{3} - e_{7},\ & \bar{e}_{3} &= h_{y} - e_{7},\ & \bar{e}_{7} &= h_{y} - e_{3},\\
		w_{\beta_{5}}&:\quad & \bar{e}_{3} &= e_{4},\ & \bar{e}_{4} &= e_{3},\\
		w_{\beta_{6}}&:\quad & \bar{e}_{7} &= e_{8},\ & \bar{e}_{8} &= e_{7},\\
		\sigma_{1}&:    & \rlap{\hskip5pt$(\bar{e}_{1},\bar{e}_{2},\bar{e}_{3},\bar{e}_{4},\bar{e}_{5},\bar{e}_{6},\bar{e}_{7},\bar{e}_{8})
		=(e_{7},e_{8},e_{1},e_{2},e_{3},e_{4},e_{5},e_{6})$,}\\
		\sigma_{2}&:    & \rlap{\hskip5pt$(\bar{e}_{1},\bar{e}_{2},\bar{e}_{3},\bar{e}_{4},\bar{e}_{5},\bar{e}_{6},\bar{e}_{7},\bar{e}_{8})
		=(e_{5},e_{6},e_{7},e_{8},e_{1},e_{2},e_{3},e_{4})$,}\\
	\end{aligned}
\end{equation}
and so they preserve constraints \eqref{eq:cond_phi2}.
%
%
%
%
The decomposition of the actions of $(\varphi)_{*}$ and $(\varphi^{2})_{*}$ on $\operatorname{Pic}(\mathcal{X})$ 
is obtained in the standard way.

The main part of the proof is to show that $\widetilde{W}(D_6^{(1)})$-subgroup 
of $W(E_8^{(1)})$, described in the statement of the theorem,	 is the full stabilizer 
of constraints \eqref{eq:cond_phi2}, i.e., 
if $w^*:\operatorname{Pic}(\bar{\mathcal{X}})\to \operatorname{Pic}(\mathcal{X}) \in W(E_8^{(1)})$ 
under the canonical identification $\operatorname{Pic}(\bar{\mathcal{X}})\simeq \operatorname{Pic}(\mathcal{X})$ 
is such that, when pulled back to 
$\operatorname{Pic}(\mathbb{T})$, $w^{*}$ preserves \eqref{eq:cond_phi2}, 
then $w\in \widetilde{W}(D_6^{(1)})$. 

For that, in addition to the roots $\beta_{i}$, let us consider two more elements $\gamma_{1}$ and $\gamma_{2}$ in the full symmetry 
sublattice $Q = Q(E_{8}^{(1)})$ that are given by 
\begin{equation}\label{eq:gamma-alpha}
\begin{aligned}
	\gamma_{1} &= \mathcal{E}_1+\mathcal{E}_2-\mathcal{E}_5-\mathcal{E}_6 = \alpha_{0} + 2 (\alpha_{3} + \alpha_{4} + \alpha_{5}) + \alpha_{6},\\
	\gamma_{2} &= \mathcal{E}_3+\mathcal{E}_4-\mathcal{E}_7-\mathcal{E}_8 = \alpha_{4} + 2 (\alpha_{5} + \alpha_{6} + \alpha_{7}) + \alpha_{8}.
\end{aligned}	
\end{equation}
Then
\begin{equation}
	\begin{aligned}\label{eq:lattice-base-change}
		\alpha_{0} &= \beta_{0},\\
		\alpha_{1} &= \beta_{3},\\
		\alpha_{2} &= \beta_{2} + (\beta_{0} + \beta_{1} - \gamma_{1})/2, \\
		\alpha_{3} &= (-3 \beta_{0} - \beta_{1} - 2 \beta_{2} + 2 \beta_{4} - \beta_{5} + \beta_{6} + \gamma_{1} - \gamma_{2})/4,\\
		\alpha_{4} &= \beta_{5},\\
		\alpha_{5} &= (\beta_{0} - \beta_{1} + 2 \beta_{2} - 2 \beta_{4} - 3 \beta_{5} - \beta_{6} + \gamma_{1} + \gamma_{2})/4,\\
		\alpha_{6} &= \beta_{1},\\
		\alpha_{7} &= (-\beta_{0} - 3 \beta_{1} - 2 \beta_{2} + 2 \beta_{4} + \beta_{5} - \beta_{6} - \gamma_{1} + \gamma_{2})/4,\\
		\alpha_{8} &= \beta_{6}.
	\end{aligned}	
\end{equation}
Thus, $\beta_{i}$ and $\gamma_{j}$ do not form a lattice basis of $Q$, but they do form a basis of the associated vector space 
$Q^{\mathbb{Q}} = Q\otimes_{\mathbb{Z}} \mathbb{Q}$. Put $Q_{G} = \operatorname{Span}_{\mathbb{Z}}\{ \gamma_{1},\gamma_{2} \}$ and
let $\operatorname{Pic}^{\mathbb{Q}}(\mathcal{X})$, $Q^{\mathbb{Q}}_{B}$, $Q^{\mathbb{Q}}_{G}$ be the associated vector spaces.
Then $Q^{\mathbb{Q}} = Q^{\mathbb{Q}}_{B} \oplus Q^{\mathbb{Q}}_{G}$.

Consider now the commuting diagram
\begin{center}
\begin{tikzcd}[row sep=1.0cm,column sep=1.5cm]
   \operatorname{{\rm Pic}}(\mathcal{X})
     \arrow[transform canvas={yshift=.5ex}]{r}{w_*}
     \arrow[transform canvas={yshift=-.5ex},leftarrow]{r}[swap]{w^*}
     \arrow{d}[swap]{\iota^*}
 & \operatorname{{\rm Pic}}(\bar{\mathcal{X}})\simeq 
            \operatorname{{\rm Pic}}(\mathcal{X})
  \arrow{ld}{\bar{\iota}^*} \\
   \operatorname{{\rm Pic}}(\mathbb{T})  &  
\end{tikzcd}
\end{center}
and put $g_i=\iota^*(\gamma_i)$, $b_i=\iota^*(\beta_i)$,
 $\bar{g}_i=\bar{\iota}^*(\gamma_i)$, $\bar{b}_i=\bar{\iota}^*(\beta_i)$, then 
\begin{align*}
g_1&=e_1+e_2-e_5-e_6,\quad  g_2=e_3+e_4-e_7-e_8\\
\bar{g}_1&=\bar{e}_1+\bar{e}_2-\bar{e}_5-\bar{e}_6,\quad  
\bar{g}_2=\bar{e}_3+\bar{e}_4-\bar{e}_7-\bar{e}_8.
\end{align*}

We can write 
\begin{align*}
w^{*}(\gamma_{i}) &= k_{i0}\beta_{0} + \cdots + k_{i6}\beta_{6} + l_{i1} \gamma_{1} + l_{i2} \gamma_{2},\quad i=1,2.\\
\intertext{Pulling it back to $\operatorname{Pic}(\mathbb{T})$, we get the equations}
\bar{g}_{i} &= k_{i0} b_{0} + \cdots + k_{i6} b_{6} + l_{i1} g_{1} + l_{i2} g_{2}.
\end{align*}
Requiring that the pullback of $w^{*}$ on $\operatorname{Pic}(\mathbb{T})$ preserves \eqref{eq:cond_phi2}
means that $\bar{g}_{i} = 0$ whenever $g_{i} = 0$, or that $k_{i0} b_{0} + \cdots + k_{i6} b_{6} = 0$, which,
under the usual genericity assumption, means that $k_{i0} = \cdots = k_{i6} = 0$. Thus, 
$w^{*}(\mathbb{Q}^{\mathbb{Q}}_{G})\subset \mathbb{Q}^{\mathbb{Q}}_{G}$. However, since 
$w^{*}\in \mathbb{GL}(Q^{\mathbb{Q}})$, 
$w^{*}(\mathbb{Q}^{\mathbb{Q}}_{G})= \mathbb{Q}^{\mathbb{Q}}_{G}$ and hence
$w^{*}(\mathbb{Q}^{\mathbb{Q}}_{B})= \mathbb{Q}^{\mathbb{Q}}_{B}$. Further, since $w\in W(E_{8}^{(1)})$,
$w^{*}(Q) = Q$. Since $Q_{G} = Q^{\mathbb{Q}}_{G}\cap Q$, $Q_{B} = Q^{\mathbb{Q}}_{B}\cap Q$, we get that 
$w^{*}({Q}_{G})= {Q}_{G}$ and $w^{*}({Q}_{B})= {Q}_{B}$.

According to \S 5.10 of \cite{Kac}, if $w^{*}(Q_{B}) = Q_{B}$, since $Q_{B}$ is a root lattice, 
the action of $w^{*}$ on $Q_{B}$ coincides with the action of 
$w'\in \pm \widetilde{W}(D_{6}^{(1)})$. But, since $w\in W(E_{8}^{(1)})$, 
$w^{*}(-\mathcal{K}_{\mathcal{X}}) = -\mathcal{K}_{\mathcal{X}}$, which implies that  
$w'\in \widetilde{W}(D_{6}^{(1)})$. To complete the proof it remains to show that the action of 
$w^{*}$ on $Q$ coincides with the extension of the action of $(w')^{*}$ on $Q$ or, equivalently, on
$Q_{G}$. Note that the generators of $\widetilde{W}(D_{6}^{(1)})$ act on $\gamma_{i}$ as follows:
\begin{equation*}
w_{\beta_{i}}^{*}(\gamma_{j}) = \gamma_{j},\quad \sigma_{1}(\gamma_{1}) = \gamma_{2},\quad
\sigma_{1}(\gamma_{2}) = -\gamma_{1},\quad \sigma_{2}(\gamma_{i}) = -\gamma_{i}.
\end{equation*} 
Let $w_{0} = w^{-1}\circ w'$, we want to show that $w_{0}^{*}$ acts as an identity on $Q$. We do know that
$w_{0}^{*}$ acts on $Q$ and also that it acts as an identity on $Q_{B}$, so it remains to show that 
$w_{0}^{*}$ acts as an identity of $Q_{G}$ as well.

Since $w_{0}^{*}(Q_{G}) = Q_{G}$, $w_{0}^{*}(\gamma_{i}) = l_{i1}\gamma_{1} + l_{i2}\gamma_{2}$, where $l_{ij}\in \mathbb{Z}$.
We know that $w^{*}$ preserves the intersection pairing, $\gamma_{i}\bullet \gamma_{j} = - 4 \delta_{ij}$, which in turn
implies that 
\begin{equation*}
l_{11}^2+l_{12}^2=1, \quad l_{21}^2+l_{22}^2=1, \quad
l_{11}l_{21}+l_{12}l_{22}=0,
\end{equation*}
and so $w_{0}^{*}(\gamma_{1},\gamma_{2}) = (\varepsilon_{1} \gamma_{i},\varepsilon_{2}\gamma_{j})$, where 
$\varepsilon_{1} = \pm1$, $\varepsilon_{2} = \pm1$, $i,j = 1,2$ and $i\neq j$. However, $w_{0}^{*}(\alpha_{i})\in Q$
together with~\eqref{eq:lattice-base-change} requires that 
\begin{equation}\label{eq:cond-gamma}
	w_{0}^{*}(\gamma_{1}) - \gamma_{1}\in 2Q,\qquad 
	(w_{0}^{*}(\gamma_{1}) - \gamma_{1}) - (w_{0}^{*}(\gamma_{2}) - \gamma_{2})\in 4Q.
\end{equation}
Indeed, we can write
\begin{align*}
	w_{0}^{*}(\alpha_{2}) &= \alpha_{2} + (\gamma_{1} - w^{*}(\gamma_{1}))/2,\\
	w_{0}^{*}(\alpha_{3}) &= \alpha_{3} + ((w_{0}^{*}(\gamma_{1}) - \gamma_{1}) - (w_{0}^{*}(\gamma_{2}) - \gamma_{2}))/4.
\end{align*}
Combining~\eqref{eq:cond-gamma} and~\eqref{eq:gamma-alpha} and observing that $\gamma_{1}\pm \gamma_{2}\notin 2Q$, 
we see that the only possibility is 
$w^{*}(\gamma_{1}) = \gamma_{1}$ and $w^{*}(\gamma_{2}) = \gamma_{2}$, which completes the proof of the theorem.
\end{proof}

We conclude this section by giving an explicit description of a birational representation of $\widetilde{W}(D_{6}^{(1)})$. 
For this it is convenient to use a different normalization $x_{1} = y_{1} = \infty$, $x_{2} = y_{2} = 0$, instead of our usual
normalization~\eqref{eq:norm-pts}. Let
\begin{equation*}
F_{ij}(u)=\frac{[u-e_j][u+e_j-h_x]}{[u-e_i][u+e_i-h_x]}, 
\qquad G_{ij}(u)=\frac{[u-e_j][u+e_j-h_y]}{[u-e_i][u+e_i-h_y]}.
\end{equation*}
The embedding $\iota: \mathbb{T}\to \mathbb{P}^{1}\times \mathbb{P}^{1}$ for this normalization is then given by 
\begin{equation}\label{eq:FG12}
	\iota(u) = (F_{12}(u),G_{12}(u)).
\end{equation}
Changing normalization back to~\eqref{eq:norm-pts} with the embedding
of $\mathbb{T}$ given by~\eqref{eq:FG}, $\iota(u) = (F_{51}(u),G_{84}(u))$, is explicitly given by 
the coordinate transformation $\nu: (x,y)\mapsto (\tilde{x},\tilde{y})$, where
%
\begin{alignat*}{3}
&\nu:\qquad & \tilde{x}&=\frac{F_{52}(e_1)}{x-F_{12}(e_5)},\quad & 
\tilde{y}&=G_{84}(e_1)\,\frac{y-G_{12}(e_4)}{y-G_{12}(e_8)}, \\
&\nu^{-1}:\qquad & x&=F_{12}(e_5)\,\frac{\tilde{x}-F_{51}(e_2)}{\tilde{x}},\quad &  
y&=G_{12}(e_8)\,\frac{\tilde{y}-G_{84}(e_2)}{\tilde{y}-G_{84}(e_1)}.
\end{alignat*}

Using Proposition~\ref{prop_blow} then immediately gives us the following result.

\begin{corollary}\label{cor:bira-D6}
	A birational representation of the extended affine Weyl group $\widetilde{W}(D_{6}^{(1)})$ is given by the following maps:
	\begin{align*}
	w_{\beta_i}:\quad& \bar{x}=x,\ \bar{y}=y\quad (i=1,5,6);\\[3pt] 
	w_{\beta_0}:\quad& \bar{x}=\frac{1}{x},\ \bar{y}=\frac{1}{y};\\[3pt]
	w_{\beta_2}:\quad&  \bar{x}-\bar{F}_{12}(\bar{e}_5)=-\frac{\bar{F}_{12}(\bar{e}_{5}) G_{12}(e_{5})}{F_{12}(e_{5})}\cdot
	 \frac{x-F_{12}(e_5)}{y-G_{12}(e_5)},\ \bar{y}=\frac{G_{52}(e_1)y}{(y-G_{12}(e_5))};\\[3pt]
	w_{\beta_3}:\quad& \bar{x}=y,\ \bar{y}=x;\\[3pt]
	w_{\beta_4}:\quad& \frac{\bar{x}-\bar{F}_{12}(\bar{e}_3)}{\bar{x}-\bar{F}_{12}(\bar{e}_7)}
	=\frac{(x-F_{12}(e_3))(y-G_{12}(e_7))}{(x-F_{12}(e_7))(y-G_{12}(e_3))},\ \bar{y} = y; \\[3pt]
	\sigma_{1}:\quad&
	 \bar{x}=F_{78}(e_1)\,\frac{x-F_{12}(e_8)}{x-F_{12}(e_7)},\ 
	\bar{y}=G_{78}(e_1)\,\frac{y-G_{12}(e_8)}{y-G_{12}(e_7)};\\[3pt]
	\sigma_{2}:\quad&
	 \bar{x}=F_{56}(e_1)\frac{x-F_{12}(e_6)}{x-F_{12}(e_5)},\ 
	\bar{y}=G_{56}(e_1)\frac{y-G_{12}(e_6)}{y-G_{12}(e_5)},
	\end{align*}
with the transformation of parameters \eqref{D6_parameters}, where
\begin{equation*}
\bar{F}_{12}(u)=\frac{[u-\bar{e}_2][u-\bar{h}_x+\bar{e}_2]}{[u-\bar{e}_1][u-\bar{h}_x+\bar{e}_1]},\qquad \bar{G}_{12}(u)=\frac{[u-\bar{e}_2][u-\bar{h}_y+\bar{e}_2]}{[u-\bar{e}_1][u-\bar{h}_y+\bar{e}_1]}.
\end{equation*}
\end{corollary}
\begin{proof} As yet another illustration of Proposition~\ref{prop_blow}, we show how to obtain the formula for $\bar{x}$ for the
$w_{\beta_{4}}$ mapping. On the level of the Picard lattice, we have
\begin{equation*}
w_{\beta_{4}}^{*}(\bar{\mathcal{H}}_{x}) = \mathcal{H}_{x} + \mathcal{H}_{y} - \mathcal{E}_{3} - \mathcal{E}_{7},\quad
w_{\beta_{4}}^{*}(\bar{\mathcal{E}}_{3}) = \mathcal{H}_{y} - \mathcal{E}_{7},\quad
w_{\beta_{4}}^{*}(\bar{\mathcal{E}}_{7}) = \mathcal{H}_{y} - \mathcal{E}_{3}.
\end{equation*}
Using $\bar{D}_{1} = V(\bar{x} - \bar{x}_{3})$ and $\bar{D}_{2} = V(\bar{x} - \bar{x}_{7})$ in 
Proposition~\ref{prop_blow}, we get
\begin{align*}
w_{\beta_{4}}^{*}(D_{1}) &= \pi_{*}\circ \tilde{w}_{\beta_{4}}^{*} ((\bar{H}_{x} - \bar{E}_{3}) + \bar{E}_{3}) \\
&= \pi_{*} ((H_{x} - E_{3}) + (H_{y} - E_{7})) = V((x-x_{3})(y-y_{7})),\\
w_{\beta_{4}}^{*}(D_{2}) &= \pi_{*}\circ \tilde{w}_{\beta_{4}}^{*} ((\bar{H}_{x} - \bar{E}_{7}) + \bar{E}_{7}) \\
&= \pi_{*} ((H_{x} - E_{7}) + (H_{y} - E_{3})) = V((x-x_{7})(y-y_{3})).
\end{align*}
Thus,
\begin{align*}
\frac{\bar{x} - \bar{x}_{3}}{\bar{x} - \bar{x}_{7}} &= C \frac{(x - x_{3})(y - y_{7})}{(x - x_{7})(y - y_{3})},
\end{align*}
where we can evaluate the normalization constant $C$ using the pull-back 
$w_{\beta_{4}}^{*}(\bar{E}_2) = E_2$ as
\begin{align*}
\frac{\bar{x}_{3}}{\bar{x}_{7}} &= C \frac{x_{3} y_{7}}{x_{7} y_{3}}\quad\implies\quad
C = \frac{\bar{F}_{12}(\bar{e}_{3}) F_{12}(e_{7}) G_{12}(e_{3})}{\bar{F}_{12}(\bar{e}_{7}) F_{12}(e_{3}) G_{12}(e_{7})}  = 1,
\end{align*}
which gives the desired result. Other formulas are established in the same way.	
\end{proof}

\subsection{Restrictions with $(A_1+A_1+A_1)^{(1)}$ symmetry}

Along the same lines, we can impose more constraints on the parameters to make 
$\varphi$ itself an elliptic difference system. 
In this case we should have
\begin{align}\label{fourth_trans}
(\bar{a}_0,...,\bar{a}_8) = (a_0,...,a_8) + (0,2,-2,1,0,-1,0,1,0)\frac{d}{4},
\end{align}
where $\bar{a}_{i} = \iota^{*} \circ \varphi^{*}(\bar{\alpha}_{i})$.
This condition, together with \eqref{eq:ais} and \eqref{eq:phi-act_points}, 
results in following constraints on $\mathbb{T}$:
\begin{equation}\label{cond_phi}
	\begin{aligned}
		e_1+e_2-e_5-e_6&=0, &\quad e_3+e_4-e_7-e_8&=0,\\
		e_1+e_3-e_2-e_4&=0, &\quad e_5+e_7-e_6-e_8&=0,\\
		\rlap{\hskip-25pt$(h_x-e_1-e_2)-(h_y-e_3-e_4)=0$.} 
	\end{aligned}
\end{equation}
Note that these constraints also imply that 
\begin{equation*}
e_{1}+e_{3} = e_{5} + e_{7},\qquad e_{3} - e_{2} = e_{4} - e_{1} = e_{7} - e_{6} = e_{8} - e_{5} = (h_{y} - h_{x})/2. 
\end{equation*}
Subject to constraints \eqref{cond_phi}, $\varphi$ becomes 
an elliptic difference system \eqref{eq:ell_xyup} with
$C_1=x_2$ and $C_2=1$.

Similar to Section~\ref{ssec:D6}, if we add to our dynamic a (non-autonomous) translation 
$\tau: \mathbb{T}\to \mathbb{T}$ given by $\tau(u) = \bar{u} = u - (e_{3} - e_{2})$,
the transformation of the parameters takes a simple form
\begin{align*}
&\bar{h}_x=h_x+d/2,\qquad \bar{h}_y=h_y,\\
&\bar{e}_1=e_1+d/4,\qquad \bar{e}_2=e_2+d/4,\qquad \bar{e}_3=e_3,\qquad \bar{e}_4=e_4,\\ 
&\bar{e}_5=e_5+d/4,\qquad \bar{e}_6=e_6+d/4,\qquad \bar{e}_7=e_7,\qquad \bar{e}_8=e_8.
\end{align*}
This changes the embedding of the curve $\bar{\mathcal{C}}$, which is now then parameterized by 
\begin{align*}
\bar{F}(\bar{u})=\frac{[\bar{u}-\bar{e}_1][\bar{u}-\bar{h}_x+\bar{e}_1]}{[\bar{u}-\bar{e}_5][\bar{u}-\bar{h}_x+\bar{e}_5]}, \quad
\bar{G}(\bar{u})=\frac{[\bar{u}-\bar{e}_4][\bar{u}-\bar{h}_y+\bar{e}_4]}{[\bar{u}-\bar{e}_8][\bar{u}-\bar{h}_y+\bar{e}_8]}.
\end{align*}

Consider now the roots $\beta_{0}^{i}$, $i=1,\dots,3$, given by
\begin{alignat*}{2}
	\beta_{0}^{1} &= 2 \mathcal{H}_{y} - \mathcal{E}_{1}- \mathcal{E}_{2}- \mathcal{E}_{5}- \mathcal{E}_{6} 
	&&= \alpha_{0} + 2(\alpha_{2} + \alpha_{3} + \alpha_{4} + \alpha_{5}) + \alpha_{6}, \\
	\beta_{0}^{2} &= \mathcal{E}_{1}- \mathcal{E}_{2}- \mathcal{E}_{3}+ \mathcal{E}_{4} &&= \alpha_{0} - \alpha_{4}, \\
	\beta_{0}^{3} &= \mathcal{E}_{5}- \mathcal{E}_{6}- \mathcal{E}_{7} + \mathcal{E}_{8} &&= \alpha_{6} - \alpha_{8}, 
\end{alignat*}
as well as the complemetary roots $\beta_{1}^{i}  = (- \mathcal{K}_{\mathcal{X}}) - \beta_{0}^{i}$,
\begin{align*}
		\beta_{1}^{1} &= 2 \mathcal{H}_{x} - \mathcal{E}_{3}- \mathcal{E}_{4}- \mathcal{E}_{7}- \mathcal{E}_{8}, \\
		\beta_{1}^{2} &= 2 \mathcal{H}_{x} + 2 \mathcal{H}_{y} - 2\mathcal{E}_{1} - 2 \mathcal{E}_{4} -  
		\mathcal{E}_{5}- \mathcal{E}_{6}- \mathcal{E}_{7}- \mathcal{E}_{8}, \\
		\beta_{1}^{3} &= 2 \mathcal{H}_{x} + 2 \mathcal{H}_{y} - \mathcal{E}_{1} -  \mathcal{E}_{2} -  
		\mathcal{E}_{3}- \mathcal{E}_{4}- 2\mathcal{E}_{5 }- 2\mathcal{E}_{8}.
\end{align*}
Each pair $(\beta_{0}^{i},\beta_{1}^{i})$ generates an $A_{1}^{(1)}$-sublattice 
$Q_{B_{i}} = \operatorname{Span}_{\mathbb{Z}}\{(\beta_{0}^{i},\beta_{1}^{i})\}$ of the lattice $Q = Q(E_{8}^{(1)})$. 
Put $Q_{B} = \operatorname{Span}_{\mathbb{Z}}\{\beta_{0}^{1},\beta_{1}^{1},\beta_{0}^{2},\beta_{1}^{2},\beta_{0}^{3},\beta_{1}^{3}\}$; 
we can choose $\beta_{0}^{1}$, $\beta_{1}^{1}$, $\beta_{0}^{2}$, and $\beta_{0}^{3}$ as a basis of $Q_{B}$.
Next, consider  additional roots corresponding to the constraints,
\begin{alignat*}{2}
	\gamma_{1} &= \mathcal{E}_{1} + \mathcal{E}_{2} - \mathcal{E}_{5} - \mathcal{E}_{6} &&= 
	\alpha_{0} + 2(\alpha_{3} + \alpha_{4} + \alpha_{5}) + \alpha_{6}, \\
	\gamma_{2} &= \mathcal{E}_{3} + \mathcal{E}_{4} - \mathcal{E}_{7} - \mathcal{E}_{8} &&= 
	\alpha_{4} + 2(\alpha_{5} + \alpha_{6} + \alpha_{7}) + \alpha_{8}, \\
	\gamma_{3} &= \mathcal{E}_{1} + \mathcal{E}_{3} - \mathcal{E}_{2} - \mathcal{E}_{4} &&= 
	\alpha_{0} +  \alpha_{4}, \\
	\gamma_{4} &= \mathcal{E}_{5} + \mathcal{E}_{7} - \mathcal{E}_{6} - \mathcal{E}_{8} &&= 
	\alpha_{6} +  \alpha_{8}, \\
	\gamma_{5} &= (\mathcal{H}_{x} - \mathcal{E}_{1} - \mathcal{E}_{2}) - (\mathcal{H}_{y} - \mathcal{E}_{3} - \mathcal{E}_{4}) 
	&&= - \alpha_{0} + \alpha_{1} - 2 \alpha_{3} - \alpha_{4}, 
\end{alignat*}
and put $Q_{G} = \operatorname{Span}_{\mathbb{Z}}\{\gamma_{1},\gamma_{2},\gamma_{3},\gamma_{4},\gamma_{5}\}$. 
Then $Q_{B} = (Q_{G})^{\perp}$.

Since the roots $\beta_{j}^{i}$ are not simple, $(\beta_{j}^{i})^{2} = -4$, they do not act on $\operatorname{Pic}(\mathcal{X})$. 
A possible workaround is to represent $\beta_{j}^{i} = \beta_{j1}^{i} + \beta_{j2}^{i}$, where 
$(\beta_{jk}^{i})^{2} = -2$ and $\beta_{j1}^{i} \bullet \beta_{j2}^{i} = 0$. Then
we can define the action of $w^{*}_{\beta_{j}^{i}}$ on $\operatorname{Pic}(\mathcal{X})$
as a composition of two commuting reflections, 
$w_{\beta_{i}^{j}}^{*}=w_{\beta_{i1}^j}^*\circ w_{\beta_{i2}^j}^*$, 
\begin{equation*}
 w^{*}_{\beta_{j}^{i}}(\bar{\mathcal{D}}) = \mathcal{D} + (\mathcal{D}\bullet \beta_{j1}^{i})\beta_{j1}^{i} 
 + (\mathcal{D}\bullet \beta_{j2}^{i})\beta_{j2}^{i}.
\end{equation*}
Explicitly, we can put
\begin{alignat*}{2}
	w_{\beta_0^1} &= w_{\mathcal{H}_y-\mathcal{E}_1-\mathcal{E}_2}\circ w_{\mathcal{H}_y-\mathcal{E}_5-\mathcal{E}_6}\\
       	w_{\beta_1^1} &= w_{\mathcal{H}_x-\mathcal{E}_3-\mathcal{E}_4}\circ w_{\mathcal{H}_x-\mathcal{E}_7-\mathcal{E}_8}\\
	w_{\beta_{0}^2} &= w_{\mathcal{E}_1-\mathcal{E}_2}\circ w_{\mathcal{E}_4-\mathcal{E}_3} \\
	w_{\beta_{1}^2} &= w_{\mathcal{H}_x+\mathcal{H}_y-\mathcal{E}_1-\mathcal{E}_4-\mathcal{E}_5-\mathcal{E}_7} 
 	                \circ w_{\mathcal{H}_x+\mathcal{H}_y-\mathcal{E}_1-\mathcal{E}_4-\mathcal{E}_6-\mathcal{E}_8} \\
	w_{\beta_{0}^3} &= w_{\mathcal{E}_5-\mathcal{E}_6}\circ w_{\mathcal{E}_8-\mathcal{E}_7} \\
	w_{\beta_{1}^3} &= w_{\mathcal{H}_x+\mathcal{H}_y-\mathcal{E}_1-\mathcal{E}_3-\mathcal{E}_5-\mathcal{E}_8} 
			\circ w_{\mathcal{H}_x+\mathcal{H}_y-\mathcal{E}_2-\mathcal{E}_4-\mathcal{E}_5-\mathcal{E}_8}.
\end{alignat*}
On $Q_B$ this action can be written in the usual way $w_{\beta_{i}}^*(\beta_j) = \beta_j - c_{ij} \beta_i$,
where $c_{ij} = 2(\beta_i \bullet \beta_j)/(\beta_i \bullet \beta_i)$, e.g.,
\begin{equation*}
w_{\beta_{0}^{1}}^{*}: (\beta_{0}^{1}, \beta_{1}^{1} ; \beta_{0}^{2}, \beta_{1}^{2}; \beta_{0}^{3}, \beta_{1}^{3}) \mapsto
(-\beta_{0}^{1}, \beta_{1}^{1} + 2 \beta_{0}^{1} ; \beta_{0}^{2}, \beta_{1}^{2}; \beta_{0}^{3}, \beta_{1}^{3}).
\end{equation*}

Then $w^{*}_{\beta_{j}^{i}}$, together with the automorphism $\sigma$ interchanging $\beta_{0}^{i}$ and 
$\beta_{1}^{i}$,
\begin{equation*}
\sigma = (\mathcal{H}_{x}\mathcal{H}_{y})(\mathcal{E}_{1}\mathcal{E}_{4})(\mathcal{E}_{2}\mathcal{E}_{3})
(\mathcal{E}_{5}\mathcal{E}_{8})(\mathcal{E}_{6}\mathcal{E}_{7}) = 
(\beta_{0}^{1}\beta_{1}^{1})(\beta_{0}^{2}\beta_{1}^{2})(\beta_{0}^{3}\beta_{1}^{3}),
\end{equation*}
that also acts on the lattice $Q_{G}$ as $\sigma(\gamma_{1})  = \gamma_{2}$, $\sigma(\gamma_{2})  = \gamma_{1}$,
$\sigma(\gamma_{3})  = -\gamma_{3}$, $\sigma(\gamma_{4})  = -\gamma_{4}$, and $\sigma(\gamma_{5})  = -\gamma_{5}$,
generate an extended affine Weyl group of type $(A_{1} + A_{1} + A_{1})^{(1)}$ that is a symmetry group of our equation,
\begin{align*}
\widetilde{W}((A_{1} + A_{1} + A_{1})^{(1)}) &= \langle w^{*}_{\beta_{0}^{i}}, w^{*}_{\beta_{1}^{i}}, \sigma \mid 
\sigma w^{*}_{\beta_{0}^{i}} = w^{*}_{\beta_{1}^{i}} \sigma, \\
&\qquad\qquad (w^{*}_{\beta_{0}^{i}})^{2} = (w^{*}_{\beta_{1}^{i}})^{2} = (w^{*}_{\beta_{0}^{i}}w^{*}_{\beta_{1}^{i}})^{\infty} = \sigma^{2}= e
\rangle.
\end{align*}
Note that $\varphi^{*}$ acts on $Q_{B}$ as
\begin{equation*}
\varphi^{*}: (\beta_{0}^{1}, \beta_{1}^{1} ; \beta_{0}^{2}, \beta_{1}^{2}; \beta_{0}^{3}, \beta_{1}^{3}) \mapsto
(\beta_{0}^{1} - \delta, \beta_{1}^{1} + \delta  ; \beta_{0}^{2}, \beta_{1}^{2}; \beta_{0}^{3}, \beta_{1}^{3}),
\end{equation*}
and so $\varphi^{*} = \sigma \circ w^{*}_{\beta_{0}^{1}}$.

\subsection*{Acknowledgement}
The authors are very grateful to the Tsinghua Sanya International Mathematics Forum where all of them were invited to a workshop on discrete integrable dystems in April, 2016. A.~D. was partially suported by the University of Northern Colorado 2016 Summer Support Initiative.
T.~T. was supported by the Japan Society for the
Promotion of Science, Grand-in-Aid (C) (25400106).

\appendix

\section{Divisor map computation for a rational mapping}\label{app:rational}
In this section we explain how to compute the mapping $\tilde{\varphi}^{*}$ in 
Example~\ref{ex:rational}. Recall that $\mathcal{X}_{12} = \bar{\mathcal{X}}_{12}$ is a 
blowup of $\mathbb{P}^{1} \times \mathbb{P}^{1}$ at the points
$A_{1}(0,0)$ and $A_{2}(\infty,\infty)$. 
Let us first introduce local affine coordinate charts  
on $\mathbb{P}^{1} \times \mathbb{P}^{1}$ and $\mathcal{X}_{12}$. We have the standard four charts on 
$\mathbb{P}^{1} \times \mathbb{P}^{1}$ with affine coordinates $(x,y)$, $(X,y)$, $(X,Y)$, $(x,Y)$, and the usual
transition functions $x X = 1$ and $y Y = 1$. Introducing the blowup coordinates
\begin{align*}
x &= u_{1} = U_{1}V_{1},\! & y &= u_{1}v_{1} = V_{1}, \! &  u_{1}&=x, & v_{1}&=y/x,\! & U_{1} &= x/y,\! & V_{1}&= y,\\
X &= u_{2} = U_{2}V_{2},\! & Y &= u_{2} v_{2} = V_{2},\! & u_{2}&= X, & v_{2}&= Y/X,\! & U_{2}&= X/Y,\! & V_{2}&=Y,
\end{align*}
we get the following six charts on $\mathcal{X}_{12}$, also shown on the picture below,
\begin{align*}
	\mathfrak{U}_{1} &=\{(U_{1},V_{1})\}, & 
	\mathfrak{U}_{2} &=\{(u_{1},v_{1})\}, & 
	\mathfrak{U}_{3} &=\{(X,y)\}, \\
	\mathfrak{U}_{4} &=\{(U_{2},V_{2})\}, & 
	\mathfrak{U}_{5} &=\{(u_{2},v_{2})\}, & 		
	\mathfrak{U}_{6} &=\{(x,Y)\};	
\end{align*}

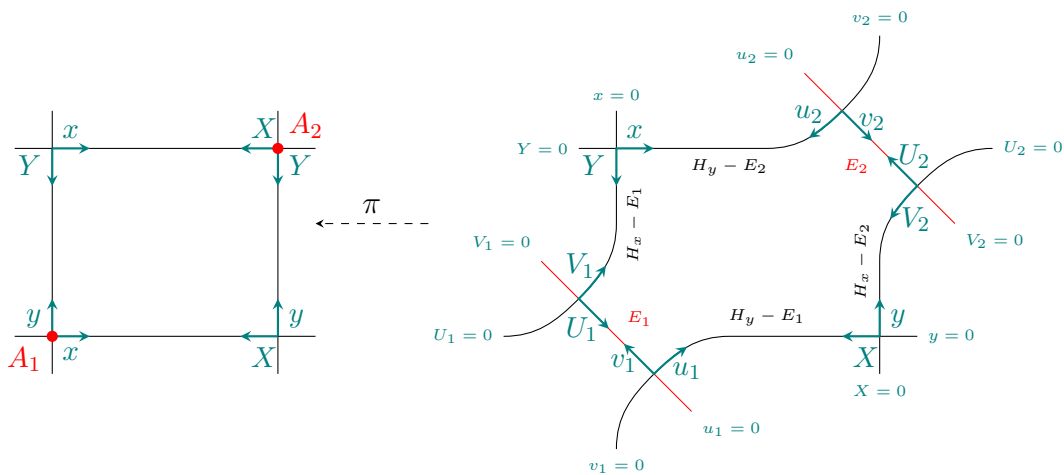
\begin{center}
	\begin{tikzpicture}[>=stealth]
		\begin{scope}
		\draw [black] (-0.5,0) -- (3.5,0) (0,-0.5)--(0,3) (3,-0.5)--(3,3) (-0.5,2.5)--(3.5,2.5);
		\draw [teal, thick, ->] (0,0)--(0.5,0); 
		\draw [teal, thick, ->] (0,0)--(0,0.5); 		
		\draw [teal, thick, ->] (3,0)--(2.5,0); 		
		\draw [teal, thick, ->] (3,0)--(3,0.5); 		
		\draw [teal, thick, ->] (3,2.5)--(3,2); 		
		\draw [teal, thick, ->] (3,2.5)--(2.5,2.5); 		
		\draw [teal, thick, ->] (0,2.5)--(0.5,2.5); 		
		\draw [teal, thick, ->] (0,2.5)--(0,2); 		
		\node at (0,0) [circle, draw=red!100, fill=red!100, thick, inner sep=0pt,minimum size=1.2mm] {};
		\node at (0,0) [red, below left] {\small $A_{1}$};
		\node at (3,2.5) [circle, draw=red!100, fill=red!100, thick, inner sep=0pt,minimum size=1.2mm] {};
		\node at (3,2.5) [red, above right] {\small $A_{2}$};
		\node at (0.5,0) [below left, teal] {\small $x$};
		\node at (0,0.5) [below left, teal] {\small $y$};
		\node at (2.5,0) [below right, teal] {\small $X$};
		\node at (3,0.5) [below right, teal] {\small $y$};
		\node at (0.5,2.5) [above left, teal] {\small $x$};
		\node at (0,2) [above left, teal] {\small $Y$};
		\node at (2.5,2.5) [above right, teal] {\small $X$};
		\node at (3,2) [above right, teal] {\small $Y$};
		\end{scope}	
		\draw [dashed , ->] (5,1.5) -- (3.5,1.5) node [above,align=center,midway] {$\pi$};
		\begin{scope}[xshift=7.5cm]
			\draw [red, -] (-1,1) -- (1,-1);
			\draw [red, -] (2.5,3.5) -- (4.5,1.5);
			\draw [black, - ] (-1.5,0) to [out=0,in=-135] (-0.5,0.5) to [out=45,in=-90]  	(0,1.5) -- (0,3);
			\draw [black, - ] (0,-1.5) to [out=90,in=-135] (0.5,-0.5) to [out=45,in=180]  	(1.5,0) -- (4,0);
			\draw [black, - ] (-0.5,2.5) -- (2,2.5) to [out=0,in=-135] (3,3) to [out=45,in=-90]  	(3.5,4);
			\draw [black, - ] (3.5,-0.5) -- (3.5,1) to [out=90,in=-135] (4,2) to [out=45,in=180]  (5,2.5);

			\draw [teal, thick, ->] (-0.01,2.5)--(0.5,2.5);
			\node at (0.5,2.5) [above left, teal] {\small $x$};
			\draw [teal, thick, ->] (0,2.5) -- (0,2);
			\node at (0,2) [above left, teal] {\small $Y$};
			\draw [teal, thick, ->] (3.5,0) -- (3.5,0.5);
			\node at (3.5,0.5) [below right, teal] {\small $y$};
			\draw [teal, thick, ->] (3.51,0) -- (3,0);
			\node at (3,0) [below right, teal] {\small $X$};
			\draw [teal, thick, ->] (0.51,-0.51)--(0.1,-0.1);
			\node at (0.1,-0.1) [below, teal] {\small $v_{1}$};
			\draw [teal, thick, ->] (0.5,-0.5) to [out=45,in=-145]  (0.94,-0.14);
			\node at (0.94,-0.2) [below, teal] {\small $u_{1}$};

			\draw [teal, thick, ->] (-0.51,0.51)--(-0.1,0.1);
			\node at (-0.1,0.1) [left, teal] {\small $U_{1}$};
			\draw [teal, thick, ->] (-0.5,0.5) to [out=45,in=-125]  (-0.14,0.94);
			\node at (-0.14,0.94) [left, teal] {\small $V_{1}$};

			\draw [teal, thick, ->] (2.99,3.01)--(3.4,2.6);
			\node at (3.4,2.6) [above, teal] {\small $v_{2}$};
			\draw [teal, thick, ->] (3,3) to [out=-135,in=35]  (2.56,2.64);
			\node at (2.56,2.7) [above, teal] {\small $u_{2}$};

			\draw [teal, thick, ->] (4.01,1.99)--(3.6,2.4);
			\node at (3.6,2.4) [right, teal] {\small $U_{2}$};
			\draw [teal, thick, ->] (4,2) to [out=-135,in=55] (3.64,1.56);
			\node at (3.64,1.56) [right, teal] {\small $V_{2}$};
			
			\node at (-1,1) [above left, teal] {\tiny $V_{1} = 0$};
			\node at (-1.5,0) [left, teal] {\tiny $U_{1} = 0$};

			\node at (0,-1.5) [below, teal] {\tiny $v_{1} = 0$};
			\node at (1,-1) [below right, teal] {\tiny $u_{1} = 0$};

			\node at (4,0) [right, teal] {\tiny $y = 0$};
			\node at (3.5,-0.5) [below, teal] {\tiny $X = 0$};

			\node at (5,2.5) [right, teal] {\tiny $U_{2} = 0$};
			\node at (4.5,1.5) [below right, teal] {\tiny $V_{2} = 0$};

			\node at (2.5,3.5) [above left, teal] {\tiny $u_{2} = 0$};
			\node at (3.5,4) [above, teal] {\tiny $v_{2} = 0$};

			\node at (0,3) [above, teal] {\tiny $x = 0$};
			\node at (-0.5,2.5) [left, teal] {\tiny $Y = 0$};

			\node at (1.5,2.5) [below] {\tiny $H_{y} - E_{2}$};
			\node at (0,1.5) [below,rotate=90] {\tiny $H_{x} - E_{1}$};
			\node at (0,0) [above right, red] {\tiny $E_{1}$};

			\node at (2,0) [above] {\tiny $H_{y} - E_{1}$};
			\node at (3.5,1) [above,rotate=90] {\tiny $H_{x} - E_{2}$};
			\node at (3.5,2.5) [below left, red] {\tiny $E_{2}$};
			
		\end{scope}	
	\end{tikzpicture}	
\end{center}	

From this picture we immediately see that, for example, the local defining expressions for the divisor 
$E_{1}$ are $E_{1} = \left\{ (\mathfrak{U_{1}}; V_{1}), (\mathfrak{U_{2}}; u_{1}) \right\}$. Then
$\tilde{\varphi}^{*}(\bar{E}_{1})$ is given by the collections 
$(\mathfrak{U}_{i}\cap \tilde{\varphi}^{-1}(\bar{\mathfrak{U}}_{1}); \tilde{\varphi}^{*}(\bar{V}_{1}))$
and
$(\mathfrak{U}_{i}\cap \tilde{\varphi}^{-1}(\bar{\mathfrak{U}}_{2}); \tilde{\varphi}^{*}(\bar{u}_{1}))$.
Writing down the expressions for $(\bar{U}_{1},\bar{V}_{1})$ in all six coordinate charts,
\begin{align*}
	\left(\bar{U}_{1}, \bar{V}_{1}\right) &= \left( \frac{V_{1} - 1}{U_{1} V_{1}^{2}}, U_{1}^{2} V_{1}  \right)
	= \left( \frac{u_{1} v_{1} - 1}{u_{1}^{2} v_{1}}, \frac{u_{1}}{v_{1}}   \right) 
	= \left( \frac{X(y-1)}{y}, \frac{1}{X^{2} y}   \right)\\
	&= \left( U_{2}V_{2} (1 - V_{2}), \frac{1}{U_{2}^{2} V_{2}}  \right)
	= \left(u_{2}(1 - u_{2} v_{2}), \frac{v_{2}}{u_{2}}  \right) 
	= \left( \frac{1 - Y}{x}, x^{2} Y \right),
\end{align*}
we note that, for example, 
$\mathfrak{U}_{1}\cap \tilde{\varphi}^{-1}(\bar{\mathfrak{U}}_{1}) = 
\mathfrak{U}_{1} \backslash \left(\{U_{1} = 0\}\cup \{V_{1} = 0\}\right)$, since the lines 
$U_{1} = 0$ and $V_{1} = 0$ map to the boundary $\bar{v}_{1} = 0$ of the $\bar{\mathfrak{U}}_{1}$ chart. Since
$\tilde{\varphi}^{*}(\bar{V}_{1}) = U_{1}^{2} V_{1}$, it does not vanish on 
$\mathfrak{U}_{1}\cap \tilde{\varphi}^{-1}(\bar{\mathfrak{U}}_{1})$. On the other hand, 
$\mathfrak{U}_{6}\cap \tilde{\varphi}^{-1}(\bar{\mathfrak{U}}_{1}) = 
\mathfrak{U}_{6} \backslash \left(\{x = 0\}\right)$, 
$\tilde{\varphi}^{*}(\bar{V}_{1}) = x^{2}Y$, which vanishes on $Y=0$ in 
$\mathfrak{U}_{6}\cap \tilde{\varphi}^{-1}(\bar{\mathfrak{U}}_{1})$, which is a local defining equation 
for $H_{y} - E_{2}$ in $\mathfrak{U}_{6}$. Proceeding along the same lines for the remaining charts, we get
\begin{align*}
	\tilde{\varphi}^{*}\left( (\mathfrak{U_{1}}; V_{1}) \right) &= 
	\left\{  \left( \mathfrak{U}_{5} \backslash  \{u_{2} = 0\}; v_{2} \right), 
	\left( \mathfrak{U}_{6} \backslash \{x = 0\}; Y \right)
	\right\},\\
	\tilde{\varphi}^{*}\left( (\mathfrak{U_{2}}; u_{1}) \right) &= 
	\left\{  \left( \mathfrak{U}_{1} \backslash  \left(\{V_{1} = 0\}\cup \{V_{1} = 1\}\right); U_{1} \right),\right.
	\\ & \qquad
	\left.
	\left( \mathfrak{U}_{5} \backslash \left(\{u_{2} = 0\}\cup \{u_{2} v_{2} = 1\}\right); v_{2} \right),
	\left( \mathfrak{U}_{6} \backslash  \{Y = 1\}; xY \right)
	\right\}.
\end{align*}
Thus, $\tilde{\varphi}^{*}(\bar{E}_{1}) = (H_{x} - E_{1}) + (H_{y} - E_{2})$. In exactly the same way we get the following 
mapping between the $-1$-curves,
\begin{alignat*}{2}
	\tilde{\varphi}^{*}(\bar{E}_{2}) &= (H_{x} - E_{2}) + (H_{y} - E_{1}), &\quad 
	\tilde{\varphi}^{*}(\bar{H}_{y} - \bar{E}_{1}) &= (H_{x} - E_{1}) + E_{1},\\
	\tilde{\varphi}^{*}(\bar{H}_{x} - \bar{E}_{2}) &= (H_{y} - E_{1}) + E_{1}, &\quad 
	\tilde{\varphi}^{*}(\bar{H}_{y} - \bar{E}_{2}) &= (H_{x} - E_{2}) + E_{2},\\
	\intertext{and also}
	\tilde{\varphi}^{*}(\bar{H}_{x} - \bar{E}_{1}) &= \pi^{*}(V(y-1))\in \mathcal{H}_{y}.	
\end{alignat*}
Passing to equivalence classes, we get
\begin{alignat*}{2}
	\tilde{\varphi}^{*}(\bar{\mathcal{H}}_{x}) &= \mathcal{H}_{x} + 2 \mathcal{H}_{y} - \mathcal{E}_{1} - \mathcal{E}_{2},
	&\quad 
	\tilde{\varphi}^{*}(\bar{\mathcal{E}}_{1}) &= \mathcal{H}_{x} + \mathcal{H}_{y} - \mathcal{E}_{1} - \mathcal{E}_{2},\\
	\tilde{\varphi}^{*}(\bar{\mathcal{H}}_{y}) &= 2\mathcal{H}_{x} + \mathcal{H}_{y} - \mathcal{E}_{1} - \mathcal{E}_{2},
	&\quad
	\tilde{\varphi}^{*}(\bar{\mathcal{E}}_{2}) &= \mathcal{H}_{x} + \mathcal{H}_{y} - \mathcal{E}_{1} - \mathcal{E}_{2}.
\end{alignat*}

\begin{remark}
	For a generic $k\in \mathbb{C}^{\times}$, 
$V(\bar{x} - k)\in \bar{\mathcal{H}}_{x}\in \operatorname{Pic}\left({\mathbb{P}^{1} \times \mathbb{P}^{1}}\right)$,
its lift to the surface is 
$\bar{\pi}^{*}(V(\bar{x} - k))\in \bar{\mathcal{H}}_{x}\in \operatorname{Pic}\left(\bar{\mathcal{X}}_{12}\right)$, and 
from the commutativity of the diagram we see that
\begin{align*}
\tilde{\varphi}^{*}(\bar{\mathcal{H}}_{x}) &= \tilde{\varphi}^{*}\circ \bar{\pi}^{*}([V(\bar{x} - k)]) = 
\left[\pi^{*}\circ \varphi^{*} (V(\bar{x} - k))\right] \\
&= \left[\pi^{*}\left(V\left(\frac{x(y-1)}{y^{2}}  - k\right)\right)\right] 
= \left[\pi^{*} \left(V\left(\frac{x(y-1) - k y^{2}}{y^{2}}\right)\right)\right]. 
\end{align*}
We need to be careful here. The curve $V(x(y-1) - k y^{2})$ is a $(1,2)$-curve, and so it belongs to the class 
$\mathcal{H}_{x} + 2 \mathcal{H}_{y}\in \operatorname{Pic}(\mathbb{P}^{1} \times \mathbb{P}^{1})$. Even though
it does pass through the points $A_{1}(0,0)$ and $A_{2}(\infty,\infty)$, its class in 
$\operatorname{Pic}(\mathcal{X}_{12})$ is still $\mathcal{H}_{x} + 2 \mathcal{H}_{y}$, and not 
$\mathcal{H}_{x} + 2 \mathcal{H}_{y} - \mathcal{E}_{1} - \mathcal{E}_{2}$. Indeed, looking at this equation 
in the chart $\mathfrak{U}_{2}$, we get
\begin{align*}
\pi^{*}\left(V(x(y-1) - k y^{2})\right) &= V(u_{1}(u_{1}v_{1} - 1 - k u_{1}v_{1}^{2}))  \\
&= V(u_{1}) + V(u_{1}v_{1} - 1 - k u_{1}v_{1}^{2}),
\end{align*}
where $V(u_{1})\in \mathcal{E}_{1}$ and 
$V(u_{1}v_{1} - 1 - k u_{1}v_{1}^{2})\in \mathcal{H}_{x} + 2 \mathcal{H}_{y} - \mathcal{E}_{1}$
(locally). However, \emph{since $A_{1}$ is an indeterminate point of the mapping}, 
the $u_{1}$-factor cancels out, and we get 
\begin{align*}
\pi^{*} \left(V\left(\frac{x(y-1) - k y^{2}}{y^{2}}\right)\right) &= 
V\left( \frac{u_{1}v_{1} - 1 - k u_{1}v_{1}^{2}}{u_{1}v_{1}^{21}}  \right) \\
&= 
V(u_{1}v_{1} - 1 - k u_{1}v_{1}^{2})\in \mathcal{H}_{x} + 2 \mathcal{H}_{y} - \mathcal{E}_{1} - \mathcal{E}_{2}.
\end{align*}


\end{remark}

\section{Deautonomization of a generic QRT mapping at a generic fiber}\label{app:deautoQRT}

In this section we show how to deautonomize a generic QRT mapping on a generic (i.e. elliptic) fiber,
thus obtaining simple expressions for the elliptic difference Painlev\'e equation.

Using the fact that $r_{y}^{*}$ is an involution preserving the intersection pairing on the Picard lattice,
it is possible to show (e.g., see Chapter 5.1 and in particular Lemma~5.1.2 of \cite{Duistermaat} 
that 
the pull-back action $\varphi^{*}$ of the half of a generic QRT mapping 
$\varphi=\sigma_{xy} \circ r_y:\mathcal{X}\to \mathcal{X}(=\bar{\mathcal{X}})$ 
on the Picard lattice is given by
\begin{align*}
\bar{\mathcal{H}}_x &\mapsto  4\mathcal{H}_x+\mathcal{H}_y-(\mathcal{E}_1+\mathcal{E}_2+\dots+\mathcal{E}_8),\\
\bar{\mathcal{H}}_y &\mapsto \mathcal{H}_x,\\
\bar{\mathcal{E}}_i &\mapsto \mathcal{H}_x-\mathcal{E}_i \quad (i=1,2,\dots,8).
\end{align*}	
%
It is well known that $(\varphi^2)^*$ acts as a translation of the $E_8^{(1)}$ lattice and its 
deautonomization gives the elliptic difference Painlev\'e equation \cite{KNY16, Murata04}.

Let us fix the embedding of the elliptic torus as in \eqref{eq:FG12}.
Then $\varphi$ acts on parameters as 
\begin{align*}
\bar{h}_x &=  4h_x+h_y-(e_1+e_2+\dots+e_8),\\
\bar{h}_y &= h_x,\\
\bar{e}_i &= h_x-e_i \quad (i=1,2,\dots,8),
\end{align*}
and if we change the embedding of $\mathcal{C}$, c.f.~\eqref{eq:changeC}, 
by adding a translation $\tau(u) = \bar{u} = u-(h_x-h_y+d)$,
$\varphi^2$ acts on $\operatorname{Pic}(\mathcal{X})$ as 
\begin{equation*}
\bar{h}_x =  h_x+d,\qquad 
\bar{h}_y = h_y-d,\qquad
\bar{e}_i = e_i \quad (i=1,2,\dots,8).	
\end{equation*}

Using Proposition~\ref{prop_blow}, we can obtain an explicit form of the mapping as follows. First,
observe that in the embedding \eqref{eq:FG12}, $\bar{x}_{1} = \infty$, $\bar{x}_{2} = 0$. Also,
\begin{equation*}
\varphi^{*}(V(\bar{x} - \bar{x}_{i})) = \varphi^{*}((\bar{H}_{x} - \bar{E}_{i}) + \bar{E}_{i})
= \left(3 H_{x} + H_{y} - \sum_{j=1,j\neq i}^{8} E_{j}\right) + (H_{x} - E_{i}).
\end{equation*}
Therefore, we get 
\begin{equation}\label{eq:generic_QRT}
	\left\{
	\begin{aligned}
		\bar{x} &= C_{1} \frac{x\, p_{2}(x,y)}{p_{1}(x,y)}\\ \bar{y} &= x 
	\end{aligned}
	\right.,
\end{equation}
where polynomials $p_{i}(x,y)$ can be explicitly obtained 
using Remark~\ref{rem:det-form},
\begin{align*}
p_1(x,y) &= -\begin{vmatrix} 
x&x^2&x^3&y&xy&x^2y&x^3y\\
x_3&x_3^2&x_3^3&y_3&x_3y_3&x_3^2y_3&x_3^3y_3\\
x_4&x_4^2&x_4^3&y_4&x_4y_4&x_4^2y_4&x_4^3y_4\\
x_5&x_5^2&x_5^3&y_5&x_5y_5&x_5^2y_5&x_5^3y_5\\
x_6&x_6^2&x_6^3&y_6&x_6y_6&x_6^2y_6&x_6^3y_6\\
x_7&x_7^2&x_7^3&y_7&x_7y_7&x_7^2y_7&x_7^3y_7\\
x_8&x_8^2&x_8^3&y_8&x_8y_8&x_8^2y_8&x_8^3y_8
\end{vmatrix},\\
p_2(x,y) &= \phantom{-}
\begin{vmatrix} 
1&x&x^2&x^3&y&xy&x^2y\\
1&x_3&x_3^2&x_3^3&y_3&x_3y_3&x_3^2y_3\\
1&x_4&x_4^2&x_4^3&y_4&x_4y_4&x_4^2y_4\\
1&x_5&x_5^2&x_5^3&y_5&x_5y_5&x_5^2y_5\\
1&x_6&x_6^2&x_6^3&y_6&x_6y_6&x_6^2y_6\\
1&x_7&x_7^2&x_7^3&y_7&x_7y_7&x_7^2y_7\\
1&x_8&x_8^2&x_8^3&y_8&x_8y_8&x_8^2y_8
\end{vmatrix},
\end{align*}
and where the normalization constant can be evaluated via 
\begin{equation*}
C_1=\frac{\bar{x}_{3} p_{1}(x_{3},y)}{x_{3}p_{2}(x_{3},y)} = \frac{\prod_{i=3}^8[e_2-e_i]}{\prod_{i=3}^8[e_1-e_i]}
\end{equation*}
after substituting $\bar{x}_{3} = \bar{F}_{12}(h_{x}-e_{3})$, $x_{i} = F_{12}(e_{i})$, and $y_{i} = G_{12}(e_{i})$.

Moreover, since the divisor class $\bar{\D}=\bar{\mathcal{H}}_x+\bar{\mathcal{H}}_y$ can be 
decomposed and pull-backed using two linearly equivalent divisors $\bar{D}_{1}\sim \bar{D}_{2}$,
\begin{align*}
\bar{D}_{1} & =(\bar{H}_x+\bar{H}_y-\bar{E}_6-\bar{E}_7-\bar{E}_8)+\bar{E}_6+\bar{E}_7+\bar{E}_8\\
&\mapsto(2H_x+H_y-\sum\nolimits_{i=1}^{4}E_i-E_5) +(H_x-E_6)+(H_x-E_7)+(H_x-E_8),\\
\bar{D}_{2} &=(\bar{H}_x+\bar{H}_y-\bar{E}_5-\bar{E}_7-\bar{E}_8)+\bar{E}_5+\bar{E}_7+\bar{E}_8\\
&\mapsto(2H_x+H_y-\sum\nolimits_{i=1}^{4}E_i-E_6)+(H_x-E_5)+(H_x-E_7)+(H_x-E_8),
\end{align*}
using Proposition~\ref{prop_d} we get the equation
\begin{align}
\frac{\bar{p}_{678}(\bar{x},\bar{y}) }{\bar{p}_{578}(\bar{x},\bar{y})} =C_3 \frac{(x-x_6)p_{345}(x,y)}{(x-x_5)p_{346}(x,y)} \label{generic_bxy} 
\end{align}
where 
\begin{align*}
\bar{p}_{ijk}(\bar{x},\bar{y})=&
\begin{vmatrix}
1&\bar{x}&\bar{y}&\bar{x} \bar{y}\\
1&\bar{x}_i&\bar{y}_i&\bar{x}_i \bar{y}_i\\
1&\bar{x}_j&\bar{y}_j&\bar{x}_j \bar{y}_j\\
1&\bar{x}_k&\bar{y}_k&\bar{x}_k \bar{y}_k
\end{vmatrix}\\
p_{ijk}(x,y)=&
\begin{vmatrix}
1&x&x^2&y&x y&x^2 y\\
0&0&0&0&0&1\\
1&0&0&0&0&0\\
1&	x_i	&x_i^2&	y_i&	x_i y_i&	x_i^2 y_i\\
1&	x_j	&x_j^2&	y_j&	x_j y_j&	x_j^2 y_j\\
1&	x_k	&x_k^2&	y_k&	x_k y_k&	x_k^2 y_k
\end{vmatrix}
=-
\begin{vmatrix}
x&x^2&y&x y\\
x_i	&	x_i^2	&	y_i	&	x_i y_i\\
x_j	&	x_j^2	&	y_j	&	x_j y_j\\
x_k	&	x_k^2	&	y_k	&	x_k y_k
\end{vmatrix}
\end{align*}
and $C_{3}$ is a normalization constant, 
\begin{align*}
C_3=& \frac{[h_x-e_1-e_5]^2[h_y-e_1-e_5][2h_x+h_y-\sum_{i=2,3,4,6,7,8}e_i]}{[h_x-e_1-e_6]^2[h_y-e_1-e_6][2h_x+h_y-\sum_{i=2,3,4,5,7,8}e_i]} \\
&\qquad \cdot \frac{[e_5-e_1]^3}{[e_6-e_1]^3}
 \prod_{i=2,3,4,7,8}\frac{[e_6-e_i]}{[e_1-e_i]}.
\end{align*}

Note that expressions~\eqref{eq:generic_QRT}  of the elliptic difference Painlev\'e equation are essentially the same as obtained by Murata
\cite{Murata04}, where degree $(4,1)$ polynomials written by $10\times10$ matrices are used without factorization. 
The combination of equation $\bar{y} = x$ and \eqref{generic_bxy}, the latter of which can be solved linearly for $\bar{x}$, 
gives a new expression of the elliptic difference Painlev\'e equation, where only degree $(1,1)$ or $(2,1)$ polynomials 
written by $4\times 4$ matrices are used.

\section{C.~F.~Schwartz's algorithm}\label{app:CFS}

In this section we explain, following \cite{Schwartz94}, how to transform a generic bi-degree $(2,2)$ curve $\Gamma$
given by the equation
$$a_{00}+a_{10}x+a_{20}x^2+a_{01}y+a_{11}xy+a_{21}x^2y+a_{02}y^2+a_{12}xy^2+a_{22}x^2y^2=0$$
to the Weierstrass's normal form $Y^2=4X^3-g_2Y-g_3$.

First we take two points $(a_1,0)$ and $(a_2,\infty)$ on $\Gamma$ and make a change of variables 
\begin{equation*}
\tilde{x}=\frac{x-a_1}{x-a_2},
\end{equation*}
then $\Gamma$ passes through the points $(0,0)$ and $(\infty,\infty)$ and its equation in coordinates $(\tilde{x},y)$, 
that we now write as $(x,y)$, takes a simpler form 
\begin{equation*}
ax^2+b xy+cy^2+dx+ey=ux^2y+vxy^2.
\end{equation*}

Next, put $X$ and $Y$ as\footnote{These equations are changed from \cite{Schwartz94}, since there was an error.}
\begin{align*}
X=& -\frac{ a c^2+c d v-c^2 u y+b c v y+e v^2 y}{c-v x},\\
Y=&\frac{1}{(c-v x)} \big(-c^2 d u+b c d v-a c e v+(-2 a c^2 u +a b c v -c d u v +a e v^2)x \\ &+(-b c^2 u+b^2 c v -2 a c^2 v -2 c d v^2 +b e v^2)y
+2(c^2 u^2 -b c u v -e u v^2 )x y\\&+2(c^2 u v-b c v^2-e v^3)y^2\big).
\end{align*}
Then $X$ and $Y$ satisfy the equation
\begin{equation*}
Y^2=4X^3+AX^2+BX+C
\end{equation*}
with
\begin{align*}
A=&b^2+8a c+4e u+4d v,\\
B=&4(a^2 c^2+a c e u+a c d v+d e u v)+2(a b^2 c+ b c d u+a b e v),\\
C=&(a b c+c d u+a e v)^2.
\end{align*}
Replacing $X$ with $X-A/12$ reduces it to the Weierstrass normal form
\begin{equation*}
Y^2=4X^3-g_1X-g_2,\quad\text{ where } g_2=-B+\frac{A^2}{12}, \quad g_3=-C+\frac{AB}{12}-\frac{A^3}{216},
\end{equation*}
and the discriminant is $\Delta= g_2^3-27g_3^2$.

\section{A birational representation of $W(E_8^{(1)})$}\label{app:bir-rep}

Let us again fix the embedding of the elliptic torus as in \eqref{eq:FG12}.

\begin{theorem}\label{tmp:bir-W8}
	A birational representation of the affine Weyl group ${W}(E_{8}^{(1)})$ is given by the following maps:
	\begin{align*}
	w_{0}:\quad& \bar{x}=\frac{1}{x} ,\ \bar{y}=\frac{1}{y},\ \bar{e}_{1}= e_{2},\ \bar{e}_{2} = e_{1};\\[3pt] 
	w_{1}:\quad& \bar{x}=y ,\ \bar{y}=x,\ \bar{h}_{x}= h_{y},\ \bar{h}_{y} = h_{x};\\[3pt] 
	w_{2}:\quad& \bar{x}=\frac{x}{y},\ \bar{y}=\frac{1}{y} ,\ \bar{h}_{x} = h_{x} + h_{y} - e_{1} - e_{2},\ 
	\bar{e}_{1} = h_{y} - e_{2},\ \bar{e}_{2} = h_{y} - e_{1};\\[3pt] 
	w_{3}:\quad&
	 \bar{x}=F_{23}(e_1) (x - F_{12}(e_{3})),\ 
	 \bar{y}=G_{23}(e_1) (y - G_{12}(e_{3})),\ \bar{e}_{2} = e_{3},\ \bar{e}_{3} = e_{2};\\[3pt]
	w_{i}:\quad& \bar{x}=x,\ \bar{y}=y,\ 
	 \bar{e}_{i-1} = e_{i},\ \bar{e}_{i} = e_{i-1},\ i=4,\dots,8.
	\end{align*}
\end{theorem}

This realization is essentially the same as in \cite{KNY16}, but here
we obtain it using the factorization formula.

\begin{proof} Similarly to the proof of Corollary~\ref{cor:bira-D6}, these expressions are 
 obtained by a direct application of Proposition~\ref{prop_blow}.
Here we show it for $w_{3}$ that is induced by a reflection in the 
root $\alpha_{3} = \mathcal{E}_{2} - \mathcal{E}_{3}$. 	

Since $(x_{1},y_{1})=(\infty,\infty)$ and $(x_{2},y_{2})=(0,0)$
\begin{align*}
w_{3}^{*}(V(\bar{x} - 0)) &= \pi_{*} \tilde{w}_{3}^{*}((\bar{H}_{x} - \bar{E}_{2}) + \bar{E}_{2}) 
= \pi_{*}((H_{x} - E_{3}) + E_{3}) = H_{x} - E_{3}\\
w_{3}^{*}(V(\bar{x} - \infty)) &= \pi_{*} \tilde{w}_{3}^{*}((\bar{H}_{x} - \bar{E}_{1}) + \bar{E}_{1}) 
= \pi_{*}((H_{x} - E_{1}) + E_{1}) = H_{x} - E_{1},
\end{align*}
and so $\bar{x} = C_{1} (x - x_{3}) = C_{1}(x - F_{12}(e_{3}))$. 
Similarly,  $\bar{y} = C_{2}(y - G_{12}(e_{3}))$. To evaluate the normalization 
constants, note that $\tilde{w}_{3}^{*}(\bar{E}_{3}) = E_{2}$, and so 
\begin{equation*}
(\bar{F}_{12}(\bar{e}_3),\bar{G}_{12}(\bar{e}_3)) = (F_{13}(e_2),G_{13}(e_2)) = - (C_{1}\, F_{12}(e_{3}), C_{2}\, G_{12}(e_{3})).
\end{equation*}
Thus,
\begin{equation*}
C_{1} = - \frac{F_{13}(e_{2})}{F_{12}(e_{3})} = F_{23}(e_{1}),\qquad 
C_{2} = -\frac{G_{13}(e_2)}{G_{12}(e_3)} = G_{23}(e_1),
\end{equation*}
and we obtain the birational realization of $w_3$.
\end{proof}

\end{document}